\documentclass[12pt, a4paper]{amsart}
\usepackage{amsmath}

\usepackage{amssymb}
\usepackage{stmaryrd}
\usepackage{amsthm}

\usepackage{amscd}
\usepackage[all]{xy}
\usepackage{url}

\usepackage{graphics}

\newtheorem{Thm}{Theorem}[subsection]

\usepackage[dvipsnames, usenames]{color}
\usepackage[colorlinks]{hyperref}
\hypersetup{%
pdftitle={Preprint},
pdfauthor={Fan QIN},
pdfkeywords={Quantum Cluster Algebra, etc.},
bookmarksnumbered,
pdfstartview={FitH},
breaklinks=true,
urlcolor=blue,
citecolor=blue,
}%
\usepackage[all]{hypcap} 
\usepackage{breakurl}

\usepackage{color}
\newenvironment{NB}{\color{red}{\bf NB:} \footnotesize}{}

\newtheorem{Lem}[Thm]{Lemma}
\newtheorem{Prop}[Thm]{Proposition}
\newtheorem{Cor}[Thm]{Corollary}
\newtheorem{Conj}[Thm]{Conjecture}

\newtheorem{Eg}[Thm]{Example}
\newtheorem{Rem}[Thm]{Remark}
\newtheorem{Def}[Thm]{Definition}

\newtheorem*{Def*}{Definition}
\newtheorem*{Thm*}{Theorem}

\newcommand{\Z}{\mathbb{Z}}
\newcommand{\N}{\mathbb{N}}
\newcommand{\Q}{\mathbb{Q}}
\newcommand{\C}{\mathbb{C}}
\newcommand{\R}{\mathbb{R}}

\newcommand{\ie}{{\em i.e.}\ }
\newcommand{\cf}{{\em cf.}\ }

\newcommand{\st}{{such that}\ }

\renewcommand{\hat}[1]{\widehat{#1}}
\renewcommand{\tilde}[1]{\widetilde{#1}}

\newcommand{\opname}[1]{\operatorname{\mathsf{#1}}}

\renewcommand{\mod}{\opname{mod}}

\newcommand{\Rep}{\opname{Rep}}
\newcommand{\per}{\opname{per}}

\newcommand{\Dfd}{\cD_{fd}}

\newcommand{\pr}{\opname{pr}}

\newcommand{\ind}{\opname{ind}}

\newcommand{\add}{\opname{add}}

\newcommand{\op}{^{op}}

\newcommand{\ra}{\rightarrow}
\newcommand{\xra}{\xrightarrow}

\newcommand{\iso}{\stackrel{_\sim}{\rightarrow}}

\newcommand{\Gr}{\opname{Gr}}

\newcommand{\dimv}{\underline{\dim}\,}

\newcommand{\rank}{\opname{rank}}

\newcommand{\Ker}{\opname{Ker}}
\renewcommand{\Im}{\opname{Im}}

\newcommand{\Tr}{\opname{Tr}}

\newcommand{\id}{\mathbf{1}}

\newcommand{\Aut}{\opname{Aut}}

\newcommand{\Hom}{\opname{Hom}}

\newcommand{\Ext}{\opname{Ext}}
\newcommand{\ext}{\opname{ext}}

\newcommand{\supp}{\opname{supp}}

\renewcommand{\deg}{\opname{deg}}

\newcommand{\Kz}{{\opname{K}_0}}
\newcommand{\Hf}{{\frac{1}{2}}}

\newcommand{\Rm}[1]{{\longmapsto}}
\newcommand{\Lm}[1]{{\longmapsfrom}}

\newcommand{\cA}{{\mathcal A}}

\newcommand{\cC}{{\mathcal C}}
\newcommand{\cD}{{\mathcal D}}
\newcommand{\cE}{{\mathcal E}}
\newcommand{\cF}{{\mathcal F}}

\newcommand{\cK}{{\mathcal K}}
\newcommand{\cL}{{\mathcal L}}
\newcommand{\cM}{{\mathcal M}}
\newcommand{\cN}{{\mathcal N}}

\newcommand{\cP}{{\mathcal P}}
\newcommand{\cQ}{{\mathcal Q}}
\newcommand{\cR}{{\mathcal R}}

\newcommand{\cT}{{\mathcal T}}
\newcommand{\cU}{{\mathcal U}}

\newcommand{\cY}{{\mathcal Y}}

\newcommand{\sT}{{\mathbb T}}

\newcommand{\tB}{{\widetilde{B}}}

\newcommand{\tL}{{\widetilde{L}}}

\newcommand{\tQ}{{\widetilde{Q}}}

\newcommand{\tW}{{\widetilde{W}}}

\newcommand{\td}{{\widetilde{d}}}

\newcommand{\tg}{{\tilde{g}}}

\newcommand{\tcF}{{\tilde{\cF}}}
\newcommand{\grRep}{{\Rep^\bullet}}
\newcommand{\grM}{{\cM^\bullet}}
\newcommand{\grMaff}{{\cM^\bullet_0}}

\newcommand{\tq}{{\widetilde{q}}}
\newcommand{\inc}{\opname{i}}

\newcommand{\pbw}{M}
\newcommand{\can}{L}
\newcommand{\gen}{\mathbb{L}}

\newcommand{\tChar}{\chi_{q,t}}
\newcommand{\HChar}{\chi^H_{q,t}}

\newcommand{\pbwTorus}{{\pbw^\torus}}
\newcommand{\canTorus}{{\can^\torus}}
\newcommand{\genTorus}{{\gen^\torus}}
\newcommand{\pbwCl}{{\pbw^\clAlg}}
\newcommand{\canCl}{{\can^\clAlg}}
\newcommand{\genCl}{{\gen^\clAlg}}

\newcommand{\pbwClBasis}{{\{\pbwCl(w)\}}}
\newcommand{\canClBasis}{{\{\canCl(w)\}}}%
\newcommand{\genClBasis}{{\{\genCl(w)\}}}%

\newcommand{\redWSet}{{\mathcal{J}}}

\newcommand{\redCanBasis}{{\{\can^\torus(w),w\in\redWSet\}}}%
\newcommand{\genRedTarg}{{\gen^\redTargSpace}}

\newcommand{\KGp}{{\cK}}
\newcommand{\dualKGp}{{\cK^*}}
\newcommand{\quotKGp}{{\cR_t}}

\newcommand{\tBase}{{\Z[t^\pm]}}
\newcommand{\vBase}{{\Z[v^\pm]}}
\newcommand{\qBase}{{\Z[q^{\pm\Hf}]}}
\newcommand{\qBaseCoeff}{{\Z P[q^{\pm\Hf}]}}
\newcommand{\ZCoeff}{{\Z P}}

\newcommand{\redTargSpace}{{\cY}}
\newcommand{\HParam}{{t^\Hf}} 
\newcommand{\HRing}{{{\cY}^H_\HParam}} 
\newcommand{\HGp}{{\cK_\HParam}} 

\newcommand{\torus}{\cT}

\newcommand{\flagVar}{{\tilde{\cF}}}
\newcommand{\dualFlagVar}{{\tilde{\cF}^\bot}}
\newcommand{\projQuot}{{\cM}}
\newcommand{\affQuot}{{\projQuot_0}}
\newcommand{\lag}{{\cL}}
\newcommand{\grProjQuot}{{\projQuot^\bullet}}
\newcommand{\grAffQuot}{{\affQuot^\bullet}}
\newcommand{\grLag}{{\lag^\bullet}}
\newcommand{\grFib}{{\mathfrak{m}^\bullet}}

\newcommand{\grRegStratum}{{\grAffQuot^\mathrm{reg}}}
\newcommand{\cor}{{\opname{cor}}}
\newcommand{\contr}{{\hat{\Pi}}}
\newcommand{\clAlg}{{\cA}}
\newcommand{\qClAlg}{{\cA^q}}
\newcommand{\subQClAlg}{{\cA^q_{sub}}}
\newcommand{\qGLSClAlg}{{\cA^q_{\mathrm{GLS}}}}
\newcommand{\coeffFree}{{^\phi}}
\newcommand{\pureCoeff}{{^f}}
\newcommand{\kerMod}{{^\sigma}}

\newcommand{\vtx}{{\opname{I}}}

\newcommand{\wtLess}{{<_w}}

\newcommand{\tOmega}{{\tilde{\Omega}}}

\newcommand{\oOmega}{{\overline{\Omega}}}
\newcommand{\oh}{{\overline{h}}}

\newcommand{\lSp}{{\opname{L}}}
\newcommand{\eSp}{{\opname{E}}}
\newcommand{\grLSp}{{\lSp^\wtLess}}
\newcommand{\grESp}{{\eSp^\wtLess}}
\newcommand{\grEndSp}{{\lSp^\bullet}}

\newcommand{\te}{{\tilde{e}}}

\newcommand{\diag}{{\delta}}
\newcommand{\canStr}{{b}}

\newcommand{\condC}{{\rm{(C)}$\ $}} 

\newcommand{\invCq}{{C_q^{-1}}}

\newcommand{\Pop}{{\bar{P}}}
\newcommand{\Iop}{{\bar{I}}}

\newcommand{\rev}{{\xi}}

\newcommand{\rscPbwCl}{{\tilde{\pbwCl}}}
\newcommand{\rscCanCl}{{\tilde{\canCl}}}

\newcommand{\rscCanBasis}{{\{\rscCanCl(w)\}}}

\newcommand{\rscQClAlg}{\tilde{\cA^q}}

\newcommand{\intA}{{\mathbb{A}}}

\newcommand{\tRes}{{\tilde{\mathrm{Res}}}}
\newcommand{\res}{{\mathrm{Res}}}

\newcommand{\eMatrix}{{{\cE'}}}
\newcommand{\dT}{{\td'}}
\newcommand{\dTW}{{\dT_W}}

\newcommand{\wLess}{{<_{\overrightarrow{w}}}}

\newcommand{\trunc}{{^{\leq 0}}}

\newcommand{\tr}{{\tilde{r}}}

\newcommand{\bracket}[1]{\left\langle#1\right\rangle}
\def\tr{\mathop{\mathrm{tr}}\nolimits}
\newcommand{\ow}{{\overrightarrow{w}}}
\newcommand{\BPos}{{B_{+}^{\mathrm{up}}}}

\newcommand{\mfr}[1]{{\mathfrak{#1}}}
\newcommand{\mbf}[1]{{\mathbf{#1}}}
\newcommand{\mbb}[1]{{\mathbb{#1}}}
\newcommand{\mca}[1]{{\mathcal{#1}}}

\newcommand{\mscr}[1]{{\mathscr{#1}}}	
\usepackage{mathrsfs}
\newenvironment{aenumerate}{%
  \begin{enumerate}%
  }{\end{enumerate}}
\newenvironment{renumerate}{%
  \begin{enumerate}%
  }{\end{enumerate}}

\newcommand{\braket}[1]{\left\langle#1\right\rangle}
\newcommand{\set}[1]{\left\{#1\right\}}

\newcommand{\Qv}{\mathbb{Q}(v)}
\newcommand{\Uv}{\mbf{U}_{v}}

\newcommand{\lemref}[1]{Lemma~\ref{#1}}

\newcommand{\fit}[1]{{\widetilde{f}}_{#1}}
\newcommand{\eit}[1]{{\widetilde{e}}_{#1}}
\newcommand{\wt}{\operatorname{wt}}
\newcommand{\Binfty}{\mscr{B}(\infty)}
\newcommand{\vep}{\varepsilon}

\newcommand{\Gup}{G^{\operatorname{up}}}
\newcommand{\Fup}{F^{\operatorname{up}}}
\newcommand{\Glow}{G^{\operatorname{low}}}

\newcommand{\+}{\oplus}
\newcommand{\up}{\operatorname{up}}
\newcommand{\low}{\operatorname{low}}

\newcommand{\rpm}{r_{\pm}}
\newcommand{\rmp}{r_{\mp}}

\usepackage{comment}

\usepackage{tikz}
 \usetikzlibrary{positioning,shapes,shadows,arrows,snakes}

\pgfdeclarelayer{edgelayer}
\pgfdeclarelayer{nodelayer}
\pgfsetlayers{edgelayer,nodelayer,main}

\tikzstyle{none}=[inner sep=0pt]
\tikzstyle{black box}=[draw=black, fill=black!25]
\tikzstyle{white box}=[draw=black, fill=white]
\tikzstyle{black circle}=[circle,draw=black!50, fill=black!25]
\tikzstyle{red circle}=[circle,draw=red!50, fill=red!25]
\tikzstyle{blue circle}=[circle,draw=blue!50, fill=blue!25]
\tikzstyle{green circle}=[circle,draw=green!50, fill=green!25]
\tikzstyle{yellow circle}=[circle,draw=yellow!50, fill=yellow!25]

\usepackage{versions}
\excludeversion{comment}
\excludeversion{NB}

\usepackage[top=3cm, bottom=3cm, left=2.5cm, right=2.5cm, twoside]{geometry}
\begin{document}
\title[Graded quiver varieties and quantum cluster algebras]{Graded quiver varieties, quantum cluster algebras and dual
  Canonical basis}
 
\author{Yoshiyuki Kimura}
\address{Yoshiyuki Kimura\\
Osaka City University Advanced Mathematical Institute,%
Osaka City University,3-3-138, Sugimoto, Sumiyoshi-ku, Osaka, 558-8585 Japan%
}
\email{ykimura@sci.osaka-cu.ac.jp}
\author{Fan Qin}
\address{Fan Qin\\
Universit\'{e} Paris Diderot - Paris 7, Institut de Math\'{e}\-ma\-ti\-ques de Jussieu, UMR 7586 du CNRS, 175 rue du chevaleret, 75013, Paris, France}
\email{qinfan@math.jussieu.fr}

\maketitle
\begin{abstract}
Inspired by a previous work of Nakajima, we consider perverse sheaves
over acyclic graded quiver
varieties and study the Fourier-Sato-Deligne transform from a
representation theoretic point of view. We obtain deformed monoidal
categorifications of acyclic quantum cluster algebras with specific
coefficients. In particular, the (quantum) positivity conjecture is verified whenever
there is an acyclic seed in the (quantum) cluster algebra.

In the second part of the paper, we introduce new quantizations and show that all quantum cluster
monomials in our setting belong
to the dual canonical basis of the corresponding quantum unipotent
subgroup. This result generalizes previous work by Lampe and by
Hernandez-Leclerc from the Kronecker and Dynkin quiver case to the acyclic case.

The Fourier transform part of this paper provides crucial input for
the second author's paper where he constructs bases of acyclic quantum cluster algebras with
arbitrary coefficients and quantization.
\end{abstract}

\tableofcontents



\section{Introduction}\label{sec:intro}

\subsection{Motivation}



Cluster algebras were invented by Fomin and Zelevinsky in
\cite{FominZelevinsky02}. They are algebras generated by certain
combinatorially defined generators (the cluster variables). The quantum deformations were defined
in \cite{BerensteinZelevinsky05}. Fomin and Zelevinsky stated their original motivation as follows:

\textit{
This structure should serve as an algebraic framework for the study of dual
canonical bases in these coordinate rings and their q-deformations. In particular, we conjecture that all monomials in the variables of any given cluster (the cluster monomials) belong to this dual canonical basis.}

However, despite the many successful applications of (quantum) cluster algebras to other
areas (\cf the introductory survey by
Bernhard Keller \cite{Keller12}), the link between (quantum) cluster monomials and the dual
canonical basis of quantum groups remains largely open. Partial
results are due to \cite{Lamp10} \cite{Lamp11} \cite{HernandezLeclerc11} for quivers of finite and affine type.

Also, the following conjecture has attracted a lot of interest since the invention of cluster algebras.
\begin{Conj}[Positivity conjecture]\label{conj:positivity}
With respect to the cluster variables in any given seed, each cluster
variable expands into a Laurent polynomial with non-negative integer coefficients.
\end{Conj}
This conjecture has been proved for cluster algebras arising from
surfaces by Gregg Musiker, Ralf Schiffler, and Lauren Williams \cite{MusikerSchifflerWilliams11}, for cluster algebras containing a
bipartite seed by Nakajima \cite{Nakajima09}, and the quantized
version for quantum cluster algebras with respect to
an acyclic initial seed by \cite{Qin10}. Recently, Efimov obtained
further partial results on this conjecture for quantum cluster algebras containing an acyclic seed using
mixed Hodge modules, \cf \cite{Efimov11}.  After this article was
posted on Arxiv, Kyungyong Lee and Ralf Schiffler informed the authors about a combinatorial proof of this conjecture for skew-symmetric coefficient-free cluster algebras of rank 3, \cf \cite{LeeSchiffler12}.

\subsection{Strategy and main results}

In \cite{HernandezLeclerc09}, Hernandez and Leclerc propose monoidal categorification as a new
approach to Conjecture \ref{conj:positivity}: for a given cluster
algebra $\clAlg$, find a monoidal category
$\cC$ such that its Grothendieck ring $R$ is isomorphic to $\clAlg$ and the
preimages of the cluster monomials are equivalence classes of simple
objects. Nakajima observed in \cite{Nakajima09} that the Grothendieck ring could be
constructed geometrically, following his series of works
\cite{Nakajima01} \cite{Nakajima04} where he studied quantum
affine algebras via (graded) quiver varieties. As a consequence, he
gave a geometric construction of the cluster algebra associated
with a bipartite quiver in the spirit of monoidal categorification.

Inspired by the previous work of Nakajima \cite{Nakajima09}, we use geometry of certain graded quiver
varieties to construct a deformed Grothendieck ring, and show that it
is isomorphic to the acyclic quantum cluster algebra. This proof consists
of the following steps:

\begin{enumerate}
\item use a new family of graded quiver varieties to construct
  a Grothendieck ring with a new $t$-deformation, which is treated in
  detail in \cite{Qin11} (\cf \cite{QinThesis});
\item use the Fourier-Sato-Deligne transform to identify the $t$-analogue
  of $q$-characters ($qt$-characters for short) of
  certain ``simple modules'' inside the Grothendieck ring with the
  quantum cluster variables whose cluster expansions were obtained in \cite{Qin10};
\item prove that the above identification is an algebra isomorphism.
\end{enumerate}

The second step is crucial. We can no longer use Nakajima's previous
construction because our quiver is not bipartite. Instead, we interpret
the graded quiver varieties using quiver representation theory. This
allows us to use the Nakayama functor to construct the pair of dual
spaces to which we apply the Fourier-Sato-Deligne transform. This
conceptual interpretation allows us to simplify and generalize Nakajima's
previous work. 

As a corollary, Conjecture \ref{conj:positivity} is true for any quantum
cluster algebra containing an acyclic seed.

Next, we change the quantizations of the $t$-deformed Grothendieck
rings, the $qt$-characters, the ring of $qt$-characters, and the
acyclic quantum cluster
algebra. Notice that the standard modules and the simple modules induce a dual PBW
basis and a dual canonical basis of the quantum cluster algebras, \cf
\cite{Qin11} for a more general treatment. Following the idea of \cite{GeissLeclercSchroeer11}, we use
$T$-systems to show that the quantum cluster algebra $\rscQClAlg$ is
isomorphic to a certain quantum coordinate ring, by comparing the dual PBW
bases of both algebras. As a consequence, up to specific $q$-powers,
we could identify the dual canonical bases of both algebras. Via this identification, up to
specific $q$-powers, the quantum cluster monomials are contained
in the dual canonical basis of the quantum coordinate ring.

\subsection{Plan of the paper}

In Section \ref{sec:reminders}, we recall the definitions and some
properties of the ice
quiver with $z$-pattern, of the graded
quiver varieties, of the geometric constructions of deformed
Grothendieck rings, and of $t$-analogues of $q$-characters.

In Section
  \ref{sec:psedoModules}, we give a representation theoretic interpretation
  of graded quiver varieties and study the Fourier-Sato-Deligne
  transforms. We obtain the deformed monoidal categorification of
  an acyclic quantum cluster algebra and the positivity conjecture in
  this case (Theorem \ref{thm:iso} and Corollary \ref{cor:stronglyPositive}).

In section \ref{sec:unipotentSubgroup}, we recall the unipotent
quantum subgroup following \cite{Kimura10}. In Section~\ref{sec:GLS},
we recall the $T$-systems of quantum minors inside quantum coordinate rings.

In Section \ref{sec:twistedQTCharacters}, we introduce new quantizations
of the deformed Grothendieck ring, the ring of $qt$-characters and the quantum cluster algebras. Then in Section
\ref{sec:dualCanonicalBasis}, we show that, up to specific $q$-powers,
the quantum
cluster monomials of these algebras can be identified with certain elements in the dual
canonical basis of the corresponding quantum unipotent subgroups (Theorem \ref{thm:2nd}).

\section*{Acknowledgments}\label{sec:ack}
The authors would like to thank their thesis supervisors Hiraku
Nakajima and Bernhard Keller for encouragement and inspiring
discussions. They thank David Hernandez and Bernard Leclerc for explaining
their recent work \cite{HernandezLeclerc11}. The work of Y.K. was supported by Kyoto University Global Center Of
Excellence Program ``Fostering top leaders in mathematics''. The work of F.Q. was
supported by the CSC scholarship and the research network ANR GTAA.

\section{Preliminaries}\label{sec:reminders}

In this Section, we give the definitions and some basic properties
of ice quivers, quantum cluster algebras, graded quiver varieties,
deformed Grothendieck rings, and $t$-analogues of
$q$-characters. More details can be found in
\cite{BerensteinZelevinsky05} \cite{Nakajima01} \cite{Nakajima04} \cite{Nakajima09}, or in \cite{Qin10} \cite{Qin11}.

\subsection{Ice quivers with $z$-pattern}
\label{sec:frozenQuiver}
A quiver $Q$ is an oriented graph, which consists of a set of vertices
$I=\{ 1,\ldots,n\}$ and a set of arrows $\Omega$. For each arrow $h$,
denote its source by $s(h)$ and its target by $t(h)$. Associate to
$h$ a new arrow $\oh$ which points from $t(h)$ to $s(h)$. Denote the
set $\{\oh|h\in\Omega\}$ by $\oOmega$. Define $H$ to be the disjoint
union of $\Omega$ and $\oOmega$. The opposite quiver $Q\op$ of $Q$ consists of the
vertices in $I$ the arrows in $\oOmega$. Sometimes we also denote $I$ and $\Omega$ by
$Q_0$ and $Q_1$ respectively.

The quiver $Q$ is called acyclic if it contains
no oriented cycles. It is called bipartite if at any vertex $i\in I$,
either there are no incoming arrows or there are no outgoing arrows. 

\begin{Eg}
  The acyclic quiver $Q$ in Figure \ref{fig:acyclicQuiver} has vertices $1$, $2$,
  $3$. Its opposite quiver $Q\op$ is given by Figure \ref{fig:oppositeQuiver}.
\end{Eg}

\begin{figure}[htb!]
 \centering
\beginpgfgraphicnamed{fig:acyclicQuiver}
  \begin{tikzpicture}
    \node [shape=circle, draw] (v1) at (1,-3) {1}; \node
    [shape=circle, draw] (v2) at (3,-2) {2}; \node [shape=circle,
    draw] (v3) at (3,0) {3};

    \draw[-triangle 60] (v1) edge (v2); \draw[-triangle 60] (v2) edge
    (v3); \draw[-triangle 60] (v1) edge (v3);
  \end{tikzpicture}
\endpgfgraphicnamed
\caption{An acyclic quiver $Q$}
\label{fig:acyclicQuiver}
\end{figure}
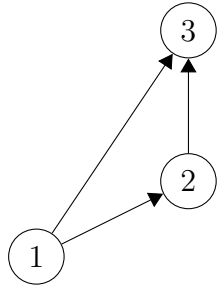

\begin{figure}[htb!]
 \centering
\beginpgfgraphicnamed{fig:oppositeQuiver}
  \begin{tikzpicture}
    \node [shape=circle, draw] (v1) at (1,-3) {1}; \node
    [shape=circle, draw] (v2) at (3,-2) {2}; \node [shape=circle,
    draw] (v3) at (3,0) {3};

    \draw[-triangle 60] (v2) edge (v1); \draw[-triangle 60] (v3) edge
    (v2); \draw[-triangle 60] (v3) edge (v1);
  \end{tikzpicture}
\endpgfgraphicnamed
\caption{The quiver $Q\op$}
\label{fig:oppositeQuiver}
\end{figure}
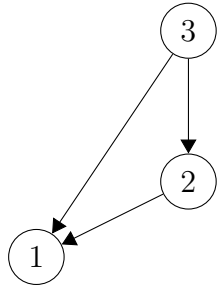

Let $Q$ be a full subquiver of another quiver $\tQ$, which has
vertices $\{ 1,\ldots,m\}$ and set of arrows $\tOmega$. We say that $\tQ$
is an ice quiver with principal part $Q$ and \emph{coefficient type} (or \emph{frozen pattern}) $\tOmega-\Omega$. The vertices $n+1,\ldots,m$ are called \emph{frozen vertices}.

\begin{Eg}
  Figure \ref{fig:levelOneQuiver} is an example of an ice quiver with
  $m=6$, whose principal part is the quiver in Figure \ref{fig:acyclicQuiver}.
\end{Eg}

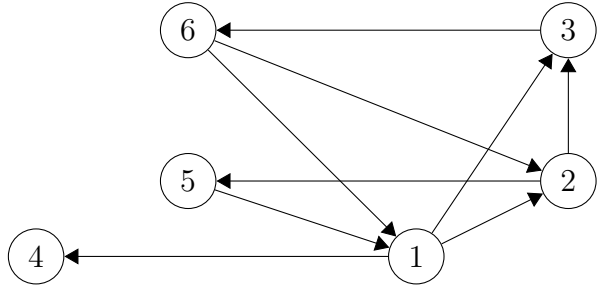
\begin{figure}[htb!]
 \centering
\beginpgfgraphicnamed{fig:levelOneQuiver}
  \begin{tikzpicture}
    \node [shape=circle, draw] (v1) at (1,-3) {1}; 
    \node  [shape=circle, draw] (v2) at (3,-2) {2}; 
    \node [shape=circle,  draw] (v3) at (3,0) {3};

\node [shape=circle, draw] (v4) at (-4,-3) {4}; 
    \node  [shape=circle, draw] (v5) at (-2,-2) {5}; 
    \node [shape=circle,  draw] (v6) at (-2,0) {6};

    \draw[-triangle 60] (v1) edge (v2); 
    \draw[-triangle 60] (v2) edge (v3); 
    \draw[-triangle 60] (v1) edge (v3);
    
    \draw[-triangle 60] (v5) edge (v1); 
    \draw[-triangle 60] (v6) edge (v1); 
    \draw[-triangle 60] (v6) edge (v2); 
    
    \draw[-triangle 60] (v1) edge (v4); 
    \draw[-triangle 60] (v2) edge (v5); 
    \draw[-triangle 60] (v3) edge (v6); 
       
  \end{tikzpicture}
\endpgfgraphicnamed
\caption{An ice quiver $\tQ^z_1$ of level $1$ with $z$-pattern}
\label{fig:levelOneQuiver}
\end{figure}

We associate to $\tQ$ an $m\times n$ matrix\footnote{Notice that this convention is opposite to that of \cite{Nakajima09}.} 
$\tB=(b_{ij})$ such that its entry in the position $(i,j)$ is
\begin{align*}
  b_{ij}=\sharp\{\text{arrows from $i$ to $j$}\}-\sharp\{\text{arrows from $j$
to $i$}\}.
\end{align*}
If further a compatible pair $(\Lambda,\tB)$ is given, we can construct
the associated quantum cluster algebra $\qClAlg$ following Section \ref{sec:quantumClusterAlgebras}. 

\begin{Eg}
The matrix $B=B_Q$ associated with the quiver $Q$ in Figure
\ref{fig:acyclicQuiver} is
  \begin{align*}
    \begin{pmatrix}
      0&1&1\\
      -1&0&1\\
      -1&-1&0
    \end{pmatrix}.
  \end{align*}

  The matrix $\tB$ associated to the ice quiver $\tQ$ in Figure
  \ref{fig:levelOneQuiver} is
  \begin{align*}
    \begin{pmatrix}
      0&1&1\\
      -1&0&1\\
      -1&-1&0\\
      -1&0&0\\
      1&-1&0\\
      1&1&-1
    \end{pmatrix}.
  \end{align*}
The matrix $B_{\tQ}$ is invertible. Thus we have a canonical choice of
$\Lambda$ given by
  \begin{align*}
\Lambda=-B_{\tQ}^{-1}=    \begin{pmatrix}
      0&0&0&1&0&0\\
      0&0&0&1&1&0\\
      0&0&0&2&1&1\\
      -1&-1&-2&0&-1&-2\\
      0&-1&-1&1&0&-1\\
      0&0&-1&2&1&0
    \end{pmatrix}.
  \end{align*}
\end{Eg}

Let $l$ be a non-negative integer and $Q$ an acyclic quiver. Denote by $\mod \C Q$ the category of
finite dimensional right $\C Q$-modules, or equivalently
representations of the opposite quiver $Q\op$. The indecomposable projectives are
denoted by $P_i$. The bounded derived category $D^b(\mod\C Q)$ has an
Auslander-Reiten quiver, from which we extract the full subquiver
supported on the vertices $\tau^d P_i$, $i\in I$, $1\leq d\leq l+1$, and
delete the arrows among the vertices $\tau^{l+1}P_i$, $i\in I$. The
resulting ice quiver $\tQ$ is called a \emph{level $l$ ice quiver with $z$-pattern},
where the vertices corresponding to $\tau^{l+1}P_i$, $i\in I$,
are chosen to be frozen. In this case we also denote it by
$\tQ^z_l$. The associated $(l+1)n\times ln$-matrix $\tB$ is denoted by
$\tB^z$ or $\tB^z_l$.

\begin{Eg}
The quiver in Figure \ref{fig:levelOneQuiver} is a level $1$ ice
quiver with $z$-pattern. The quiver in Figure \ref{fig:levelTwoQuiver} is a level $2$ ice
quiver with $z$-pattern.
\end{Eg}

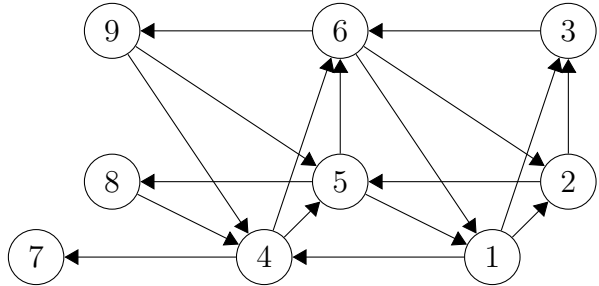
\begin{figure}[htb!]
 \centering
\beginpgfgraphicnamed{fig:levelTwoQuiver}

  \begin{tikzpicture}
    \node [shape=circle, draw] (v1) at (2,-3) {1}; 
    \node  [shape=circle, draw] (v2) at (3,-2) {2}; 
    \node [shape=circle,  draw] (v3) at (3,0) {3};

\node [shape=circle, draw] (v4) at (-1,-3) {4}; 
    \node  [shape=circle, draw] (v5) at (0,-2) {5}; 
    \node [shape=circle,  draw] (v6) at (0,0) {6};

\node [shape=circle, draw] (v7) at (-4,-3) {7}; 
    \node  [shape=circle, draw] (v8) at (-3,-2) {8}; 
    \node [shape=circle,  draw] (v9) at (-3,0) {9};

    \draw[-triangle 60] (v1) edge (v2); 
    \draw[-triangle 60] (v2) edge (v3); 
    \draw[-triangle 60] (v1) edge (v3);
    
    \draw[-triangle 60] (v5) edge (v1); 
    \draw[-triangle 60] (v6) edge (v1); 
    \draw[-triangle 60] (v6) edge (v2); 
    
    \draw[-triangle 60] (v1) edge (v4); 
    \draw[-triangle 60] (v2) edge (v5); 
    \draw[-triangle 60] (v3) edge (v6);

    \draw[-triangle 60] (v4) edge (v5); 
    \draw[-triangle 60] (v4) edge (v6); 
    \draw[-triangle 60] (v5) edge (v6);
    
    \draw[-triangle 60] (v4) edge (v7); 
    \draw[-triangle 60] (v5) edge (v8); 
    \draw[-triangle 60] (v6) edge (v9); 
    
    \draw[-triangle 60] (v8) edge (v4); 
    \draw[-triangle 60] (v9) edge (v5); 
    \draw[-triangle 60] (v9) edge (v4);

  \end{tikzpicture}
\endpgfgraphicnamed
\caption{An ice quiver $\tQ^z_2$ of level $2$ with $z$-pattern}
\label{fig:levelTwoQuiver}
\end{figure}

\subsection{Quantum cluster algebras}

\label{sec:quantumClusterAlgebras}

We refer the reader to \cite{Qin10} for detailed definitions and important properties.
		
Following \cite{BerensteinZelevinsky05}, we define (generalized) quantum cluster algebras over
$(R,v)$, where $R$ is an integral domain and $v$ an invertible element
in $R$. In the present paper we shall only be
interested in the case where $(R,v)=(\Z[v^\pm],v)$ for a formal
parameter $v$. We also denote $v^2$ by $q$ and $v$ by $q^\Hf$.


Let $m\geq n$ be two positive integers. Let $\Lambda$ be an $m\times m$ skew-symmetric integer matrix and $\tB$ an $m\times n$ integer matrix. The upper $n\times n$ submatrix of $\tB$, denoted by $B$, is called the \emph{principal part} of $\tB$.
\begin{Def}[Compatible pair]\label{def:compatible}
The pair $(\Lambda, \tB)$ is called \emph{compatible} if we have
\begin{align}\label{eq:BZ_compatible}
\Lambda(-\tB)=\begin{bmatrix}D\\0 \end{bmatrix}
\end{align}
for some $n\times n$ diagonal matrix $D$ whose diagonal entries are
strictly positive integers. It is called a \emph{unitally compatible
  pair} if moreover $D$ is the identity matrix $\id_n$.
\end{Def}

Let $(\Lambda,\tB)$ be a compatible pair. The component $\Lambda$ is called the \emph{$\Lambda$-matrix} of $(\Lambda,\tB)$, and the component $\tB$ the \emph{$B$-matrix} of $(\Lambda,\tB)$.

\begin{Prop}\cite[Proposition 3.3]{BerensteinZelevinsky05}
The $B$-matrix $\tB$ has full rank $n$, and the product $D B$ is skew-symmetric.
\end{Prop}

We write $\Lambda(g,h)$ for $g^T \Lambda h$, $g,h\in \Z^m$, where $(\ )^T$ means
taking the matrix transposition.

\begin{Def}[Quantum torus]\label{def:quantum_torus}
The \emph{quantum torus }$\cT=\cT(\Lambda)$ over $(R,v)$ is the Laurent polynomial ring
$\vBase[x_1^\pm,\ldots,x_m^\pm]$, endowed with the twisted product $*$
\st we have
\[
x^{g}*x^{h}=v^{\Lambda(g,h)}x^{g+h}
\]
for any $g$ and $h$ in $\Z^m$. Here for any $g=(g_i)_{1\leq i\leq
  n}\in\Z^m$, $x^g$ denotes the monomial $\prod_{1\leq i\leq m} X_i^{g_i}$.
\end{Def}
We denote the usual product in $\cT$ by $\cdot$, and often omit this
notation. 

 Assume that $R$ is endowed with an involution sending each element
 $r$ to $\overline{r}$, \st $\overline{v}$ equals $v^{-1}$. We extend the
 involution of $R$ to an involution (anti-automorphism) of $\cT$ by defining
 $\overline{x^g}=x^g$ for any $g\in \Z^m$.

A \emph{sign} $\epsilon$ is an element in $\{-1,+1\}$. Denote by $b_{ij}$ the entry in position $(i,j)$ of $\tB$. For any $1\leq k\leq n$ and any sign $\epsilon$, we associate to $\tB$ an $m\times m$ matrix $E_\epsilon$ whose entry in position $(i,j)$ is 
\begin{align}\label{eq:E_epsilon}
e_{ij}=\left\{
\begin{array}{ll}
\delta_{ij} & \textrm{if $j\neq k$}\\
-1 & \textrm{if $i=j=k$}\\
max(0,-\epsilon b_{ik}) & \textrm{if $i\neq k, j=k$},
\end{array} \right.
\end{align}
and an $n\times n$ matrix $F_\epsilon$ whose entry in position $(i,j)$ is
\[
f_{ij}=\left\{
\begin{array}{ll}
\delta_{ij} & \textrm{if $i\neq k$}\\
-1 & \textrm{if $i=j=k$}\\
max(0,\epsilon b_{kj} ) & \textrm{if $i=k$, $j\neq k$}.
\end{array} \right.
\]

Fix a compatible pair $(\Lambda,\tB)$ and the quantum torus
$\cT=\cT(\Z^m,\Lambda)$. Notice that the quantum torus $\cT=\cT(+,*)$ is contained in
its \emph{skew-field of fractions}, which is denoted by $\cF$, \cf \cite[Appendix]{BerensteinZelevinsky05}.

In the following, we consider triples $(\Lambda',\tB',x')$ \st $(\Lambda',\tB')$ is a
compatible pair and $x'=(x'_1,\cdots,x'_m)$ is an $m$-tuple of
elements in the quantum torus $\cF$.

Let $\sT_n$ be an $n$-regular tree with root $t_0$. There is a
unique way of associating a triple $(\Lambda(t),\tB(t),x(t))$ with each vertex $t$ of $\sT_n$ \st we have
\begin{enumerate}
\item $(\Lambda(t_0),\tB(t_0),x(t_0))=(\Lambda,\tB,x)$, where
  $x=(x_1,\cdots,x_m)$, and
\item if two vertices $t$ and $t'$ are linked by an edge labeled $k$,
  then the seed 
  $(\Lambda(t'),\tB(t'),X(t'))$ is obtained from
  $(\Lambda(t),\tB(t),X(t))$ by the mutation at $k$ defined as below.
\end{enumerate}

\begin{Def}[Mutation \cite{BerensteinZelevinsky05}]\label{def:seed_mutation}
Given any sign $\epsilon$, the new triple \linebreak[4]$(\Lambda(t'),\tB(t'),x(t'))$ obtained from $(\Lambda(t),\tB(t),x(t))$ by the mutation
at $k$ is given by
\begin{align}
(\Lambda(t'),\tB(t'))=(E_\epsilon(t) ^T \Lambda(t)  E_\epsilon(t) , E_\epsilon(t)  \tB(t)  F_\epsilon(t) ),
\end{align}
and 
\begin{align}
\begin{split}x_k(t)* x_k(t')=&v^{\Lambda(t)(e_k,\sum_{1\leq i\leq m}[b_{ik}(t)]_{+} e_i)}\prod_{1\leq i\leq m}x_i(t)^{[b_{ik}]_{+}(t)}\\
&+v^{\Lambda(t)(e_k,\sum_{1\leq j\leq m}[-b_{jk}(t)]_{+} e_j)}\prod_{1\leq j\leq m}x_j(t)^{[-b_{jk}]_{+}(t)},\end{split}\\
x_i(t')=&x_i(t) ,\quad 1\leq i\leq m,\quad i\neq k.
\end{align}
\end{Def}
Notice that here $[\ ]_+$ denotes the function $\max\{0,\ \}$.
We recall from \cite{BerensteinZelevinsky05} that $\mu_k$ is an
involution, and is independent of the choice of
$\epsilon$.

The triples
$(\Lambda(t),\tB(t),x(t))$,
$t\in\sT_n$, are called the (quantum) seeds. The elements $x_i(t)$, $1\leq i\leq m$, $t\in\sT_n$, are called the $x$-variables. The
$x$-variables $x_i(t)$, $1\leq i\leq n$, $t\in\sT_n$, are called the
\emph{quantum cluster variables}. For each $t\in \sT_n$, each monomial
in the $x_i(t)$, $1\leq i\leq n$, is called a \emph{quantum cluster monomial}. Notice that for $j>n$, the $x$-variables $x_{j}(t)$ do not depend on $t$.

\begin{Def}[Quantum cluster algebra]
The \emph{quantum cluster algebra} $\qClAlg=\qClAlg(+,*)$ over $(R,v)$ is the
$R$-subalgebra of $\cF$ generated by the quantum cluster variables
$x_i(t)$ for all the vertices $t$ of $\sT_n$ and $1\leq i\leq n$, and
the elements $x_{j}$ and $x_{j}^{-1}$ for all $j>n$.
\end{Def}

The specialization $\clAlg=\qClAlg|_{v\mapsto 1}$ is called the \emph{classical
cluster algebra}, which is also denoted by $\cA^\Z$. If $R=\Z[v^\pm]$, we say that $\qClAlg$ and $\clAlg$ are \emph{integral}. 

\begin{Thm}[Quantum Laurent phenomenon]\cite[Section 5]{BerensteinZelevinsky05}\label{thm:Laurent}
The quantum cluster algebra $\qClAlg$ is a subalgebra of $\torus$.
\end{Thm}

Similarly, the cluster algebra $\clAlg=\cA^\Z$ is a subalgebra of the
Laurent polynomial ring $\torus|_{v\mapsto 1}=\torus^\Z$.

\subsection{Cluster category and quantum cluster variables} \label{sec:cluster_category}

Let $\tQ$ and $\tB$ be given as in Section \ref{sec:frozenQuiver}. We associate to $\tB$ the cluster algebra $\cA^\Z$ as in \cite{FominZelevinsky07}. Let
the base field $k$ be the complex field $\C$. Let $\tW$ be a generic potential on $\tQ$ in the sense of \cite{DerksenWeymanZelevinsky08}. As in \cite{KellerYang09}, with the quiver with potential $(\tQ,\tW)$ we can associate the Ginzburg algebra $\Gamma=\Gamma(\tQ,\tW)$. Denote the perfect derived category of $\Gamma$ by $\per\Gamma$ and denote the full subcategory of $\per\Gamma$ whose objects are dg modules with finite dimensional homology by $\Dfd\Gamma$. The \emph{generalized cluster category} $\cC=\cC_{(\tQ,\tW)}$ in the sense of \cite{Amiot09} is the quotient category
\[
\cC=\per\Gamma/\Dfd\Gamma.
\]
Denote the quotient functor by $\pi:\per\Gamma\ra \cC$ and define
\begin{align*}
T_i&=\pi(e_i\Gamma),\quad 1\leq i\leq m,\\
T&=\oplus_{1\leq i\leq m} T_i.
\end{align*}
It is shown in \cite{Plamondon10c} that the endomorphism algebra of $T$ is isomorphic to $H^0\Gamma$.

For any triangulated category $\cU$ and any rigid object $X$ of $\cU$, we define the subcategory $\pr_\cU (X)$ of $\cU$ to be the full subcategory consisting of the objects $M$ \st there exists a triangle in $\cU$
\[
M_1\ra M_0\ra M\ra \Sigma M_1,
\]
for some $M_1$ and $M_0$ in $\add X$. The \emph{presentable cluster category} $\cD\subset \cC$ is defined as the full subcategory consisting of the objects $M$ \st
\[
M\in\pr_\cC (T)\cap\pr_\cC (\Sigma^{-1}T)\quad\text{and}\quad\dim\Ext^1_\cC(T,M)<\infty,
\]
\cf \cite{Plamondon10a}.

We refer to \cite{Plamondon10a} for the definition of the iterated mutations of the object $T$. There is a unique way of associating an object $T(t)=\oplus_{1\leq i\leq m} T_i(t)$ of $\cD$ with each vertex $t$ of $\sT_n$ \st we have
\begin{enumerate}
\item $T(t_0)=T$, and
\item if two vertices $t$ and $t'$ are linked by an edge labeled $k$, then the object $T(t')$ is obtained from $T(t)$ by the mutation at $k$. 
\end{enumerate}

Let $\cF\subset\per\Gamma$ denote the full subcategory $\pr_{\per\Gamma}(\Gamma)$. The quotient functor $\pi:\per\Gamma\ra \cC$ induces an equivalence $\cF\iso\pr_\cC(T)$. Denote by $\pi^{-1}$ the inverse equivalence. For an object $M\in \pr_{\cC}(T)$, we define its \emph{index }$\ind_T M$ as the class $[\pi^{-1}M]$ in $\Kz(\per\Gamma)$. 

\begin{Thm}\cite{Plamondon10a}\label{thm:g_vector_coordinate}
(1) For any vertex $t$ of $\sT_n$, the classes $[\ind_T T_i(t)]$ form a basis of $\Kz(\per\Gamma)$. 

(2) For a class $[P]$ in $\Kz(\per\Gamma)$, let $[[P]:T_i(t)]$ denote
its $i$th coordinate in this basis. Then we have $([\ind_T
T_i(t):T_j])_{1\leq j\leq m}=\tg_i(t)$, where $\tg_i(t)$ is the $i$th
extended $g$-vector associated with $t$, \cf \cite{FominZelevinsky07}.
\end{Thm}

\begin{Def}[Coefficient-free objects]\label{def:CoefFree}
An object $M$ in $\cC$ is called \emph{coefficient-free} if
\begin{enumerate}
\item  the object $M$ does not contain a direct summand $T_i$, $i>n$, and
\item  the space $\Ext^1_{\cC}(T_i,M)$ vanishes for $i>n$.
\end{enumerate}
\end{Def}
For a coefficient-free object $M\in\cC$, the space $\Ext^1_{\cC}(T,M)$ is a right $H^0\Gamma$-module whose support is concentrated on $Q$. Thus, it can be viewed as a $\cP(Q,W)$-module, where $W$ is the potential on $Q$ obtained from $\tW$ by deleting all cycles through vertices $j>n$ and $\cP(Q,W)$ is the Jacobi algebra of $(Q,W)$. Denote by $\phi:\Kz(\mod \cP(Q,W))\ra \Kz(\per\Gamma)$ the map induced by the composition of inclusions $\mod\cP(Q,W)\ra \Dfd\Gamma \ra \per\Gamma$. For any vertex $i$ of $\tQ$, it sends $[S_i]$ to
\begin{align}\label{eq:phi}
\sum_{\mathrm{arrows}\, i\ra j}[e_j\Gamma]-\sum_{\mathrm{arrows}\, l\ra i}[e_l\Gamma],
\end{align}
as one easily checks using the minimal cofibrant resolution of the
simple dg $\Gamma$-module $S_i$, \cf \cite{KellerYang09}. Thus, the
matrix of $\phi$ in the natural bases is $-\tB$.

By the \emph{twisted Poincar\'{e}
polynomial} of a topological space $Z$, we mean the polynomial
$p_t(Z)=\sum_p (-1)^p\dim H^p(Z,\Q)$. When $Q$ is acyclic, we have the following construction.
\begin{Def}[Quantum CC-formula, {\cite{Qin10}}]\label{def:CC_formula}
For any coefficient-free and rigid object $M\in\cD$, we denote by $m$ the class of $\Ext^1_{\cC}(T,M)$ in $\Kz(\mod kQ)$, and associate to $M$ the following element in $\cT^\Z$:
\begin{align}\label{eq:CC_formula}
x_M=\sum_e p_{q^{\Hf}}(\Gr_e (\Ext^1_{\cC}(T,M))) q^{-\Hf\dim \Gr_e (\Ext^1_{\cC}(T,M))}x^{\ind_T(M)-\phi( e)},
\end{align}
where $p_{q^{\Hf}}(\ )$ denotes the twisted Poincar\'{e}
polynomial, and $\Gr_e (\Ext^1_{\cC}(T,M))$ is the submodule Grassmannian of $\Ext^1_{\cC}(T,M)$ whose $\C$-points are the submodules of the class $e$ in $\Kz(\mod\cP(Q,W))$.
\end{Def}

The following theorem is the main result of \cite{Qin10}.
\begin{Thm}[{\cite{Qin10}}] \label{thm:CC_formula}
Assume that the quiver $Q$ is acyclic. For any vertex $t$ of $\sT_n$ and any $1\leq i\leq n$, we have
\[
x_{T_i(t)}=x_i(t).
\]
Moreover, the map taking an object $M$ to $x_M$ induces
a bijection from the set of isomorphism classes of coefficient-free
rigid objects of $\cD$ to the set of quantum cluster monomials of $\qClAlg$.
\end{Thm}

\subsection{Deformed Grothendieck ring via graded quiver
  varieties}\label{sec:notations}

In \cite{Nakajima09}, Nakajima used graded quiver varieties associated
with bipartite quivers to construct deformed Grothendieck rings. In
order to generalize his result to study acyclic quantum cluster
algebras, we will use a new family of graded quiver varieties and a
modified version of deformed Grothendieck rings for acyclic quivers
$Q$, which have been studied in \cite{Qin11}. For the convenience of the reader, we shall recall the
basic definitions and properties of these constructions.

Notice that, by the dimension of a complex variety, we always mean the
complex dimension. By default, we \emph{only} consider geometric points.

\subsubsection*{Graded quiver varieties}

Assume that $Q$ is an acyclic quiver with the set of vertices
$I=\{1,\ldots, n\}$, \st $b_{ij}\leq 0$ whenever $i\geq j$, \cf Section
\ref{sec:frozenQuiver} for the definition of $b_{ij}$.

We denote the finitely supported bigraded vectors in $\N^{I\times \Z}$ by $w=(w_i(a))_{i\in I,
  a\in \Z}$, and the finitely supported bigraded vectors in $\N^{I\times(\Z+\Hf)}$ by
$v=(v_i(a))_{i,a}$. Let the associated graded complex vector spaces be $W=\C^w=\oplus_{i,a}W_i(a)=\oplus_{i,a} \C^{w_i(a)}$ and similarly $V=\C^v=\oplus_{i,a} V_i(a)=\oplus_{i,a} \C^{v_i(a)}$. 

The vectors $w$ and $v$ can be naturally viewed as elements in
$\Z^{I\times \R}$. For any $d\in\R$, define the degree shift $[d]$ of vectors
$\eta=\eta_i(a)\in\Z^{I\times \R}$ by $\eta[d]_i(a)=\eta_i(a+d)$. For any two vectors $\eta^1=\eta^1_i(a)$, $\eta^2=\eta^2_i(a)$ in $\Z^{I \times \R}$, if at least
one of them has finite support, their inner product is defined as $\eta^1\cdot \eta^2=\sum_{i,a}\eta^1_i(a)\eta^2_i(a)$.

The Cartan matrix $C$ associated with $Q$ is the $I\times I$ matrix whose entry in position $(i,j)$ is 
\begin{align}
c_{ij}=\left\{
\begin{array}{ll}
2 & \textrm{if $i= j$}\\
-|b_{ij}| & \textrm{if $i\neq j$}
\end{array} \right. .
\end{align}

\begin{Def}[$q$-Cartan matrix]\label{def:cartanMatrix}
We define the linear map
\begin{align*}
  C_q:\Z^{I\times (\Hf+\Z)}\ra \Z^{I\times \Z}
\end{align*}
\st for each $\eta\in \Z^{I\times (\Hf+\Z)}$, we have
\begin{align}
(C_q \eta)_k(a)=\eta_k(a+\Hf)+\eta_k(a-\Hf)-\sum_{i:1\leq i< k}b_{ik} \eta_i(a+\Hf)-\sum_{j:k< j\leq n}b_{kj}\eta_j(a-\Hf).
\end{align}
It is called a \emph{$q$-analogue} of the Cartan matrix $C$, or a \emph{$q$-Cartan matrix} for short.
\end{Def}

It is easy to see that $C_q$ naturally extends to a linear map from $\Z^{I\times \R}$ to $\Z^{I\times \R}$, which is also denoted by $C_q$. Furthermore, $C_q$ commutes with $[d]$, $d\in \R$.

The $q$-Cartan matrix $C_q$ is symmetric in the following sense.
\begin{Lem}\label{lem:symmetricCartan}
  For any two vectors $\eta^1,\eta^2\in \Z^{I\times \R}$, if \emph{at
    least one of them has finite support}, we have the following
  identity:
  \begin{align}
    \label{eq:symmetricForm}
    \eta^2 \cdot C_q\eta^1[-\Hf]=C_q\eta^2\cdot \eta^1[-\Hf].
  \end{align}
\end{Lem}

\begin{Def}[$l$-Dominance]
A pair $(v,w)$ is
called \emph{$l$-dominant} \footnote{The $l$ here stands for ``loop''
  in the quantum loop algebras.} if $w-C_q v\in\N^{I\times \Z}$.

Define the
\emph{dominance order} on the set of pairs $(v,w)$, \st
$(v',w')\leq (v,w)$ if $w'-C_q v'=w-C_q (v+v'')$, for some $v''\in
\N^{I\times(\Z+\Hf)}$.

We say $v'\leq v$ if $(v',w)\leq (v,w)$, or equivalently, if $v_i'(a)\geq v_i(a)$, for all $(i,a)$. We say
$w'\leq w$ if $(0,w')\leq (0,w)$.
\end{Def}




\begin{Def}[Weight order]\label{def:weightOrder}
We define the total order (weight order) $\wtLess$ on the set of pairs
$(i,a)$, $i\in I$, $a\in \R$, \st $(i',a')\wtLess(i,a)$ if $a'<a$, or if $a'=a$ and $i'<i$.
\end{Def}

Denote the set $\{\eta\in \Z^{I\times \Z}|\eta_k(a)=0,\ \forall k\in
I,\ a\ll 0\}$ by $E$. Denote the set $\{\eta[\Hf]|\eta\in E\}$ by $E[\Hf]$.

\begin{Lem}
1) $C_q$ restricts to an isomorphism from $E[\Hf]$ to $E$.
 
2) $C_q$ restricts to an isomorphism from $E$ to $E[-\Hf]$. 
\end{Lem}

We denote the inverses of both the restriction maps by $C_q^{-1}$.

Let $e_{k,a}\in \N^{I\times \Z}$ be the unit vector concentrated at
the degree $(k,a)$. Denote by $p_{i,j}$ the dimension of $\Hom_{\mod \C Q}(P_i,P_j)$, $i\leq j$, namely the number of (possibly trivial) paths from $i$ to $j$ in the quiver $Q$.

\begin{Lem}\label{lem:inverseAlgorithm}
The vector $\te_{k,a}=C_q^{-1}(e_{k,a})$ satisfies, for all $k'\in I$, 
\begin{equation}
\label{eq:cartanInverse}
\te_{k,a}\cdot e_{k',b}=\left\{
  \begin{array}{l l}
    0&\quad \text{if }b<a+\Hf\\
p_{kk'}&\quad \text{if }b=a+\Hf
  \end{array}\right.
.
\end{equation}
\end{Lem}

Consider the opposite quiver $Q\op$ and denote its set of
arrows by $\Omega$. Define the function $\epsilon:H\ra \{\pm 1\}$ \st
$\epsilon(\Omega)=\{1\}$, $\epsilon(\oOmega)=\{-1\}$. Similarly, for a
linear map $B_h$ indexed by $h\in H$, we define $(\epsilon
B)_h=\epsilon(h)B_h$. Given finitely supported vectors $v$, $v'$ in $\N^{I\times(\Z+\Hf)}$ and $w$ in $\N^{I\times \Z}$, we define the following graded vector spaces
\begin{align*}
\grEndSp(v,v')&=\oplus_{(i,a)} \Hom(V_i(a),V'_i(a)),\\
\grLSp(w,v)&=\oplus_{(i,a)} \Hom(W_i(a),V_i(a-\Hf)),\\
\grLSp(v,w)&=\oplus_{(i,a)}  \Hom(V_i(a),W_i(a-\Hf)),\\
 \grESp(v,v')&=(\oplus_{h\in\Omega,a}
 \Hom(V_{s(h)}(a),V'_{t(h)}(a)))\\&\qquad \oplus (\oplus_{\oh\in\oOmega,a} \Hom(V_{s(\oh)}(a),V'_{t(\oh)}(a-1))).
\end{align*}
Notice that the weight order strictly decreases along the linear maps in the last three spaces.

Consider the affine space
\begin{equation}
\grRep(Q\op,v,w)=\grESp(v,v)\oplus\grLSp(w,v) \oplus \grLSp(v,w).
\end{equation}
Denote the coordinates of its points by 
\begin{equation}
  \begin{split}
(B,\alpha,\beta)=((B_{h})_{h\in H},\alpha,\beta)=((b_h)_{h\in\Omega}, (b_{\oh})_{\oh\in\oOmega},(\alpha_i)_i,(\beta_i)_i)\\=((\oplus_a b_{h,a})_h,(\oplus_a b_{\oh,a})_{\oh},(\oplus_a\alpha_{i,a})_i,(\oplus_a\beta_{i,a})_i).
  \end{split}
\end{equation}
Define the analogue of the moment map $\mu: \grRep(Q\op, v,w)\ra \grEndSp(v,v[-1])$ \st
\begin{align}
  \begin{split}
    \mu(B,\alpha,\beta)&=\oplus_{a\in \Z+\Hf,i\in
      I}\mu(B,\alpha,\beta)_{i,a}\\
    &=\oplus_{i,a}(\sum_{h\in\Omega}(b_{h,a} b_{\oh,a+1}-b_{\oh,a+1}
    b_{h,a+1}) +\alpha_{i,a+\Hf} \beta_{i,a+1}),
  \end{split}
\end{align}
or $\mu(B,\alpha,\beta)=(\epsilon B) B+\alpha\beta$ for short.

\begin{Eg}
Let $Q$ be given as in Figure \ref{fig:acyclicQuiver}. Figure \ref{fig:repSpace} is an example of the vector space $\grRep(Q\op,v,w)$,
  where the vertices denote the $(i,a)$-degree components of $W$ and
  $V$ and the arrows
  denote the corresponding linear maps. The components in the same
  rows have the same $i$-degrees, and those in the same columns have
  the same $a$-degrees (or degrees for short).
\end{Eg}

\begin{figure}[htb!]
\centering
  \begin{tikzpicture}
	\begin{pgfonlayer}{nodelayer}
		\node [style=white box] (0) at (-2.5, 2) {$V_3(-\frac{1}{2})$};
		\node [style=white box] (1) at (2.5, 2) {$V_3(\frac{1}{2})$};
		\node [style=black box] (2) at (-5, 1.5) {$W_3(-1)$};
		\node [style=black box] (3) at (0, 1.5) {$W_3(0)$};
		\node [style=black box] (4) at (5, 1.5) {$W_3(1)$};
		\node [style=white box] (5) at (-3.5, 0) {$V_{2}(-\frac{1}{2})$};
		\node [style=white box] (6) at (1.5, 0) {$V_{2}(\frac{1}{2})$};
		\node [style=black box] (7) at (-6, -0.5) {$W_{2}(-1)$};
		\node [style=black box] (8) at (-0.5, -0.5) {$W_{2}(0)$};
		\node [style=black box] (9) at (4, -0.5) {$W_{2}(1)$};
		\node [style=white box] (10) at (-2.5, -2) {$V_{1}(-\frac{1}{2})$};
		\node [style=white box] (11) at (2.5, -2) {$V_{1}(\frac{1}{2})$};
		\node [style=black box] (12) at (-5, -2.5) {$W_{1}(-1)$};
		\node [style=black box] (13) at (0, -2.5) {$W_{1}(0)$};
		\node [style=black box] (14) at (5, -2.5)
                {$W_{1}(1)$};

		\node [] (15) at (-3.75, -2) {$\beta_1$};
		\node [] (16) at (-0.9, -2) {$\alpha_1$};
		\node [] (17) at (-3.25, -1.3)
                {$h$};
		\node [] (18) at (-0.5, -1.3)
                {$\overline{h}$};

	\end{pgfonlayer}
	\begin{pgfonlayer}{edgelayer}
		\draw [->, thick, shorten <=2 pt, shorten >=2 pt, >=stealth'] (1) to (3);
		\draw [->, thick, shorten <=2 pt, shorten >=2 pt, >=stealth'] (3) to (0);
		\draw [->, thick, shorten <=2 pt, shorten >=2 pt, >=stealth'] (14) to (11);
		\draw [->, thick, shorten <=2 pt, shorten >=2 pt, >=stealth'] (11) to (13);
		\draw [->, thick, shorten <=2 pt, shorten >=2 pt, >=stealth'] (13) to (10);
		\draw [->, thick, shorten <=2 pt, shorten >=2 pt, >=stealth'] (1) to (6);
		\draw [->, thick, shorten <=2 pt, shorten >=2 pt,
                >=stealth'] (10) to (12);
		\draw [->, thick, shorten <=2 pt, shorten >=2 pt, >=stealth'] (8) to (5);
		\draw [->, thick, shorten <=2 pt, shorten >=2 pt, >=stealth'] (5) to (7);
		\draw [->, thick, shorten <=2 pt, shorten >=2 pt, >=stealth'] (6) to (11);
		\draw [->, thick, shorten <=2 pt, shorten >=2 pt, >=stealth'] (6) to (0);
		\draw [->, thick, shorten <=2 pt, shorten >=2 pt, >=stealth'] (9) to (6);
		\draw [->, thick, shorten <=2 pt, shorten >=2 pt, >=stealth'] (4) to (1);
		\draw [->, thick, shorten <=2 pt, shorten >=2 pt, >=stealth'] (0) to (5);
		\draw [->, thick, shorten <=2 pt, shorten >=2 pt, >=stealth'] (11) to (0);
		\draw [->, thick, shorten <=2 pt, shorten >=2 pt, >=stealth'] (6) to (8);
		\draw [->, thick, shorten <=2 pt, shorten >=2 pt, >=stealth'] (5) to (10);
		\draw [->, thick, shorten <=2 pt, shorten >=2 pt, >=stealth'] (1) to (11);
		\draw [->, thick, shorten <=2 pt, shorten >=2 pt, >=stealth'] (0) to (10);
		\draw [->, thick, shorten <=2 pt, shorten >=2 pt, >=stealth'] (0) to (2);
		\draw [->, thick, shorten <=2 pt, shorten >=2 pt, >=stealth'] (11) to (5);
	\end{pgfonlayer}
\end{tikzpicture}
\caption{Vector space $\grRep(Q\op,v,w)$}
\label{fig:repSpace}
\end{figure}

The group $G_v=\prod_{i,a} GL(V_{i,a})$ acts naturally on the level
set $\mu^{-1}(0)$ such that for $g=(g_{i,a})\in G_v$, we have $g(\alpha)=g
\alpha$, $g(\beta)=\beta g^{-1}$, $g(b_h)= g_{t(h)}b_h g_{s(h)}^{-1}$, $g(b_{\oh})= g_{t({\oh})}b_{\oh}
g_{s({\oh})}^{-1}$. We fix the character of $\chi$ of $G_v$ such that $\chi(g)=\prod_{i,a}(\det g_{i,a})^{-1}$.

The graded quasi-projective quiver variety $\grProjQuot(v,w)$ is
defined to be the geometric invariant theory quotient (GIT quotient
for short) of $\mu^{-1}(0)$ with
respect to $\chi$, and the graded affine quiver variety
$\grAffQuot(v,w)$ is defined to be the categorical quotient of
$\mu^{-1}(0)$ by the action of $G_v$. Then there is a natural projective morphism $\pi$ from $\grProjQuot(v,w)$ to $\grAffQuot(v,w)$. 

Denote the fibre of $\pi$ over a point $x$ by $\grFib_x(v,w)$. When
$x=0$, we also denote the fiber by $\grLag=\grLag(v,w)$.

\begin{Rem}
If the quiver $Q\op$ is bipartite, our $q$-Cartan matrix $C_q$ and quiver varieties are isomorphic
to those defined in \cite{Nakajima09}. This can be seen by applying
appropriate shifts in the degrees $a$ of the vectors $v_i(a)$, $w_i(a)$.
\end{Rem}

Let us define $\grRegStratum(v,w)$ to be the set of points in
$\grAffQuot(v,w)$ \st the stabilizers in $G_v$ of their representatives are
trivial. Then the morphism $\pi$ is an isomorphism from
$\pi^{-1}(\grRegStratum(v,w))$ to $\grRegStratum(v,w)$, \cf
\cite{Qin11}.

\begin{Prop}
  \label{prop:nonemptyStratum}
$\grRegStratum(v,w)$ is non-empty if and only if $(v,w)$ is $l$-dominant.
\end{Prop}

\begin{Thm}[Transversal slice] \label{thm:grTransversalSlice}
Assume $x$ is a point in $\grRegStratum(v^0,w)$, which is naturally embedded into a quotient $\grAffQuot(v,w)$. Let $T$ be the tangent space of $\grRegStratum(v,w)$ at $x$. As $(v^0,w)$ is $l$-dominant, define $w^\bot=w-C_qv^0$, $v^\bot=v-v^0$. Then there exist neighborhoods $U$, $U_T$, $U^\bot$ of $x\in \grAffQuot(V,W)$, $0\in T$, $0\in \grAffQuot(v^\bot,w^\bot )$ respectively, and biholomorphic maps $U\ra U_T\times U^\bot $, $\pi^{-1}(U)\ra U_T\times \pi^{-1}(U^\bot )$, \st the following diagram commutes:
\begin{equation*}
\begin{CD}
  \grProjQuot(v,w) \;@. \supset \; @. \pi^{-1}(U) @>>\cong> U_T \times \pi^{-1}(U^\bot ) \;@.\subset \;@. T\times \projQuot(v^\bot ,w^\bot )\\
   @. @. @V{\pi}VV @VV{\id\times\pi}V @.@. \\
   \grAffQuot(v,w)\;@.\supset \; @. U @>>\cong> U_T\times U^\bot  \;@. \subset \;@. T\times\grAffQuot(v^\bot ,w^\bot )
\end{CD}
\end{equation*}

In particular, the fibre $\grFib_x(v,w)=\pi^{-1}(x)$ is biholomorphic to the zero fibre $\grLag(v^\bot ,w^\bot )$ over $0\in \grAffQuot(v^\bot ,w^\bot )$. 
\end{Thm}

\begin{Prop}
  \label{prop:grStratification}
We have a stratification
\begin{align}
  \label{eq:grStratification}
  \grAffQuot(v,w)=\sqcup_{(v',w)\geq (v,w)}\grRegStratum(v',w).
\end{align}
In particular, the variety $\grAffQuot(w)=\cup_v\grAffQuot(v,w)$ has a stratification
\begin{align*}
  \grAffQuot(w)=\sqcup_v\grRegStratum(v,w).
\end{align*}
\end{Prop}

\begin{Cor}\label{cor:degeneratedStratum}
Let $(v,w)$,$(v^0,w)$ be $l$-dominant pairs. Then we have $\grRegStratum(v^0,w)\subset \overline{\grRegStratum(v,w)}$ if and
only if $(v^0,w)\geq (v,w)$.
\end{Cor}

For any two pairs of vectors $(v^1,w^1)$, $(v^2,w^2)$, define quadratic forms $d((v^1,w^1),(v^2,w^2))$, $\dT((v^1,w^1),(v^2,w^2))$, $\dTW(w^1,w^2)$, and
$\eMatrix(w^1,w^2)$ \st
\begin{align}
  \label{eq:dimension}
d((v^1,w^1),(v^2,w^2))=(w^1-C_qv^1)\cdot v^2[-\Hf]+ v^1\cdot w^2[-\Hf],
\end{align}
\begin{align}
  \label{eq:dimensionW}
\dTW(w^1,w^2)=-w^1[\Hf] \cdot C_q^{-1}w^2,
\end{align}
\begin{align}
  \label{eq:twistDimension}
  \eMatrix(w^1,w^2)=-w^1[\Hf] \cdot C_q^{-1}w^2+w^2[\Hf] \cdot C_q^{-1}w^1.
\end{align}
\begin{align}
  \dT((v^1,w^1),(v^2,w^2))=d((v^1,w^1),(v^2,w^2))+\dTW(w^1,w^2),
\end{align}
Notice that $\dT((0,w^1),(0,w^2))$ equals $\dTW(w^1,w^2)$.

\begin{Rem}
  Our $\dT$ and $\dTW$ are different from $\td$ and $\td_W$ in
  \cite{Nakajima04}, but the properties are similar.
\end{Rem}

\begin{Lem}
The following equality holds:
\begin{align}\label{eq:chopTwistDimension}
  \dT((v^1,w^1),(v^2,w^2))=\dT((0,w^1-C_q v^1),(0,w^2-C_q v^2)).
\end{align}
\end{Lem}

\begin{Thm}[{\cite{Qin11}}]\label{thm:oddVanish}
The graded quiver variety $\grProjQuot(v,w)$ is smooth connected of
dimension $d((v,w),(v,w))$. Furthermore, it is homotopic to the zero fiber
$\grLag(v,w)$, and its odd homology vanishes.
\end{Thm}

\subsubsection*{Deformed Grothendieck ring}

Let $(v,w)$ be a pair of vectors. The map $\pi:\grProjQuot(v,w)\ra \grAffQuot(w)$ is proper since it is the
composition of a projective morphism and a closed embedding. The rank $1$ trivial local system over
$\grProjQuot(v,w)$ yields a perverse sheaf
$1_{\grProjQuot(v,w)}=\underline{\C}_{}[\dim\grProjQuot(v,w)]$. Define $IC_w(v')$ to be the simple perverse sheaf generated by the rank $1$ trivial local system on $\grRegStratum(v',w)$.

By the celebrated decomposition theorem \cite{BeilinsonBernsteinDeligne82}, the sheaf $\pi_w(v)=\pi_!(1_{\grProjQuot(v,w)})$ decomposes into a direct sum of shifts of simple perverse sheaves on $\grAffQuot(w)$. The results in \cite[Theorem 14.3.2]{Nakajima01} can be translated into the following.
\begin{Thm}\label{thm:decomposition}
We have a decomposition
  \begin{align}
    \label{eq:decomposition}
    \pi_w(v)=\oplus_{v':(v',w)\text{ is $l$-dominant}}\oplus_{d\in\Z}a_{v,v';w}^dIC_w(v')[d],
  \end{align}
where the coefficients $a_{v,v';w}^d$ satisfy $a_{v,v';w}^d\in \N$,
$a_{v,v';w}^d=a_{v,v';w}^{-d}$, $a_{v,v;w}^d=\delta_{d0}$ if $(v,w)$
is $l$-dominant, and
$a_{v,v';w}^d$ vanishes unless $(v',w)\geq(v,w)$.
\end{Thm}



We define the Laurent polynomial $a_{v,v';w}(t)$ in the Laurent
polynomial ring $\Z[t^\pm]$ to be
\begin{align*}
  a_{v,v';w}(t)=\sum_{d\in\Z}a_{v,v';w}^d t^d.
\end{align*}

For each $l$-dominant pair $(v,w)$, we define a set 
\begin{align*}
  \cP_w=\{IC_w(v)|(v,w)\text{ is $l$-dominant}\}.
\end{align*}
This set is of finite cardinality.

Let $\cD_c(\grAffQuot(w))$ be the bounded derived category of
constructible sheaves on $\grAffQuot(w)$,  and $\cQ_w$ its full
subcategory whose objects are isomorphic to the direct sums of the
shifts of the objects in $\cP_w$. Then
$\cQ_w$ and $\cP_w$ are stable under the Verdier duality $D$. Let
$\KGp_w$ be the quotient of the free abelian group generated by the
isomorphism classes $(L)$ in $\cQ_w$ modulo the relation
$(L)=(L')+(L'')$ whenever $L$ is isomorphic to $L'\oplus L''$. The
group $\KGp_w$ has a natural $\tBase$-structure \st $t(L)=(L[1])$. The
duality $D$ induces an involution $\overline{(\ )}$ on $\KGp_w$ which
satisfies $\overline{t(L)}=t^{-1}\overline{(L)}$ and $\overline{(IC_w(v))}=(IC_w(v))$

The decomposition \eqref{eq:decomposition} implies that, by abuse of
notation, $\KGp_w$ has two $\tBase$-bases: $\{IC_w(v)|(v,w)\text{ is
  $l$-dominant}\}$ and $\{\pi_w(v)|(v,w)\text{ is $l$-dominant}\}$.

As in \cite{Nakajima09}, we define the abelian group $\dualKGp_w=\Hom_\tBase(\KGp_w,\tBase)$. Let $\{\can_w(v)\}$,
$\{\chi_w(v)\}$ be the bases of $\dualKGp_w$ dual to $\{IC_w(v)\}$,
$\{\pi_w(v)\}$ respectively. Define another basis
$\{\pbw_w(v)|(v,w)\text{ is $l$-dominant}\}$ of $\dualKGp_w$ by
\begin{align*}
\langle \pbw_w(v),L\rangle=\sum_k t^{\dim \grRegStratum(v,w)-k}\dim H^k(i^!_{x_{v,w}}L),
\end{align*}
where $x_{v,w}$ is any point in $\grRegStratum(v,w)$, $i_{x_{v,w}}$
is the inclusion, and $\langle\ ,\ \rangle$ is the canonical
pairing. Notice that the definition of $\pbw_w(v)$ is independent of
the choice of $x_{v,w}$. Indeed, it suffices to check this for the
elements $L$ in the basis $\{\pi_w(v)\}$ and here it follows from
Theorem \ref{thm:grTransversalSlice}.

In the situation of Theorem
\ref{thm:grTransversalSlice}, we have
\begin{align*}
  \langle \pbw_w(v'),IC_w(v)\rangle=\langle
  \pbw_{w^\bot}({v'}^\bot),IC_{w^\bot}(v^\bot)\rangle.
\end{align*}
By the
properties of perverse sheaves, we have
\begin{align}
  \label{eq:canToPbw}
  \can_w(v)\in\pbw_w(v)+\sum_{(v',w)<(v,w)}t^{-1}\Z[t^{-1}]\pbw_{w}(v').
\end{align}
Therefore, $\{\pbw_w(v)|(v,w)\text{ is $l$-dominant}\}$ is a
basis. 

\begin{Def}[{\cite[3.3]{Nakajima09}}]\label{def:quotKGp}
Define $\quotKGp$ to be the infinite rank free
  $\tBase$-module consisting of the functionals
  $(f_w)\in\prod_w\dualKGp_w$ \st we have $\langle
    f_w,IC_w(v)\rangle=\langle
    f_{w^\bot},IC_{w^\bot}(v^\bot)\rangle$ for any $l$-dominant pairs
    $(v,w)$, $(w^\bot,v^\bot)$ appearing in Theorem \ref{thm:grTransversalSlice}.
\end{Def}

Let $M(w)=(f_{w'})_{w'}$ denote the functional determined by $f_w=M_w(0)$ and
$L(w)=(f_{w'})_{w'}$ the functional determined by $f_w=L_w(0)$. Then $\{M(w)\}$ and
$\{L(w)\}$ are two bases of $\quotKGp$.

By \cite{VaragnoloVasserot03},
resp. \cite[Section 3.5]{Nakajima09}, for any $w$, $w^1$, $w^2$, \st $w^1+w^2=w$, we have a restriction functor
\begin{align*}
  \tRes_{w^1,w^2}^w:\cD_c(\grAffQuot(w))\ra
\cD_c(\grAffQuot(w^1))\times \cD_c(\grAffQuot(w^2)).
\end{align*}
Furthermore, $\tRes^w_{w^1,w^2}$ sends $\pi_w(v)$ to 
\begin{align*}
\oplus_{v^1+v^2=v}\pi_{w^1}(v^1)\boxtimes  \pi_{w^2}(v^2)[d((v^2,w^2),(v^1,w^1))-d((v^1,w^1),(v^2,w^2))],
\end{align*}
where $(v^1,w^1)$, $(v^2,w^2)$ are not necessarily $l$-dominant.

We define
$\res^w=\sum_{w^1+w^2=w}\tRes^w_{w^1,w^2}[-\eMatrix(w^1,w^2)]$ for
each $w$. Because these functors are compatible with Theorem \ref{thm:grTransversalSlice}, they induce a
multiplication of $\quotKGp$, which we denote by $\otimes$.

The arguments of \cite{VaragnoloVasserot03} imply the following result.
\begin{Thm}\label{thm:positiveSimples}
The structure constants of the multiplication $\otimes$ with respect to the
  basis $\{\can(w)\}$ of $\quotKGp$ are positive:
  \begin{align}
    \label{eq:otimes}
    \can(w^1)\otimes \can(w^2)=\sum_{w^3}\canStr^{w^3}_{w^1,w^2}(t)\can(w^3)
  \end{align}
with $\canStr^{w^3}_{w^1,w^2}(t)\in\N[t^\pm]$.
\end{Thm}

\subsection{qt-characters}
\label{sec:qtCharacters}

Define the ring of formal power series
  \begin{align}
    \redTargSpace=\tBase[[Y_i(a)^\pm]]_{i\in \vtx, a\in \Z},
  \end{align}
  where $t$, $Y_i(a)$ are indeterminates. We denote its product
  by $\cdot$, and often omit this notation.

Given vectors $w$, $v$ as before, we denote the monomial $m=Y^{w-C_q
  v}$ by $m(v,w)$. By \eqref{eq:chopTwistDimension}, we have a naturally
defined bilinear form $\dT$ for such monomials. For $i\in I$, $b\in (\Hf+\Z)$, we sometimes denote
$m(e_{i,b},0)^{-1}$ by $A_{i,b}$.

Endow $\redTargSpace$ with the twisted product $*$ and the bar involution
$\overline{(\ )}$ \st for any two monomials $m^1=m(v^1,w^1)$, $m^2=m(v^2,w^2)$, we have
\begin{align}
\overline{t}=t^{-1},\ \overline{m^1}=m^1,\\
  m^1*m^2=t^{-\dT(m^1,m^2)+\dT(m^2,m^1)}m^1m^2.\label{eq:twistedMultiplication}
\end{align}

The $t$-analogue of the
$q$-character map is defined to be the $\tBase$-linear map $\tChar(\ )$ from
  $\quotKGp$ to $\redTargSpace$ \st we have
  \begin{align}
    \label{eq:redQtChar}
    \tChar(\ )=\sum_v \langle\  ,\pi_w(v)\rangle Y^{w-C_qv}.
  \end{align}

The following results follow from \cite{VaragnoloVasserot03} and \cite[Theorem 3.5]{Nakajima04}.

\begin{Thm}\label{thm:injectiveHom}
$\tChar(\ )$ is an injective algebra homomorphism from $\quotKGp$ to $\redTargSpace$.
\end{Thm}

\section{Monoidal categorification}\label{sec:psedoModules}

Let us fix the following conventions.
\begin{itemize}
\item We always assume $w\in\N^{I\times\{-1,0\}}$ (level $1$
  case). Furthermore, we assume $v\in\N^{I\times\{-\Hf\}}$, which is
  naturally identified with a dimension vector $v\in\N^I$.
\item All the representations are those of the opposite quiver $Q\op$.

\item We assume the ice quiver $\tQ$ is of level $1$ with $z$-pattern.
\end{itemize}

By putting the above restriction of $(v,w)$ on the definition of $\tChar$, we obtain the truncated character map $\tChar\trunc$. Notice
that the truncation preserves all the $l$-dominant pairs. Theorem \ref{thm:injectiveHom} still holds
  for the truncated characters, \cf \cite[Proposition 6.1]{HernandezLeclerc09}.

\subsection{Fourier-Sato-Deligne transform}


For each multiplicity parameter $m=(m_i)_{i\in I}\in\N^I$, define the
injective $Q\op$-representation $I^m=\oplus_{i\in I} I_i^{\oplus
  m_i}$ and similarly the projective representation $P^m=\oplus_{i\in I} P_i^{\oplus
  m_i}$.

Let $(B,\alpha,\beta)$ be any point of $\grRep(v,w)$ and $r$ any path of $Q\op$. Define\footnote{
In \cite{Nakajima09}, our $z_r$ is denoted by $y_r$ if $r$ is
nontrivial and $x_i$ if $r$ is $e_i$ for some vertex $i\in I$.}
the homomorphism $z_r$ to be the composition $\beta_{t(r)} B_r
\alpha_{s(r)}$ along the path $\beta_{t(r)} r\alpha_{s(r)}$ if $r$ is nontrivial and $\beta_{t(r)}
\alpha_{s(r)}$ if $r$ is $e_i$ for some vertex $i\in I$. Notice that
each $z_r$ determines a morphism from $P_{s(r)}^{w_{s(r)}(0)}$ to $I_{t(r)}^{w_{t(r)}(-1)}$. Furthermore, we have the following result.

\begin{Prop}
When $v$ is big
enough, the affine quiver variety $\grMaff(v,w)$ stabilizes to the
affine space $E_w=\Hom(P^{w(0)},I^{w(-1)})$.
\end{Prop}
\begin{proof}
The proof is the same as that of Proposition 4.6(1) {\cite{Nakajima09}}.
\end{proof}

\begin{Eg}
  Figure \ref{fig:affQuot} is an example of the space $E_w$.
\end{Eg}

\begin{figure}
 \centering
\beginpgfgraphicnamed{affQuot}
\begin{tikzpicture}[double distance=2pt, scale=0.7]
\node [color=blue] (deg2) at (9,4) {$\mathrm{deg}=0$};
\node [color=blue] (w1) at (8,-5) {$W_1(0)$};
\node [color=blue] (w2) at (10,-2) {$W_2(0)$};
\node [color=blue] (w3) at (10,2) {$W_3(0)$};

\node [color=blue] (deg3) at (1,4) {$\mathrm{deg}=-1$} ;
\node [color=blue] (w4) at (0,-5) {$W_1(-1)$};
\node [color=blue] (w5) at (2,-2) {$W_2(-1)$};
\node [color=blue] (w6) at (2,2) {$W_3(-1)$};

\draw[-triangle 60] (w1) edge  node[below] {$\beta_1\alpha_1$} (w4);
\draw[-triangle 60] (w2) edge (w5);
\draw[-triangle 60] (w3) edge (w6);

\draw[-latex] (w3) edge[line width=0.5 pt,double] (w4);
\draw[-triangle 60] (w2) edge (w4);
\draw[-triangle 60] (w3) edge (w5);

\end{tikzpicture}.
\endpgfgraphicnamed
\caption{Affine quiver variety $\grAffQuot(v,w)=E_w$}
\label{fig:affQuot}
\end{figure}

\begin{Def}
For each dimension vector $v\in \N^I$, let $\cF(v,w)$ denote the
variety parameterizing all the $I$-graded submodules
$X=(X_i)_{i\in I}$ of $I^{w(-1)}$, \st $\dimv X=v$.
\end{Def}

Since $Q\op$ is
acyclic, the variety $\cF(v,w)$ is smooth projective and irreducible \cite[Theorem 4.10]{Reineke:framed}.

We define the variety $\tcF(v,w)$ by
\begin{align}
  \tcF(v,w)=\{((X,z)\in \cF(v,w)\times E_w\mid\Im z\subset X\}.
\end{align}

\begin{Prop}\label{prop:flag_bundle}
The quasi-projective quiver variety $\grProjQuot(v,w)$ is isomorphic to the
variety $\tcF(v,w)$.
\end{Prop}
\begin{proof}
  The proof goes the same as that of Proposition 4.6 {\cite{Nakajima09}}.
\end{proof}

\begin{Cor}
The quasi-projective quiver variety $\grM(v,w)$ is smooth and irreducible.
\end{Cor}
\begin{proof}
  View $\cF(v,w)\times E_w$ as a trivial vector bundle over
  $\cF(v,w)$. Then $\tcF(v,w)$ is a subbundle. Therefore,
  $\tcF(v,w)$ is smooth. Moreover, since $\cF(v,w)$ is connected, so
  is $\tcF(v,w)$. As a smooth and connected variety, $\tcF(v,w)$ is irreducible.
\end{proof}

Let $E_w^*$ denote the natural dual space of the complex vector space
$E_w$. Let us write $\tcF(v,w)^\bot$ for the annihilator sub-bundle in $\cF(v,w)\times E_w^*$
of the sub-bundle $\tcF(v,w)\subset \cF(v,w)\times E_w$. Its fibre over any given
point $X=(X_i)\in \cF(v,w)$ consists of the linear maps
$z^*=(z_r^*)\in E_w^*$ such that, for any point $(X,z)\in \tcF(v,w)$,
the natural pairing $\langle z,z^*\rangle=\sum_r \Tr(z_r^*z_r)$
vanishes.

Using the Nakayama functor $\nu(\ )$, we obtain 
\begin{align*}
  E_w^*=(\Hom(P^{w(0)},I^{w(-1)}))^*=\Hom(I^{w(-1)},\nu(P^{w(0)}))=\Hom(I^{w(-1)},I^{w(0)}).
\end{align*}

 Notice that each $z_r^*$ determines a morphism from
$I_{t(r)}^{w_{t(r)}(-1)}$ to $I_{s(r)}^{w_{s(r)}(0)}$.

\begin{Eg}
  Figure \ref{fig:dualSpace} is an example of the dual space $E_w^*$.
\end{Eg}

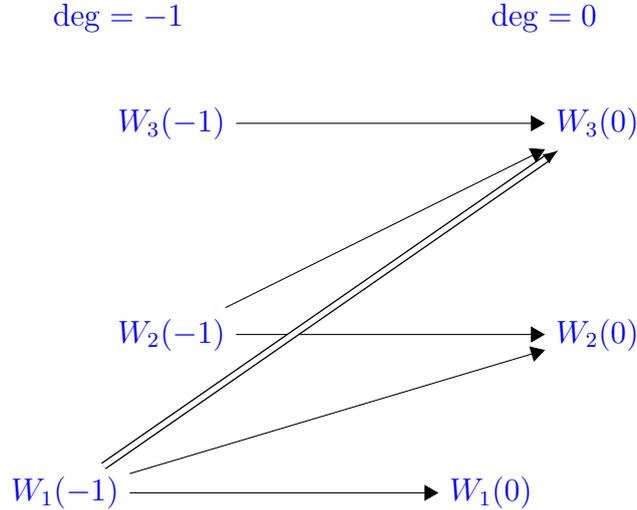
\begin{figure}[htb!]
 \centering
\beginpgfgraphicnamed{dualSpace}
\begin{tikzpicture}[double distance=2pt, scale=0.7]
\node [color=blue] (deg2) at (9,4) {$\mathrm{deg}=0$};
\node [color=blue] (w1) at (8,-5) {$W_1(0)$};
\node [color=blue] (w2) at (10,-2) {$W_2(0)$};
\node [color=blue] (w3) at (10,2) {$W_3(0)$};

\node [color=blue] (deg3) at (1,4) {$\mathrm{deg}=-1$} ;
\node [color=blue] (w4) at (0,-5) {$W_1(-1)$};
\node [color=blue] (w5) at (2,-2) {$W_2(-1)$};
\node [color=blue] (w6) at (2,2) {$W_3(-1)$};

\draw[-triangle 60] (w4) edge (w1);
\draw[-triangle 60] (w5) edge (w2);
\draw[-triangle 60] (w6) edge (w3);

\draw[-latex] (w4) edge[line width=0.5 pt,double] (w3);
\draw[-triangle 60] (w4) edge (w2);
\draw[-triangle 60] (w5) edge (w3);

\end{tikzpicture}.
\endpgfgraphicnamed
\caption{Dual space $E_w^*$}
\label{fig:dualSpace}
\end{figure}

\begin{Lem}\label{lem:dual_fibre}
The fibre of $\tcF(v,w)^\bot$ over any given point $X\in\cF(v,w)$
consists of the maps $z^*=(z_r^*)\in \Hom(I^{w(-1)},I^{w(0)})$ \st
\begin{align}
z^*X=0
\end{align}
\end{Lem}
\begin{proof}
The composition $z^* z\in \Hom(P^{w(0)},I^{w(0)})$ is a direct
sum of components $(z^*z)_{ij}$, where $i,j\in I$ and
$(z^*z)_{ij}\in\Hom(P_j^{w_j(0)},I_i^{w_i(0)})$. We can write the
pairing $\langle z,z^*\rangle$ as $\sum_{i\in I} \Tr(z^*z)_{ii}$.

Denote the quotient module $I^{w(-1)}/X$ by $Y$. We can choose decompositions of vector spaces
\begin{align*}
  \Hom(P^{w(0)},I^{w(-1)})\iso\Hom(P^{w(0)},X)\oplus \Hom(P^{w(0)},Y),\\
\Hom(I^{w(-1)},I^{w(0)})\iso\Hom(X,I^{w(0)})\oplus
  \Hom(Y,I^{w(0)})
\end{align*}
and write $z=(z_X,z_Y)$
and $z^*=(z_X^*,z_Y^*)$ correspondingly. Then
$(z^*z)_{ii}$ equals $(z_X^*z_X+z_Y^*z_Y)_{ii}$, $\forall i\in I$. Notice that the natural pairing $\langle
z_X,z_X^*\rangle=\sum_i\Tr(z_X^*z_X)_{ii}$ between
$\Hom(P^{w(0)},X)$ and $\Hom(X,I^{w(0)})$ is non-degenerate.

The fibre of $\tcF(v,w)$ over a given point $X$ consists of the pairs
$(X,z)$ \st $z_Y=0$. Therefore, the fibre of $\tcF(v,w)^\bot$ over $X$
consists of the pairs $(X,z^*)$ \st $z_X^*=0$. In other words, $z^*X$ vanishes.

\end{proof}

Fix any element $z^*\in
\Hom(I^{w(-1)},I^{w(0)})$. We define $\kerMod W=(\kerMod W_i)_{i\in I}$ to be
the kernel of $z^*$ and denote its dimension vector by $\kerMod
w$. Then $(\kerMod W, z^*)$ is contained in $\tcF^\bot(^\sigma
w,w)$. For any $1\leq i\leq n$, $X_i$ is a subspace of $^\sigma W_i$.

\begin{Def}\label{def:pureCoeff}
  For any vector $w\in\N^{I\times\{-1,0\}}$, its
  \emph{coefficient-free part} $\coeffFree w\in\N^{I\times \{-1,0\}}$
  and its \emph{pure coefficient part} $\pureCoeff w\in
  span_{\N}\{e_{i,-1}+e_{i,0}|i\in I\}$ are defined \st $w=\coeffFree
  w+\pureCoeff w$ and, for any $i\in I$, either $\coeffFree w_i(-1)$
  or $\coeffFree w_i(0)$ vanishes. Let $\redWSet$ denote the set of
  $w$ \st $w=\coeffFree w$.
\end{Def}

\begin{Prop} 

1) The fibre of $\tcF^\bot(v,w)$ over $z^*$ is isomorphic to the
submodule Grassmannian consisting of the $v$-dimensional submodules of $\kerMod W$.

2) When $z^*$ is generic, the module $\kerMod W$ is a generic representation of $Q\op$ with dimension $^\sigma w$.
\end{Prop}
\begin{proof}
1) Any element $(X,z^*)$ in the fibre must
satisfy $X\subset \kerMod W$. Conversely, given any submodule $\oplus
V_i$ of $^\sigma W$, the collection $(V_i)_i$ is contained in $\tcF^\bot(v,w)$ by the definition of $\kerMod W$ and Lemma \ref{lem:dual_fibre}.

2) Since we are interested in generic maps, we can replace $w$ by
$\coeffFree w$ without changing the
generic kernels. A generic kernel in this case is known to be a
generic module, \cf \cite{Plamondon10c}.
\footnote{More precisely, it follows from \cite{Plamondon10c} that there is a bijection
between the generic objects of the cluster category associated with our quiver $Q$ and
the reduced $w$ above.}
\end{proof}

 \begin{Rem}
    If the quiver $Q\op$ is bipartite, let $I_0$ denote the sink points
    and $I_1$ denote the source points. Assume that $w$ is contained
    in $\redWSet$. In \cite{Nakajima09}, Nakajima defined the dimensions $^\sigma
    W_i(q^a)_{i\in I, a\in\Z}$. Then we have $\kerMod w_i=^\sigma W_i(1)$ for $i\in
    I_0$ and $\kerMod w_j=^\sigma W_j(q^3)$ for $j\in I_1$.
  \end{Rem}

\subsection{Generic characters}

Let the Fourier-Sato-Deligne transform from $\pi:\tcF(v,w)\ra E_w$
to $\pi^\bot:\tcF^\bot(v,w)\ra E_w^*$ be denoted by $\Psi$. Define the set
\begin{align*}
  \mathscr L_w=\{IC_w(v)\in \mathscr P_w| \supp\Psi IC_w(v)=E_w^*\}.
\end{align*}

Recall that $\Psi (IC_w(v))$ is a perverse
sheaf over $E_w^*$.
\begin{Def}[Twisted rank]\label{def:twistedRank}
  For any $w,w'\in\N^{I\times\{-1,0\}}$, we define the integer $r_{ww'}$ to be
  \begin{align}
    r_{ww'}=r(v,w)=(-1)^{\dim_\C \grAffQuot(v,w)}\rank \Psi (IC_w(v))
  \end{align}
  if we have $w'=w-C_q v$ for some $v\in\N^{I\times \{\Hf\}}$ \st $\supp \Psi (IC_w(v))=E_w^*$ and zero
  otherwise.
\end{Def}

Notice that $IC(0,w)$ is always contained in
$\mathscr L_w$. It follows that $r_{w,w}=1$. Furthermore, for each fixed $w$,
only finitely many of the $r_{w,w'}$ are nonzero. The matrix
$(r_{w,w'})$ is upper unitriangular with respect to the dominance order.
\begin{Rem}
  Our definition of $r_{w,w'}$ is the twisted version of that of \cite[6.1]{Nakajima09}.
\end{Rem}

We define the \emph{almost simple pseudo-module} $\gen(w)$ to be the
element in the Grothendieck group $\quotKGp$ \st we have
\begin{align}
  \gen(w)=\sum_{w'}r_{w,w'}\can(w').
\end{align}

Denote the truncated $qt$-characters $\tChar\trunc(\gen(w))$ by $\genRedTarg(w)$.

\begin{Thm}\label{thm:genChar} \label{thm:z_coeff}
  The truncated $qt$-characters of the almost simple pseudo-modules in
  $\redTargSpace$ are given by
  \begin{align}\label{eq:genChar}
    \genRedTarg(w)=\sum_{v}t^{-\dim (\Gr_v \kerMod W)}P_t(\Gr_v \kerMod
    W)Y^{w-C_q v}.
  \end{align}
\end{Thm}
\begin{proof}

    By the arguments of the proof of \cite[Theorem 6.3]{Nakajima09}, for
    \begin{align*}
      \gen(w)'=\sum_{v:IC_w(v)\in\mathscr L_w}\rank\Psi IC_w(v)L_w(v),
    \end{align*}
we have
    \begin{align*}
      \tChar\trunc(\gen(w)')&=\sum_{v'}t^{-\dim(\Gr_{v'} \kerMod W)}\sum_k t^k\dim H^k(\Gr_{v'}
      \kerMod W)Y^{w-C_q v'}\\
&=\sum_{v'}\sum_{v:IC_w(v)\in\mathscr L_w}\rank\Psi
      IC_w(v)a_{v',v;w}(t)Y^{w-C_q v'}.
    \end{align*}
    By Theorem \ref{thm:grTransversalSlice} and \ref{thm:oddVanish}, for any point
    $x_{v,w}\in\grRegStratum(v,w)\subset\grAffQuot(v',w)$, the odd homology of
    $\grFib_{x_{v,w}}(v',w)$ vanishes. Therefore the contribution of
    $a_{v,v'}(t)IC_w(v)$ to the odd homology is zero, \ie
    $a_{v,v'}(t)t^{\dim\flagVar(v',w)-\dim\grAffQuot(v',w)}$ is
    contained in $\N[t^{\pm
      2}]$. Also, $\dim\flagVar(v',w)+\dim\Gr_{v'}\kerMod W=\dim\flagVar(v',w)+\dim\dualFlagVar(v',w)-\dim
    E_w^*$ is divisible by $2$. Therefore, by twisting the sign of
    the term $L(w-C_q v)$ appearing in $\gen(w)'$ by $(-1)^{\dim\grAffQuot(v,w)}$, we obtain the alternating sums of Betti
    numbers of the fibre $\Gr_{v'} \kerMod W$.

\end{proof}

Denote the possibly non-reductive group $\Aut(I^{w(-1)})\times \Aut(I^{w(0)})$ by
$\Aut(w)$. It acts on the affine space
$E_w^*=\Hom(I^{w(-1)},I^{w(0)})$. Let $p$ be any element in $E_w^*$. By \cite{DerksenFei09}, the space $E(p,p)=\Hom_{K^b(\C Q\op)}
  (p,p[1])$ describes the normal space of the orbit $\Aut(w)p$ in
  $E_w^*$, \cf also \cite[Lemma 5.4.6]{Plamondon10c}. Here $K^b(\C
  Q\op)$ is the bounded homotopy category of injective complexes of
  $Q\op$-representations.

\begin{Lem}[{\cite[Lemma 5.3.6]{Plamondon10c}}]\label{lem:minimalResolution}
    If $p$ is a minimal injective resolution of $\Ker p$, then we have
    \begin{align*}
      E(p,p)=\Hom(\ker p, \tau' \ker p),
    \end{align*}
    where $\tau'$ is the Auslander-Reiten-translation of the category
    of $\C Q\op$-modules.
\end{Lem}

\begin{Lem}\label{lem:rigidKernel}
Let $p$ be generic. If the kernel $\Ker p$ is a rigid module, then the
codimension of the orbit $\Aut(w)p$ is zero.
\end{Lem}
\begin{proof} 
This lemma is the application of Lemma 5.4.4 \cite{Plamondon10c} to
generic maps with rigid kernels.
\end{proof}

\begin{Prop}\label{prop:clusterMonomial}
1) If $M$ is a rigid module with a minimal injective resolution
$z^*\in E_w$, then $\gen(w)=\can(w)$.

2) If a generic map $z^*$ in $E_w^*$ has a rigid kernel, then
$\gen(w)=\can(w)$.
\end{Prop}
\begin{proof}
Observe that the fibre map $\pi^\bot:\tcF^\bot(v,w)\ra E_w^*$ is
$\Aut(w)$-equivariant. Since the $\Aut(w)$-stabilizer in $E_{w}^{*}$ is connected, if there is an open dense orbit in
the affine space $E_w^*$, then $IC(0,w)$ is the only element in the set $\mathscr
L_w$. The proposition follows from Lemma \ref{lem:minimalResolution}
and Lemma \ref{lem:rigidKernel}.
\end{proof}

\begin{Prop}
In the notation of Definition \ref{def:pureCoeff}, we have a factorization
  \begin{align}\label{eq:coeffFactorization}
    \genRedTarg(w)=\genRedTarg(\coeffFree w)\cdot\genRedTarg(\pureCoeff w).
  \end{align}
\end{Prop}
\begin{proof}
This proposition is a direct consequence of Theorem \ref{thm:genChar}.
\end{proof}

\subsection{From deformed Grothendieck rings to quantum cluster algebras}
\label{sec:correspondence}

In this section, we construct $\Z$-linear algebra homomorphisms from the ring of
formal power series $\redTargSpace$ in which (truncated) $qt$-characters live to the quantum
torus $\torus$ in which quantum cluster algebras
live. Detailed proofs of these maps' properties can be
found in \cite{Qin11}, where the ice quiver is not restricted to the $z$-pattern, and the maps might not be algebra homomorphisms.

As in Section \ref{sec:frozenQuiver}, let $\tQ$ be an ice quiver
whose principal part $Q$ is acyclic. Using the notation of Section
\ref{sec:cluster_category}, we define the linear map $\ind(\ )$ from
the set of finitely supported vectors in $\N^{I\times\Z}$ to $\Z^m$
\st $\ind(e_{i,a})$, $i\in I$, $a\in \Z$, is the vector of coordinates
in the basis $[e_i\Gamma]$, $1\leq i\leq m$, of the index of the
coefficient-free object whose image in the cluster category $\cC_Q$ is
$T_i[-a]$.

\begin{Lem}[{\cite{Qin11}}]\label{lem:zPatternInd}
  We have, for $1\leq k\leq n$,
  \begin{align}
    \ind(e_{k,0})=e_k,\\
\ind(e_{k,-1})=e_{k+n}-e_k.
  \end{align}
\end{Lem}

\begin{Lem}[{\cite{Qin11}}]\label{lem:zReduction}
We have
\begin{align}
  \ind(w-C_qv)=\ind(w)+\tB v.
\end{align}
\end{Lem}

Define $N=2n$.  Notice that
$\Lambda(e_{i},\tB v)=0$, for any $n+1\leq i\leq N$ and
$v\in\N^n$.

Recall that the associated quantum torus is the Laurent polynomial ring
  \begin{align}
    \label{eq:dTorus}
    \torus=\qBase[x_1^\pm,\ldots,x_N^\pm],
  \end{align}
together with the twisted product $*$ \st for any $g^1$, $g^2\in
\Z^N$, we have
\begin{align*}
  x^{g^1}* x^{g^2}=q^{\Hf \Lambda(g^1,g^2)}x^{g^1+g^2}.
\end{align*}
It has the bar involution $\overline{(\ )}$ given by $\overline{(q^{\Hf}x^g)}=q^{-\Hf}x^g$.

Define the coefficient ring to be
\begin{align*}
  \qBaseCoeff=\qBase[x_{n+1}^\pm,\ldots,x_{N}^\pm].
\end{align*}
Let $\ZCoeff$ denote its semi-classical limit
under the specialization $q^\Hf\mapsto 1$.

\begin{Def}[Correspondence map]\label{def:correction_map}
The $\Z$-linear map $\cor$ from $\redTargSpace$ to $\torus$ is given by
  \begin{align}
    \label{eq:cor}
    \cor(t^{\lambda}Y^w)=q^{\frac{\diag}{2}\lambda}x^{\ind(w)},
  \end{align}
for any $w$, and integer $\lambda$.
\end{Def}

\begin{Rem}
  The map $\cor$ is denoted by the composition $\contr\cor$ in
  \cite{Qin11}.
\end{Rem}

\begin{Lem}[{\cite{Qin11}}] We have, for any $w^i$, $i=1$, $2$,
  \begin{align}\label{eq:failureCor}
      \cor(Y^{w^1}*Y^{w^2})&=\cor(Y^{w^1})*\cor(Y^{w^2}).
  \end{align}
\end{Lem}

Let us define
\begin{align}
\pbwTorus(w)=\cor\tChar\trunc(\pbw(w)),\\
\canTorus(w)=\cor\tChar\trunc(\can(w)).
\end{align}
Explicitly, we have
  \begin{align*}
    \pbwTorus(w)&=\sum_vP_{q^{\frac{\diag}{2}}}(\lag(v,w))q^{-\frac{\diag}{2}
      \dim\grProjQuot(v,w)}x^{\ind(w)}x^{\tB v}\\
&=\sum_v\cor(\langle M_w(0),\pi_w(v)\rangle) x^{\ind(w)}x^{\tB v}.
  \end{align*}
  
It follows from definition that the truncated $qt$-characters of the simple modules are given by
  \begin{align}
    \label{eq:qtChar}
    \tChar\trunc(\can(w))=\sum_v a_{v,0;w}(t)Y^{w-C_qv}.
  \end{align}
Since $a_{v,0;w}(t)$ equals $a_{v,0;w}(t^{-1})$, $\tChar(\ )$ commutes with
$\overline{(\ )}$.

\begin{Prop}[{\cite{Qin11}}]
  \label{prop:mult_pbw}
Fix $w^1$ and $w^2$. If for all $i,j\in I$ and
$a>b\in\Z$, either $(w^1)_i(a)$ or $(w^2)_j(b)$ vanishes, then the multiplicative property holds:
\begin{align*}
\pbwTorus(w^2)*\pbwTorus(w^1)&= q^{\Hf\eMatrix(w^1,w^2)} \pbwTorus(w^1+w^2).
\end{align*}
\end{Prop}

\begin{Thm}[Deformed monoidal categorification]\label{thm:iso}
  The map $\tChar\trunc$ is an algebra isomorphism from $\quotKGp$ to
  $\qClAlg$. Furthermore, the preimage of any cluster monomial is the class of a
  simple module.
\end{Thm}
\begin{proof}
Let $\subQClAlg$ denote the vector space spanned by the standard
basis elements $\pbwTorus(w)$ over $\qBaseCoeff$. By Theorem
\ref{thm:injectiveHom}, it is the image of the injective algebra
homomorphism $\tChar\trunc$. In particular, it is closed under the
involution $\overline{(\ )}$ and the twisted products (\cf \cite{Qin11} for another proof).

By Theorem \ref{thm:genChar} and Theorem \ref{thm:CC_formula}, $\subQClAlg$ contains all the quantum
cluster variables and the frozen variables
$x_{n+1},\cdots,x_{2n}$. Therefore it is equal to $\qClAlg$.

The second statement follows from Theorem \ref{thm:genChar}, Theorem \ref{thm:CC_formula}, and
Proposition \ref{prop:clusterMonomial}.
\end{proof}

\begin{Def}[Strong positivity]
  A cluster algebra is called \emph{strongly positive}, if it has a basis
  such that the structure constants of the basis are positive and all
  the cluster monomials are contained in the basis.
\end{Def}

\begin{Cor}\label{cor:stronglyPositive}
  The quantum cluster algebra $\qClAlg$ is strongly positive.
\end{Cor}
\begin{proof}
The basis $\redCanBasis$ has positive structure constants by Theorem
\ref{thm:positiveSimples} and Theorem \ref{thm:injectiveHom}. It contains
all the cluster monomials by Proposition \ref{prop:clusterMonomial}.
\end{proof}

\begin{Cor}(Quantum positivity)\label{cor:positivity}
Any quantum cluster monomial $m$ can be written as a Laurent
polynomial of the quantum cluster variables $x_i$, $1\leq i\leq n$, in
any given seed with coefficients in $\N[q^{\pm\frac{1}{2}},x_{n+1}^\pm,\ldots, x_{m}^\pm]$.
\end{Cor}
\begin{proof}
By the quantum Laurent phenomenon, we have
\begin{align*}
  m=\frac{\sum_{m_*=(m_i)} c_{m_*}\prod_{1\leq i\leq n}
    x_i^{m_i}}{\prod_i x_i^{d_i}},
\end{align*}
where $m_*=(m_i)_{i\in I}$, $d_*=(d_i)_{i\in I}$ are sequences of nonnegative integers and the
coefficients $c_{m_*}$ are contained in $\qBaseCoeff$. Notice that we
use the usual product $\cdot$ in this expression.

The quantum cluster monomial $m$ equals $\canCl(w)$ for some
$w$. Also, the
quantum $X$-variable $x_i$, $1\leq i\leq m$, equals $\canCl(w_i)$ for
some $w_i$. We can rewrite the above equation as
\begin{align*}
\sum_{m_*=(m_i)} c_{m_*}\canCl(\sum_i m_iw_i)&=\prod_i\canCl(w_i)^{d_i} \cdot \canCl(w)\\
&=q^{-\Hf\Lambda(\ind(\sum_i d_iw^i),\ind(w))}\canCl(\sum_i d_iw_i) * \canCl(w).
\end{align*}

The statement follows from Theorem \ref{thm:iso} and \eqref{eq:otimes}.
\end{proof}






\section{A reminder on quantum unipotent subgroups}\label{sec:unipotentSubgroup}

In this section, we recall the definitions and some properties of quantum groups, the dual canonical basis and quantum unipotent
subgroups following \cite{Kimura10}.
\subsection{Quantum groups}
A \emph{root datum} is a collection $(\mfr{h}, I, P, P^{\vee}, \{\alpha_{i}\}_{i\in I},
\{h_{i}\}_{i\in I}, (~, ~))$, where
\begin{enumerate}
\item $\mfr{h}$ is a finite-dimensional $\mbb{Q}$-vector space;
\item $I$ is a finite index set;
\item $P\subset \mfr{h}^*$ is a lattice (weight lattice);
\item $P^\vee=\Hom_\mbb{Z}(P,\mbb{Z})$ is the dual of $P$ with respect
  to the natural pairing $\bracket{~, ~}\colon P^{\vee}\otimes P\to \mbb{Z}$;
\item $\alpha_i$, $i\in I$, belongs to $P$ (simple root);
\item $h_i$, $i\in I$, belongs to $P^\vee$ (simple coroot);
\item $(\  ,\ )$ is a $\mbb{Q}$-valued symmetric bilinear form on $\mfr{h}^*$,
\end{enumerate}
\st we have
\begin{aenumerate}
\item $\bracket{h_i,\lambda}=2(\alpha_i,\lambda)/(\alpha_i,\alpha_i)$ for $i\in I$ and $\lambda\in P$;
\item the generalized Cartan matrix $C$, whose entry in position
  $(i,j)$ is defined as
  \begin{align*}
    a_{ij}=\bracket{h_i,\alpha_j}=2(\alpha_i,\alpha_j)/(\alpha_i,\alpha_i),
  \end{align*}
is symmetrizable;

\item for each $i\in I$, the value $(\alpha_i,\alpha_i)/2$ is
  contained in $\mbb{Z}_{>0}$, which we denote by
  $d_i$;
\item $\set{\alpha_i}_{i\in I}$ is linearly independent.
\end{aenumerate}
The collection $(I, \mfr{h}, (~, ~))$ is called a \emph{Cartan datum}.


 
Define the \emph{root lattice} $Q$ to be the sub-lattice
$\bigoplus_{i\in I}{\mbb{Z}}\alpha_i$ of $P$. Let $Q_{\pm}$ denote
$\pm \sum_{i\in I}\mathbb{Z}_{\geq 0}\alpha_i$. For $\xi=\sum_{i\in
  I}\xi_i\alpha_i\in Q_{\pm}$, where $\xi_i\in\Z$, we define $\mathrm{tr}(\xi)=\sum_{i\in I}\xi_i$.
And we assume that there exists $\varpi_i\in P$  such that $\bracket{h_i, \varpi_{j}}=\delta_{i, j}$ for any $i,j \in I$.
We call $\varpi_{i}$ the \emph{fundamental weight} corresponding to $i\in I$.
We say $\lambda\in P$ is \emph{dominant} if $\bracket{h_i, \lambda}\geq 0$ for any $i\in I$ and denote by $P_+$
the set of dominant integral weights.
Define $\overline{P}=\bigoplus_{i\in I}\mbb{Z}\varpi_{i}$ and
$\overline{P}_+=\overline{P}\cap P_+=\bigoplus_{i\in I}\mbb{Z}_{\geq
  0}\varpi_i$.


We assume that the root datum is always symmetric, \ie the matrix $C$ is
symmetric. Then, for all $i\in I$, we have $d_i=d$ for some $d\in\Z_{>0}$. We introduce an indeterminate $v$.
For $i\in I$, we set $v_i=v^{(\alpha_i,\alpha_i)/2d}=v$ for all $i\in I$.
For $\xi=\sum_{i\in I}\xi_i\alpha_i \in Q$, we define $v_\xi=\prod_{i\in
  I}(v_i)^{\xi_i}=v^{(\xi, \rho)/d}=v^{\tr(\xi)}$, where $\rho$ is the sum of all the
fundamental weights.


Let $\mfr{g}$ be the corresponding Kac-Moody Lie algebra. Let
$\Uv(\mfr{g})$ be the corresponding quantum enveloping algebra which
is the $\mbb{Q}(v)$-algebra generated by $\{e_{i}, f_{i}\}_{i\in
  I}\cup \{v^{h}\}_{h\in P^{\vee}}$ with the following relations:
\begin{renumerate}
\item $v^0=1,v^{h}v^{h'}=v^{h+h'}$,
\item $v^he_{i}v^{-h}=v^{\bracket{h,\alpha_i}}e_i, v^hf_{i}v^{-h}=v^{-\bracket{h,\alpha_i}}f_i$,
\item $e_if_j-f_je_i=\delta_{ij}{(t_i-t_i^{-1})}/{(v_i-v_i^{-1})}$,
\item $\displaystyle \sum_{k=0}^{1-a_{ij}}(-1)^ke_i^{(k)}e_je_i^{(1-a_{ij}-k)}=%
\sum_{k=0}^{1-a_{ij}}(-1)^kf_i^{(k)}f_jf_i^{(1-a_{ij}-k)}=0$,
\end{renumerate}
where 
$t_i=v^{d_{i}h_i}$,
$[n]_i=(v_i^n-v_i^{-n})/(v_i-v_i^{-1})$,
$[n]_i!=[n]_i[n-1]_i\cdots [1]_i$ for $n>0$ and
$[0]!=1$,
$e_i^{(k)}=e_i^k/[k]_i!, f_i^{(k)}=f_i^k/[k]_i!$
for $i\in I$ and $k\in {\mbb{Z}}_{\geq 0}$.


Let $\Uv^+(\mfr{g})$ (resp.\ $\Uv^-(\mfr{g})$) denote the $\Qv$-subalgebra of $\Uv(\mfr{g})$
generated by $e_i$ (resp.\ $f_i$) for $i\in I$.
Then we have the triangular decomposition
\[\Uv(\mfr{g})\simeq \Uv^-(\mfr{g})\otimes_{\Qv}\Qv[P^\vee]\otimes_{\Qv}\Uv^+(\mfr{g}),\]
where $\Qv[P^\vee]$ is the group algebra over $\Qv$, i.e., $\bigoplus_{h\in P^\vee}\Qv v^h$.
For $\xi=\sum \xi_{i}\alpha_{i} \in Q$, we set $t_{\xi}=v^{\sum_{i\in I}d_{i}\xi_{i}h_{i}}$.
We have $t_{\alpha_{i}}=t_{i}$.
We set $\Uv(\mfr{g})_{\xi}:=\{x\in \Uv(\mfr{g})\mid t_{i}x t_{i}^{-1}=v^{\langle h_{i},\xi\rangle}x\; \text{~for all~}i\in I\}$.
We have the following root space decomposition:
\[\Uv^{\pm}(\mfr{g})=\bigoplus_{\xi\in Q_{\pm}}\Uv^{\pm}(\mfr{g})_{\xi}.\]
\begin{NB}
We also have the natural decompositions of the unipotent quantum groups into homogeneous components
$\Uv^\pm(\mathfrak{g})=\oplus_{\xi\in\Phi_{\pm}}\Uv^\pm(\mathfrak{g})_{\xi}$,
where $\Phi_+$ (resp. $\Phi_-$) denote the set of all the positive roots
(resp. negative roots).
\end{NB}

\subsubsection*{Automorphisms of $\mathbf{U}_{v}(\mathfrak{g})$}
Let $\overline{\phantom{x}}$ denote the $\mathbb{Q}$-algebra involution $\overline{\phantom{x}}\colon \mathbf{U}_{v}(\mathfrak{g})\to \mathbf{U}_{v}(\mathfrak{g})$ given by
\begin{align*}
\overline{e_{i}}=e_{i}, && \overline{f_{i}}=f_{i}, && \overline{v}=v^{-1},&& \overline{v^{h}}=v^{-h}.
\end{align*}
Let $*$ denote the $\mathbb{Q}(v)$-algebra anti-involution $*\colon \mathbf{U}_{v}(\mathfrak{g})\to \mathbf{U}_{v}(\mathfrak{g})$ given by
\begin{align*}
*(e_{i})=e_{i}, && *(f_{i})=f_{i},&& *(v^{h})=v^{-h}.
\end{align*}
Let $\vee$ be the $\mathbb{Q}(v)$-algebra involution $\vee\colon \mathbf{U}_{v}(\mathfrak{g})\to \mathbf{U}_{v}(\mathfrak{g})$ given by 
\begin{align*}
\vee(e_{i})=f_{i},&& \vee(f_{i})=e_{i},&& \vee(v^{h})=v^{-h}.
\end{align*}
Let $\varphi$ be the composite $\vee\circ *$. Then $\varphi$ is a $\mathbb{Q}(v)$-linear anti-involution
satisfying 
\begin{align*}
\varphi(e_{i})=f_{i}, && \varphi(f_{i})=e_{i}, &&\varphi(v^{h})=v^{h}.
\end{align*}
Let $\Omega$ be the composition $\overline{\phantom{x}}\circ \vee\circ
*$. Then $\Omega$ is a $\mathbb{Q}$-linear anti-involution satisfying 
\begin{align*}
\Omega(e_{i})=f_{i}, && \Omega(f_{i})=e_{i}, &&\Omega(v^{h})=v^{-h}, && \Omega(v)=v^{-1}.
\end{align*}
This anti-involution is called by $\overline{\varphi}$ in \cite[Section 6.4]{GeissLeclercSchroeer11}.

\subsubsection*{Coproducts and twisted
  coproducts}\label{sec:coproduct}
We have two coproducts $\Delta_{\pm}$ on $\Uv(\mfr{g})$ (\cf \cite[Section 1.4]{Kas:crystal}):
\begin{subequations}
\begin{align}
	\Delta_{+}(v^h)&=v^h\otimes v^h, \\
	\Delta_{+}(e_i)&=e_i\otimes 1+t_{i}\otimes e_i, \\
	\Delta_{+}(f_i)&=f_i\otimes t_{i}^{-1}+1\otimes f_i;
\end{align}
\end{subequations}
\begin{subequations}
\begin{align}
	\Delta_{-}(v^h)&=v^h\otimes v^h, \\
	\Delta_{-}(e_i)&=e_i\otimes t_i^{-1}+1\otimes e_i, \\
	\Delta_{-}(f_i)&=f_i\otimes 1+t_i\otimes f_i.
\end{align}
\end{subequations}

Define the $\Qv$-algebra structure on
$\Uv^\pm (\mfr{g})\otimes\Uv^\pm(\mfr{g})$ \st we have
\[(x_1\otimes y_1)(x_2\otimes y_2)=v^{\pm(\wt(x_2), \wt(y_1))}x_1x_2\otimes y_1y_2,\]
for any homogeneous elements $x_i, y_i~(i=1, 2)$. Let 
$\rpm\colon \Uv^{\pm}(\mfr{g})\to \Uv^{\pm}(\mfr{g})\otimes
\Uv^{\pm}(\mfr{g})$ be the $\Qv$-algebra homomorphisms such that we
have, for any $i\in I$,
\begin{align*}
  r_{+}(e_i)=e_i\otimes 1+1\otimes e_i,\\
r_{-}(f_{i})=f_{i}\otimes 1+1\otimes f_{i}.
\end{align*}
They are called the \emph{twisted coproducts}.
The relations between the coproducts $\Delta_{\pm}$ and the twisted
coproducts $\rpm$ are given by the following Lemma.

\begin{Lem}
For any homogeneous element $x\in\mathbf{U}_{v}^{\pm}(\mathfrak{g})_{\xi}$, we have 
\[
\Delta_{\pm}(x)=\sum x_{(1)}t_{\pm\wt(x_{(2)})}\otimes x_{(2)},
\]
where $\rpm(x)=\sum x_{(1)}\otimes x_{(2)}$.
\end{Lem}

We have the following relation between the twisted coproducts.
\begin{Lem}
We have 
\[\rmp\circ\Omega=\mathrm{flip}\circ(\Omega\otimes\Omega)\circ \rpm,\]
where $\mathrm{flip}(x\otimes y)=y\otimes x$ for any $x,y\in \Uv^\pm(\mathfrak{g})$.
\end{Lem}
\begin{proof}
We prove the claim by induction. For any homogeneous element $x=x'x''$ such that the claim holds for $x'$ and $x''$, that is $\rmp(\Omega(x'))=\sum\Omega(x'_{(2)})\otimes\Omega(x'_{(1)})$
and $\rmp(\Omega(x''))=\sum\Omega(x''_{(2)})\otimes\Omega(x''_{(1)})$,
where $\rpm(x')=\sum x'_{(1)}\otimes x'_{(2)}$, $\rpm(x'')=\sum
x''_{(1)}\otimes x''_{(2)}$, we want to check the claim for $x$. Note that $\rpm(x)=\sum\sum v^{\pm(\wt x''_{(1)},\wt x'_{(2)})}x'_{(1)}x''_{(1)}\otimes x'_{(2)}x''_{(2)}$.
Therefore, we have 
\begin{align*}
\rmp(\Omega(x)) & =\rmp(\Omega(x'x''))\\
 & =\rmp(\Omega(x''))\rmp(\Omega(x'))\\
 & =\sum\sum\Omega(x''_{(2)})\otimes\Omega(x''_{(1)})\cdot\Omega(x'_{(2)})\otimes\Omega(x'_{(1)})\\
 & =\sum\sum v^{\mp(\wt x''_{(1)},\wt x'_{(2)})}\Omega(x''_{(2)})\Omega(x'_{(2)})\otimes\Omega(x''_{(1)})\Omega(x'_{(1)})\\
 & =(\Omega\otimes\Omega)\left(\sum\sum v^{\pm(\wt x''_{(1)},\wt x'_{(2)})}(x'_{(2)}x''_{(2)})\otimes(x'_{(1)}x''_{(1)})\right)\\
 & =\mathrm{flip}\circ(\Omega\otimes\Omega)\circ \rpm(x).
\end{align*}
Hence the assertion holds.
\end{proof}

\subsubsection*{Bilinear forms}

For $i\in I$, we define the unique $\Qv$-linear map  ${_ir}\colon \Uv^{\pm}(\mfr{g})\to \Uv^{\pm}(\mfr{g})$ (resp.\ $r_i\colon \Uv^{\pm}(\mfr{g})\to \Uv^{\pm}(\mfr{g})$) given by
${_ir}(1)=0, {_ir}(x^{\pm}_j)=\delta_{i, j}$ (resp.\ $r_i(1)=0,
r_i(x^{\pm}_j)=\delta_{i, j}$) for all $i, j\in I$ ($x$ is $e$ or $f$)
and 
\begin{subequations}
\begin{align}
	{_ir}(xy)&={_ir}(x)y+v^{(\wt x, \alpha_i)}{x}~{_ir(y)}, \label{eq:ir}\\
	r_i(xy)&=v^{(\wt y, \alpha_i)}r_i(x)y+xr_i(y)\label{eq:ri}
\end{align}
\end{subequations}
for homogeneous $x, y\in \Uv^-(\mfr{g})$.


By Kashiwara \cite[\S 3.4]{Kas:crystal}, there exist unique symmetric
non-degenerate bilinear forms $(\;,\;)_{\pm}$ on
$\mathbf{U}_{v}^{\pm}(\mathfrak{g})$ such that we have
\begin{align*}
(x_{i}^{\pm},x_{j}^{\pm})_{\pm} & =\delta_{i,j}\;(x=e\;\text{or}\; f)\\
(1,1) & =1\\
(\rpm(x),y\otimes z)_{\pm} & =(x,yz)\;\text{for}\; x,y,z\in\mathbf{U}_{v}^{\pm}(\mathfrak{g}).
\end{align*}

Define the dual bar-involutions $\sigma_{\pm}$ on $\mathbf{U}_{v}^{\pm}(\mathfrak{g})$
by 
\[
(\sigma_{\pm}(x),y)_{\pm}=\overline{(x,\overline{y})_{\pm}}\;\text{for arbitrary \ensuremath{x,y\in\mathbf{U}_{v}^{\pm}(\mathfrak{g})}}.
\]
We often denote $\sigma_\pm$ by $\sigma$ for simplicity.

We have the following compatibility properties between Kashiwara's bilinear form $(\;,\;)_{\pm}$ 
and the anti-involution $\Omega$.

\begin{Lem}[{\cite[Lemma 6.1(b)]{GeissLeclercSchroeer11}}]
For $x,y\in\mathbf{U}_{v}^{\pm}(\mathfrak{g})$, we have 
\[
\overline{(x,y)_{\pm}}=(\Omega(x),\Omega(y))_{\mp}.
\]

\end{Lem}

\subsection{Dual canonical basis}
\subsubsection*{Crystal basis}
We define $\mbb{Q}$-subalgebras $\mca{A}_0$, $\mca{A}_\infty$ and $\mca{A}$ of $\Qv$ by 
\begin{align*}
	\mca{A}_0&=\{f\in \Qv; f \text{~is regular at~} v=0\}, \\
	\mca{A}_\infty&=\{f\in \Qv; f \text{~is regular at~} v=\infty\}, \\
	\mca{A}&=\mbb{Q}[v^\pm].
\end{align*}

\begin{Lem}[{\cite[Lemma 3.4.1]{Kas:crystal}, \cite{Nak:CBMS}}]\label{lem:qBoson}
For $x\in \Uv^-(\mfr{g})$ and any $i\in I$, we have
\[[e_i, x]=\frac{r_i(x)t_i-t_i^{-1}{_ir}(x)}{v_i-v_i^{-1}}.\]
\end{Lem}

The \emph{reduced $v$-analogue} $\mscr{B}_v(\mfr{g})$ of a symmetrizable Kac-Moody Lie algebra $\mfr{g}$
is the $\Qv$-algebra generated by ${_ir}$ and $f_i$ with the $v$-Boson relations ${_ir}f_j=v^{-(\alpha_i, \alpha_j)}{f_j}~{_ir}+\delta_{i, j}$ for $i, j\in I$
and the $v$-Serre relations for ${_ir}$ and $f_i$ for $i\in I$.
Then $\Uv^-(\mfr{g})$ becomes a $\mscr{B}_v(\mfr{g})$-module by \lemref{lem:qBoson}.


By the $v$-Boson relation, any element $x\in \Uv^-(\mfr{g})$ can be uniquely written as $x=\sum_{n\geq 0}f_i^{(n)}x_n$ with 
${_ir}(x_n)=0$ for any $n\geq 0$.
So we define Kashiwara's \emph{modified root operators} $\fit{i}$ and $\eit{i}$ by
\begin{align*}
	\eit{i}x&=\sum_{n\geq 1}f_i^{(n-1)}x_n,\\
	\fit{i}x&=\sum_{n\geq 0}f_i^{(n+1)}x_n.
\end{align*}
By using these operators, Kashiwara introduced the crystal basis
$(\mscr{L}(\infty), \mscr{B}(\infty))$ of $\Uv^-(\mfr{g})$:


\begin{Thm}[{\cite{Kas:crystal}}]
We define
	\begin{align*}
	\mscr{L}(\infty)&=\sum_{l\geq 0, i_1, i_2, \cdots, i_l\in I}\mca{A}_0\fit{i_1}\cdots \fit{i_l}1\subset \Uv^-(\mfr{g}), \\
	\mscr{B}(\infty)&=\{\fit{i_1}\cdots \fit{i_l}1 \mod v\mscr{L}(\infty); {l\geq 0, i_1, i_2, \cdots, i_l\in I}\}\subset \mscr{L}(\infty)/v\mscr{L}(\infty).
	\end{align*}
Then we have the following:
\begin{enumerate}
	\item $\mscr{L}(\infty)$ is a free $\mca{A}_0$-module with $\Qv\otimes_{\mca{A}_0}\mscr{L}(\infty)=\Uv^-(\mfr{g})$;
	\item $\eit{i}\mscr{L}(\infty)\subset \mscr{L}(\infty)$ and $\fit{i}\mscr{L}(\infty)\subset \mscr{L}(\infty)$;
	\item $\mscr{B}(\infty)$ is a $\mbb{Q}$-basis of $\mscr{L}(\infty)/v\mscr{L}(\infty)$;
	\item $\fit{i} \colon \mscr{B}(\infty)\to \mscr{B}(\infty)$ and $\eit{i} \colon \mscr{B}(\infty)\to \mscr{B}(\infty)\cup \{0\}$;
	\item For $b\in\Binfty$ with $\eit{i}(b)\neq 0$, we have $\fit{i}\eit{i}b=b$.
\end{enumerate}
\end{Thm}
We call $(\mscr{L}(\infty), \mscr{B}(\infty))$ the \emph{(lower) crystal basis} of $\Uv^-(\mfr{g})$,
and $\mscr{L}(\infty)$ the \emph{(lower) crystal lattice}.
We denote $1\mod v\mscr{L}(\infty)\in \mscr{B}(\infty)$ by $u_{\infty}$ hereafter.
For $b\in \mscr{B}(\infty)$, we set $\vep_{i}(b)=\max\{n\in \mbb{Z}_{\geq 0}; \eit{i}^nb\neq 0\}<\infty$,
and $\eit{i}^{\max}(b)=\eit{i}^{\vep_{i}(b)}b\in \mscr{B}(\infty)$.
\subsubsection*{Canonical basis}
Let $\overline{\phantom{A}}\colon \Qv\to \Qv$ be the $\mbb{Q}$-algebra involution sending $v$ to $v^{-1}$.
Let $V$ be a vector space over $\Qv$,
$\mscr{L}_0$ be an $\mca{A}_0$-submodule of $V$,
$\mscr{L}_\infty$ be an $\mca{A}_\infty$-submodule of $V$, 
and $V_{\mca{A}}$ be an $\mca{A}$-submodule of $V$.
We define $E=\mscr{L}_0\cap \mscr{L}_\infty \cap V_{\mca{A}}$.
\begin{Def}
We say that a triple $(\mscr{L}_0, \mscr{L}_{\infty}, V_{\mca{A}})$ is \emph{balanced} if 
each $\mscr{L}_0, \mscr{L}_{\infty}$, and $V_{\mca{A}}$ generates $V$ as $\Qv$-vector space and if one of the following equivalent conditions is satisfied
\begin{enumerate}
	\item $E\to \mscr{L}_0/v\mscr{L}_0$ is an isomorphism,
	\item $E\to \mscr{L}_\infty/v^{-1}\mscr{L}_\infty$ is an isomorphism,
	\item $(\mscr{L}_0\cap V_{\mca{A}})\+ (v^{-1}\mscr{L}_\infty\cap V_{\mca{A}})\to V_{\mca{A}}$
	is an isomorphism,
	\item $\mca{A}_0\otimes_{\mbb{Q}} E\to \mscr{L}_0$,
	$\mca{A}_\infty\otimes_{\mbb{Q}} E\to \mscr{L}_\infty$,
	$\mca{A}\otimes_{\mbb{Q}} E\to V_{\mca{A}}$,
	and $\Qv\otimes_{\mbb{Q}}E\to V$ are isomorphisms.
\end{enumerate}
\end{Def}

Let $\Uv^{-}(\mfr{g})_{\mca{A}}$ be the $\mca{A}$-subalgebra generated by $\{f_{i}^{(n)}\}_{i\in I,n\geq 1}$.
This is called Kostant-Lusztig $\mca{A}$-form.
\begin{Thm}[{\cite[Theorem 6]{Kas:crystal}}]
	The triple $(\mscr{L}(\infty), \overline{\mscr{L}(\infty)}, \Uv^-(\mfr{g})_{\mca{A}})$
	is balanced.
\end{Thm}
Let $\Glow\colon \mscr{L}(\infty)/v\mscr{L}(\infty)\to E=\mscr{L}(\infty)\cap \overline{\mscr{L}(\infty)} \cap \Uv^-(\mfr{g})_{\mca{A}}$
be the inverse of the isomorphism $E\xrightarrow{\sim} \mscr{L}(\infty)/v\mscr{L}(\infty)$.
Then $\mbf{B}^{\low}_{-}:=\{\Glow(b) ; b\in \mscr{B}(\infty)\}$ forms an $\mca{A}$-basis of $\Uv^-(\mfr{g})_{\mca{A}}$.
This basis is called the \emph{canonical basis} of $\Uv^-(\mfr{g})$.

We define the \emph{dual canonical basis} $\mbf{B}_-^{\mathrm{up}}$ of $\Uv^-(\mfr{g})$ as the dual basis of $\mbf{B}$ under Kashiwara's  bilinear form $(~, ~)_-$. 
\begin{Prop}
We set 
\[\Uv^-(\mfr{g})_{\mca{A}}^{\mathrm{up}}=\{x\in \Uv^-(\mfr{g}); (x, \Uv^-(\mfr{g})_{\mca{A}})_{-}\subset \mca{A}\}.\]
Then $(\mscr{L}(\infty), \sigma(\mscr{L}(\infty)), \Uv^-(\mfr{g})_{\mca{A}}^{\mathrm{up}})$
is a balanced triple for the dual canonical basis $\mbf{B}^{\mathrm{up}}$.
\end{Prop}
Here we have the following isomorphism of $\mbb{Q}$-vector spaces:
\[\mscr{L}(\infty)\cap \sigma(\mscr{L}(\infty))\cap \Uv^-(\mfr{g})_{\mca{A}}^{\mathrm{up}}\xrightarrow{\sim} \mscr{L}(\infty)/v\mscr{L}(\infty).\]
Denote its inverse by $\Gup$.
Then we have $\mbf{B}_-^{\mathrm{up}}=\Gup(\mscr{B}(\infty))$, \cf
\cite[Theorem 4.26]{Kimura10}.
Then the \emph{dual canonical basis}
$\mbf{B}_+^{\mathrm{up}}$ of $\Uv^+(\mfr{g})$ is defined to be
$\Omega(\mbf{B}_-^{\mathrm{up}})$. Notice that the dual canonical bases are
dual bar-involution invariant ($\sigma_\pm$-invariant).


\subsection{Quantum unipotent subgroup}
Let $W$ be the Weyl group associated with the given root datum and
$s_i$ the reflection associated with the root $\alpha_i$, $1\leq
i\leq n$. Let $\ell\colon W\to \Z_{\geq 0}$ denote the natural length function on $W$.
For any given group element $w\in W$, we denote by $R(w)$ the set of reduced
words of $w$. Define $\Phi_{+}(w)=\{\alpha\in \Phi_{+};
w^{-1}\alpha\in \Phi_{-}\}$.

Following \cite[37.1.3]{Lus:intro}, we define the $\Qv$-algebra automorphisms\footnote{This automorphism is denoted by $T'_{i, -1}$ in \cite[37.1.3]{Lus:intro} and this is denoted by $T_{i}^{-1}$ in \cite{Kimura10}.} $T_{i}\colon \mbf{U}_{q}(\mfr{g})\to \mbf{U}_{q}(\mfr{g})$ for $i\in I$ by
\begin{subequations}
\begin{align}
	T_{i}(v^h)&=v^{s_i(h)}, \\
	T_{i}(e_i)&=-t_i^{-1}f_i, \\
	T_{i}(f_i)&=-e_it_i, \\
	T_{i}(e_j)&=\sum_{r+s=-\braket{h_i, \alpha_j}}(-1)^r v_i^{-r}e_i^{(r)}e_j e_i^{(s)} \text{~for~} j\neq i, \\
	T_{i}(f_j)&=\sum_{r+s=-\braket{h_i, \alpha_j}}(-1)^r v_i^{r}f_i^{(s)}f_j f_i^{(r)} \text{~for~} j\neq i.
\end{align}
\end{subequations}

Fix an element $w\in W$ with $\ell(w)=\ell$ and a reduced word $\overrightarrow{w}=(i_{1}, \cdots, i_{\ell})\in R(w)$.
We set 
\[\beta_{k}=s_{i_{1}}\cdots s_{i_{k-1}}(\alpha_{i_k}).\]
Then we have $\{\beta_{k}\}_{1\leq k\leq \ell}=\Phi_+(w)$. 
\begin{Eg}\label{eg:A_3}
  Let the Cartan matrix $C$ be given by
  \begin{align*}
    C=
    \begin{pmatrix}
      2&0&-1\\
      0&2&-1\\
      -1&-1&2
    \end{pmatrix}.
  \end{align*}
Take the Coxeter element $c=s_3s_2s_1$. Take $w=c^2$ and choose the reduced
word $\overrightarrow{w}=(3,2,1,3,2,1)\in R(w)$. Then we have
\begin{align*}
  \beta_1&=\alpha_3,\\
  \beta_2&=\alpha_2+\alpha_3,\\
  \beta_3&=\alpha_1+\alpha_3,\\
  \beta_4&=\alpha_1+\alpha_2+\alpha_3,\\
  \beta_5&=\alpha_1,\\
  \beta_6&=\alpha_2.
\end{align*}
\end{Eg}

Define the lexicographic order $\wLess$ on $\mathbb{Z}_{\geq0}^{\ell}$ associated
with $\overrightarrow{w}\in R(w)$ by 
\begin{align*}
 \mathbf{c}=(c_{1},c_{2},\cdots,c_{\ell})<_{\overrightarrow{w}}\mathbf{c}'=(c'_{1},c'_{2},\cdots,c'_{\ell})
\end{align*}
if and only if there exists $1\leq p\leq\ell$ such that we have $c_{1}=c'_{1},\cdots,c_{p-1}=c'_{p-1},c_{p}<c'_{p}$.

Following Lusztig, for any non-negative integer $m$ and any vector
$\mbf{c}\in\N^\ell$, we denote
\begin{align*}
F(m\beta_{k})&=T_{i_{1}}\cdots T_{i_{k}}(f_{i_{k}}^{(m)}),\\
E(m\beta_{k})&=T_{i_{1}}\cdots T_{i_{k}}(e_{i_{k}}^{(m)}).\\
F(\mbf{c},\ow)&=F(c_\ell\beta_{\ell})F(c_{\ell-1}\beta_{\ell-1})\cdots
F(c_1\beta_{1}),\\
E(\mbf{c},\ow)&=E(c_1\beta_1)E(c_2\beta_2)\cdots E(c_l\beta_{l}).
\end{align*}

Let $\Uv^{-}(w)$ and $\Uv^{+}(w)$ be the $\Qv$-subspace of $\Uv^{-}(\mfr{g})$ and $\Uv^{+}(\mfr{g})$ spanned by 
\[\mscr{P}_{\overrightarrow{w}}=\{F(\mbf{c},
\overrightarrow{w})\mid \mbf{c}\in \mbb{Z}_{\geq 0}^{\ell}\}.\]
and $\{E(\mbf{c}, \overrightarrow{w})\mid\mbf{c}\in \mbb{Z}_{\geq
  0}^{\ell}\}$ respectively. 
Lusztig has shown that the space $\mbf{U}_{v}^{-}(w)$ is independent
of the choice of the reduced word $\overrightarrow{w}\in R(w)$.
As  a consequence of Levendorskii-Soibelman's formula, it can be shown that the subspace $\Uv^-(w)$ is the $\Qv$-subalgebra generated by $\{F(\beta_k)\}_{1\leq k\leq \ell}$.
(\cf \cite[2.4.2 Proposition Theorem b)]{LevSoi:qWeyl} and \cite[2.2 Proposition]{DKP:solvable}.)

\begin{Rem}
  By its construction, $\Uv^\pm(w)$ can be considered as a quantum
  analogue of the universal enveloping algebra of the nilpotent Lie
  algebra $n_{\pm}(w)=\bigoplus_{\pm\alpha\in
    \Phi_{+}(w)}\mfr{g}_{\alpha}$.
\end{Rem}

For any  $\mathbf{c}\in\mathbb{Z}_{\geq0}^{\ell}$, we set
\begin{align*}
F^{\mathrm{up}}(\mathbf{c},\overrightarrow{w})&=\frac{1}{(F(\mathbf{c},\overrightarrow{w}),F(\mathbf{c},\overrightarrow{w}))_{-}}F(\mathbf{c},\overrightarrow{w}),\\
E^{\mathrm{up}}(\mathbf{c},\overrightarrow{w})&=\frac{1}{(E(\mathbf{c},\overrightarrow{w}),E(\mathbf{c},\overrightarrow{w}))_{+}}E(\mathbf{c},\overrightarrow{w}).
\end{align*}

Define the $\mathbb{Q}[v^{\pm1}]$-form $\mathbf{U}_{v}^{-}(w)_{\mathbb{Q}[v^{\pm1}]}^{\mathrm{up}}$
of $\mathbf{U}_{q}^{-}(w)$ by 
\[
\mathbf{U}_{v}^{-}(w)_{\mathbb{Q}[v^{\pm1}]}^{\mathrm{up}}=\bigoplus_{\mathbf{c}\in\mathbb{Z}_{\geq0}^{\ell}}\mathbb{Q}[v^{\pm1}]F^{\mathrm{up}}(\mathbf{c},\overrightarrow{w}).
\]
$\mathbf{U}_{v}^{+}(w)_{\mathbb{Q}[v^{\pm1}]}^{\mathrm{up}}$ is
defined similarly.

By the Levendorskii-Soibelman formula \cite[Theorem 4.24]{Kimura10}
with respect to the set $\mathscr{P}_{\overrightarrow{w}}^{\mathrm{up}}=\{F(\mbf{c}, \overrightarrow{w})|\mbf{c}\in \mbb{Z}_{\geq 0}^{\ell}\}$,
$\mathbf{U}_{v}^{-}(w)_{\mathbb{Q}[v^{\pm1}]}^{\mathrm{up}}$
is a $\mathbb{Q}[v^{\pm1}]$-subalgebra generated by $\{F^{\mathrm{up}}(\beta_{k})\}_{1\leq k\leq\ell}$.
We also obtain the following upper unitriangular property of dual bar involution $\sigma_{-}$ with respect to $\mathscr{P}_{\overrightarrow{w}}^{\mathrm{up}}$ by the Levendorskii-Soibelman formula.

\begin{Prop}[{\cite[Proposition 4.25]{Kimura10}}]
We have 
\[\sigma(F^{\mathrm{up}}(\mathbf{c},\overrightarrow{w}))-F^{\mathrm{up}}(\mathbf{c},\overrightarrow{w})\in\sum_{\mathbf{c}'<_{\overrightarrow{w}}\mathbf{c}}\mathbb{Q}[v^{\pm1}]F^{\mathrm{up}}(\mathbf{c}',\overrightarrow{w}).\]
\end{Prop}


\begin{Prop}[{\cite[Theorem 4.26]{Kimura10}}]
We have the following upper unitriangular property:
\begin{align*}
  G^{\mathrm{up}}(b(\mbf{c}, \overrightarrow{w}))-\Fup(\mbf{c},
  \overrightarrow{w})\in \sum_{\mbf{c'}<\mbf{c}}v\mbb{Z}[v]\Fup(\mbf{c}',
  \overrightarrow{w}).
\end{align*}
\end{Prop}

\begin{Rem}
In \cite[Theorem 4.26]{Kimura10}, we stated a slightly weaker statement about the coefficient in the right hand side.
But we have the integrality property \cite[6.1]{Kas:crystal} and \cite[Proposition 41.1.3]{Lus:intro}, so we get the result in the above form.
\end{Rem}

Via the identifications of the dual canonical bases $\mbf{B}_\pm^{\up}$ with
the $G^{\mathrm{up}}(\mscr{B}(\infty))$, the corresponding dual
canonical basis elements are denoted by $B_\pm(\mbf{c},\ow)$ respectively.


We have the following properties between $\Omega$ and
$T_{i,\epsilon}$ on $\mathbf{U}_{v}(\mathfrak{g})$.

\begin{Lem}[{\cite[Lemma 7.2]{GeissLeclercSchroeer11}}]
We have $T_{i,\epsilon}\circ\Omega=\Omega\circ T_{i,\epsilon}$. \end{Lem}
 \begin{NB}
 The following is just a check.
 \begin{proof}
 Since both are $\mathbb{Q}$-algebra anti-isomorphism which maps $v$
 to $v^{-1}$. We only have to check on generators of $\mathbf{U}_{v}(\mathfrak{g})$
 \begin{align*}
 (T'_{i,\epsilon}\circ\Omega)(v^{h}) & =T'_{i,\epsilon}(v^{-h})=v^{-s_{i}(h)}\\
 (T'_{i,\epsilon}\circ\Omega)(e_{i}) & =T'_{i,\epsilon}(f_{i})=-e_{i}t_{i}^{-\epsilon}\\
 (T'_{i,\epsilon}\circ\Omega)(f_{i}) & =T'_{i,\epsilon}(e_{i})=-t_{i}^{\epsilon}f_{i}\\
 (T'_{i,\epsilon}\circ\Omega)(e_{j}) & =T'_{i,\epsilon}(f_{j})=\sum_{r+s=-a_{ij}}(-1)^{r}v_{i}^{-\epsilon r}f_{i}^{(s)}f_{j}f_{i}^{(r)}\;\text{for \ensuremath{i\neq j}}\\
 (T'_{i,\epsilon}\circ\Omega)(f_{j}) & =T'_{i,\epsilon}(e_{j})=\sum_{r+s=-a_{ij}}(-1)^{r}v_{i}^{\epsilon r}e_{i}^{(r)}e_{j}e_{i}^{(s)}\;\text{for \ensuremath{i\neq j}}
 \end{align*}

 \begin{align*}
 (\Omega\circ T'_{i,\epsilon})(v^{h}) & =\Omega(v^{s_{i}(h)})=v^{-s_{i}(h)}\\
 (\Omega\circ T'_{i,\epsilon})(e_{i}) & =-\Omega(t_{i}^{\epsilon}f_{i})=-\Omega(f_{i})\Omega(t_{i}^{\epsilon})=-e_{i}t_{i}^{-\epsilon}\\
 (\Omega\circ T'_{i,\epsilon})(f_{i}) & =-\Omega(e_{i}t_{i}^{-\epsilon})=-\Omega(t_{i}^{-\epsilon})\Omega(e_{i})=-t_{i}^{\epsilon}f_{i}\\
 (\Omega\circ T'_{i,\epsilon})(e_{j}) & =\Omega\left(\sum_{r+s=-a_{ij}}(-1)^{r}v_{i}^{\epsilon r}e_{i}^{(r)}e_{j}e_{i}^{(s)}\right)=\sum_{r+s=-a_{ij}}(-1)^{r}v_{i}^{-\epsilon r}f_{i}^{(s)}f_{j}f_{i}^{(r)}\text{for \ensuremath{i\neq j}}\\
 (\Omega\circ T'_{i,\epsilon})(f_{j}) & =\Omega\left(\sum_{r+s=-a_{ij}}(-1)^{r}v_{i}^{-\epsilon r}f_{i}^{(s)}f_{j}f_{i}^{(r)}\right)=\sum_{r+s=-a_{ij}}(-1)^{r}v_{i}^{\epsilon r}e_{i}^{(r)}e_{j}e_{i}^{(s)}\text{for \ensuremath{i\neq j}}
 \end{align*}
 \end{proof}

 \end{NB}

\begin{Lem}
We have $\Omega(F^{\mathrm{up}}(\mathbf{c},\overrightarrow{w}))=E^{\mathrm{up}}(\mathbf{c},\overrightarrow{w})$.
\end{Lem}
\begin{proof}
Applying the anti-automorphism $\Omega$, we obtain
  \begin{align*}
    \Omega(F^{\mathrm{up}}(\mathbf{c},\overrightarrow{w})) & =\Omega(F^{\mathrm{up}}(\mbf{c}_{\ell}\beta_{\ell})\cdots F^{\mathrm{up}}(\mbf{c}_{1}\beta_{1}))\\
    & =\Omega(F^{\mathrm{up}}(\mbf{c}_{1}\beta_{1}))\cdots\Omega(F^{\mathrm{up}}(\mbf{c}_{\ell}\beta_{\ell}))\\
    & =E^{\mathrm{up}}(\mbf{c}_{1}\beta_{1})\cdots
    E^{\mathrm{up}}(\mbf{c}_{\ell}\beta_{\ell}).
  \end{align*}
\end{proof}

As a consequence, we obtain the following property.
\begin{Prop}[{\cite[Proposition 12.8]{GeissLeclercSchroeer11}}]
We have 
\[
B_{+}^{\mathrm{up}}(\mathbf{c},\overrightarrow{w})\in E^{\mathrm{up}}(\mathbf{c},\overrightarrow{w})+\sum_{\mathbf{c}'<\mathbf{c}}v^{-1}\mathbb{Z}[v^{-1}]E^{\mathrm{up}}(\mathbf{c}',\overrightarrow{w}).
\]
\end{Prop}

\subsection{Compatibility}

For $b_1, b_2\in \mscr{B}(\infty)$,
we say $b_1$ and $b_2$ (or $\Gup(b_1)$ and $\Gup(b_2)$) are \emph{multiplicative} or \emph{compatible} if
 there exists a unique element in $\mscr{B}(\infty)$, which we denote
 by $b_1\circledast b_2$, such that 
$\Gup(b_1\circledast b_2)$ equals $v^N \Gup(b_1)\Gup(b_2)$ for some $N\in \Z$. By \cite[Corollary 3.8]{Kimura10}, this condition is independent of the order on $b_{1}$ and $b_{2}$.
We write $b_{1}\bot b_{2}$ when this holds.

\section{$T$-system in quantum unipotent subgroup}\label{sec:GLS}

\subsection{Quantum coordinate ring}

Following Kashiwara \cite[\S 7]{Kashiwara93} and Gei\ss-Leclerc-Schr\"{o}er
\cite[\S 2]{GeissLeclercSchroeer11}, we define the quantum coordinate
ring $A_{v}(\mathfrak{g})$ as the subspace of $\Hom(\mathbf{U}_{v}(\mathfrak{g}),\mathbb{Q}(v))$
consisting of the linear forms $\psi$ such that the left module
$\mathbf{U}_{v}(\mathfrak{g})\psi$ belongs to $\mathcal{O}_{\mathrm{int}}(\mathfrak{g})$
and the right module $\psi\mathbf{U}_{v}(\mathfrak{g})$ belong to $
\mathcal{O}_{\mathrm{int}}(\mathfrak{g}^{\mathrm{op}})$. Its
multiplication is chosen to be the transpose of one of the coproducts
$\Delta_\pm$.

  For any $\lambda\in P_+$, let $V(\lambda)$ and $V(\lambda)^r$ denote
  the left irreducible highest weight module and the right irreducible
  highest weight module respectively with highest weight $\lambda$. The highest
  weight vectors are denoted by $m_\lambda$ and $n_\lambda$ respectively.
  Let $\langle~,~\rangle_{\lambda}\colon V(\lambda)^{r}\otimes V(\lambda)\to \Qv$ be the 
  bilinear form which is characterized by 
  $\langle n_{\lambda},m_{\lambda} \rangle_{\lambda}=1$ and $\langle n,xm\rangle_{\lambda}=\langle nx,m\rangle$
  for $n\in V(\lambda)^{r}$, $m\in V(\lambda)$ and $x\in \Uv(\mathfrak{g})$.
  We also have the bilinear form $(\;,\;)_{\lambda}\colon V(\lambda)\otimes V(\lambda)\to \Qv$ which is
  characterized by $(m_{\lambda},m_{\lambda})_{\lambda}=1$ and $(m,xm')_{\lambda}=(\varphi(x)m,m')_{\lambda}$
  for $x\in \Uv(\mfr{g})$ and $m,m'\in V(\lambda)$.
  
  The following is a $v$-analogue of the Peter-Weyl decomposition theorem for the strongly regular functions on the Kac-Moody group $G_{\min}$ in the sense of Kac-Peterson.
\begin{Prop}[{\cite[Proposition 7.2.2]{Kashiwara93}, \cite[Proposition 2.1]{GeissLeclercSchroeer11}}]

Let $\Phi_{\lambda}\colon V(\lambda)^{r}\otimes V(\lambda)\to A_{v}(\mathfrak{g})$
defined by 
\[
\langle\Phi_{\lambda}(n\otimes m),x\rangle=\langle n,xm\rangle_{\lambda}\;(n\in V(\lambda)^{r},m\in V(\lambda),x\in\mathbf{U}_{v}(\mathfrak{g}))
\]
Then $\Phi=\bigoplus_{\lambda\in P_{+}}\Phi_{\lambda}\colon\bigoplus_{\lambda\in P_{+}}V(\lambda)^{r}\otimes V(\lambda)\to A_{v}(\mathfrak{g})$
gives an isomorphism of $\mathbf{U}_{v}(\mathfrak{g})$-bimodules.
\end{Prop}

\subsection{Quantum $T$-systems for quantum minors}
Following \cite[Section 9.3]{BerensteinZelevinsky05}, we define the
$v$-analogue $\Delta^{\lambda}$ of a principal minor as :
\[
\langle\Delta^{\lambda},x\rangle=\langle
n_{\lambda},xm_{\lambda}\rangle\;\forall
x\in\mathbf{U}_{v}(\mathfrak{g}), \lambda\in P_{+}.
\]

By its definition, 
we have $\langle\Delta^{\lambda},yv^{h}x\rangle=\varepsilon(x)v^{\langle h,\lambda\rangle}\varepsilon(y)$
for $x\in\mathbf{U}_{v}^{+}(\mathfrak{g}),y\in\mathbf{U}_{v}^{-}(\mathfrak{g}),h\in P^{\vee}$,
where $\varepsilon\colon \mathbf{U}_{v}(\mathfrak{g})\to \mathbb{Q}(v)$ is the counit.

For $w\in W$ and $\lambda\in P_{+}$, let us denote by $m_{w\lambda}$ the (dual) canonical basis element
of weight $w\lambda$. We have the following description:
\[m_{w\lambda}=f_{i_{1}}^{(a_{1})}\cdots f_{i_{\ell}}^{(a_{\ell})}m_{\lambda},\]
where $(i_{1},i_{2},\cdots,i_{\ell})\in R(w)$ and $a_{k}=\langle s_{i_{\ell}}\cdots s_{i_{k+1}}(h_{i_{k}}),\lambda\rangle~(1\leq k\leq \ell)$. It is known that $m_{w\lambda}$ does not depend on the choice of a reduced word $(i_{1},\cdots,i_{\ell})\in R(w)$.
Similarly we define $n_{w\lambda}$ by $n_{w\lambda}=n_{\lambda}e_{i_{\ell}}^{(a_{\ell})}\cdots e_{i_{1}}^{(a_{1})}$.

For $w_{1}, w_{2}\in W$ and $\lambda\in P_{+}$, we define the \emph{(generalized) quantum minor} $\Delta_{w_{1}\lambda, w_{2}\lambda}$ associated with $(w_{1}\lambda,w_{2}\lambda)$ by
\[\Delta_{w_{1}\lambda, w_{2}\lambda}=\Phi_{\lambda}(n_{w_{1}\lambda}\otimes m_{w_{2}\lambda}).\]
By construction, we have $\Delta_{w_{1}\lambda, w_{2}\lambda}\in A_{v}(\mathfrak{g})_{w_{1}\lambda,w_{2}\lambda}$
and 
\begin{align*}
\langle\Delta_{w_{1}\lambda,w_{2}\lambda},x\rangle&=\langle n_{w_{1}\lambda},xm_{w_{2}\lambda}\rangle_{\lambda}=\langle n_{w_{1}\lambda}x,m_{w_{2}\lambda}\rangle_{\lambda} \\
&=(m_{w_{1}\lambda}, xm_{w_{2}\lambda})_{\lambda}=(\varphi(x)m_{w_{1}\lambda}, m_{w_{2}\lambda})_{\lambda}
\end{align*}
for $x\in \mathbf{U}_{v}(\mathfrak{g})$.

Denote $\gamma_{i}=\varpi_{i}+s_{i}\varpi_{i}\in P_{+}$.

\begin{Prop}[{\cite[Proposition 3.2]{GeissLeclercSchroeer11}}]\label{prop:DeltaTsys}

\textup{(1)} For $i\in I$, we have 
\begin{align*}
\Delta_{\gamma_{i},\gamma_{i}} & =\Delta_{s_{i}\varpi_{i},s_{i}\varpi_{i}}\Delta_{\varpi_{i},\varpi_{i}}-v_{i}^{-1}\Delta_{s_{i}\varpi_{i},\varpi_{i}}\Delta_{\varpi_{i},s_{i}\varpi_{i}}\\
 & =\Delta_{\varpi_{i},\varpi_{i}}\Delta_{s_{i}\varpi_{i},s_{i}\varpi_{i}}-v_{i}\Delta_{s_{i}\varpi_{i},\varpi_{i}}\Delta_{\varpi_{i},s_{i}\varpi_{i}}.
\end{align*}

\textup{(2)} For $w_{1},w_{2}\in W$ and $i\in I$ with $\ell(w_{1}s_{i})=\ell(w_{1})+1$
and $\ell(w_{2}s_{i})=\ell(w_{2})+1$, we have 
\begin{align*}
\Delta_{w_{1}\gamma_{i},w_{2}\gamma_{i}} & =\Delta_{w_{1}s_{i}\varpi_{i},w_{2}s_{i}\varpi_{i}}\Delta_{w_{1}\varpi_{i},w_{2}\varpi_{i}}-v_{i}^{-1}\Delta_{w_{1}s_{i}\varpi_{i},w_{2}\varpi_{i}}\Delta_{w_{1}\varpi_{i},w_{2}s_{i}\varpi_{i}}\\
 & =\Delta_{w_{1}\varpi_{i},w_{2}\varpi_{i}}\Delta_{w_{1}s_{i}\varpi_{i},w_{2}s_{i}\varpi_{i}}-v_{i}\Delta_{w_{1}s_{i}\varpi_{i},w_{2}\varpi_{i}}\Delta_{w_{1}\varpi_{i},w_{2}s_{i}\varpi_{i}}.
\end{align*}
\end{Prop}

We note that this relation does not depend on a choice of coproduct
$\Delta_{+}$ or $\Delta_{-}$.
\subsection{Quantum $T$-systems for unipotent quantum minors}
Let $A_{v}(\mfr{n}_{\pm})$ be the graded dual of $\Uv^{\pm}(\mfr{g})$
with respect to Kashiwara's bilinear form. We define the product $\rpm^{*}\colon A_{v}(\mfr{n}_{\pm})\otimes_{\Qv} A_{v}(\mfr{n}_{\pm})\to A_{v}(\mfr{n}_{\pm})$ by
\[\braket{r^{*}_{\pm}(\psi_{1}\otimes \psi_{2}),
  x}=\braket{\psi_{1}\otimes\psi_{2}, \rpm(x)}.\]
Let $\rho_{\pm}\colon A_{v}(\mfr{g}) \to A_{v}(\mfr{n}_{\pm})$ be the restriction linear homomorphism which is defined by
$\braket{\rho_{\pm}(\psi), x}=\braket{\psi, x}$ for $x\in \Uv^{\pm}(\mfr{g})$.

We have the following twisting formula:
\begin{Lem}\label{lem:rtwist}
Let $\psi_{1}\in A_{v}(\mfr{g})_{\nu_{1}, \mu_{1}}$ and  $\psi_{2}\in A_{v}(\mfr{g})_{\nu_{2}, \mu_{2}}$.
Then we have
\begin{align}
\rho_{\pm}(\Delta^{*}_{\pm}(\psi_{1}\otimes \psi_{2}))=v^{\pm(\nu_{2}-\mu_{2}, \mu_{1})}r^{*}_{\pm}(\rho_{\pm}(\psi_{1})\otimes \rho_{\pm}(\psi_{2})).
\end{align}
\end{Lem}
\begin{proof}
This can be proved by the following straightforward calculation.
\begin{align*}
\braket{\Delta^{*}_{\pm}(\psi_{1}\otimes \psi_{2}), x}&=\braket{\psi_{1}\otimes \psi_{2}, \Delta_\pm(x)}~(x\in \Uv^{\pm}(\mfr{g}))\\
&=\sum \braket{\psi_{1}\otimes \psi_{2}, x_{(1)}t_{\pm\wt(x_{(2)})}\otimes x_{(2)}}\\
&=\sum \braket{(t_{\pm\wt(x_{(2)})}\psi_{1})\otimes \psi_{2}, x_{(1)}\otimes x_{(2)}} \\
&=v^{\pm(\nu_{2}-\mu_{2}, \mu_{1})}\braket{r^{*}_\pm(\psi_{1}\otimes
  \psi_{2}), x}~(\mathrm{because}\ \wt x_{(2)}+\mu_{2}=\nu_{2}).
\end{align*}
\end{proof}

\begin{NB}
By the normalization between Kashiwara's form and Lusztig's form,
we have 
\[(x'x'',y)_{\pm}=(x'\otimes x'',\rpm(y))_{\pm}\]
for $x',x'',y\in \Uv^{\pm}(\mathfrak{g})$ and 
the linear isomorphisms $\psi_{\pm}\colon \Uv^{\pm}(\mathfrak{g})\to A_{v}(\mathfrak{n}_{\pm})$ defined by $\langle\psi_{\pm}(x),y\rangle=(x,y)_{\pm}$ give algebra isomorphisms between $\Uv^{\pm}(\mathfrak{g})$ and 
$(A_{v}^{\pm}(\mathfrak{n}_{\pm}),r^{*}_{\pm})$.
Hence it means $(\Uv^{\pm}(\mathfrak{g}), (~,~)_{\pm})$ form self-dual
(twisted) bialgebras.
\end{NB}
Let $\psi_{\pm}\colon  A_{v}(\mathfrak{n}_{\pm})\to \Uv^{\pm}(\mathfrak{g})$ be the $\Qv$-linear isomorphism defined by 
$(\psi_\pm(f),x)_\pm=\langle f,x\rangle$.
It can be shown that $\psi_\pm$ are $\mathbb{Q}(v)$-algebra isomorphisms which intertwine $r_{\pm}^{*}$ and the usual product of $\mathbf{U}_v^\pm(\mathfrak{g})$.
For any given $w\in W$, the associated \emph{quantum coordinate ring} $A_{v}(\mathfrak{n}_+(w))$
is defined to be the subalgebra $(\psi_+)^{-1}(\Uv^+(w))$ of
$A_{v}(\mathfrak{n}_+)$.

\subsubsection*{Unipotent quantum minors}

\begin{Def}
We define the quantum unipotent minor $D_{w_{1}\lambda,w_{2}\lambda}^{\pm}$
on $\mathbf{U}_{v}^{\pm}(\mathfrak{g})$ by the following formula.
\[
(D_{w_{1}\lambda,w_{2}\lambda}^{\pm},x)_{\pm}=(m_{w_{1}\lambda},xm_{w_{2}\lambda})_{\lambda}=(\varphi(x)m_{w_{1}\lambda},m_{w_{2}\lambda})_{\lambda},
\]
where $m_{w_{1}\lambda}$ (resp. $m_{w_{2}\lambda}$) is the \emph{extremal weight
vector} of weight $w_1\lambda$ (resp. $w_2\lambda$).
\end{Def}

\begin{Rem}
We have $D_{w_{1}\lambda,w_{2}\lambda}^{\pm}=\psi_{\pm}\rho_{\pm}(\Delta_{w_{1}\lambda,w_{2}\lambda}^{\pm})$.
\end{Rem}

By construction we have
$\sigma^{\pm}D_{w_{1}\lambda,w_{2}\lambda}^{\pm}=D_{w_{1}\lambda,w_{2}\lambda}^{\pm}$. We
also have, for any $x\in\mathbf{U}_{v}^{-}(\mathfrak{g})$, 
\begin{align*}
(\Omega(D_{w_{1}\lambda,w_{2}\lambda}^{+}),x)_{-} & =\overline{(D_{w_{1}\lambda,w_{2}\lambda}^{+},\Omega(x))_{+}}=(\sigma_{+}(D_{w_{1}\lambda,w_{2}\lambda}^{+}),\varphi(x))_{+}\\
 & =(D_{w_{1}\lambda,w_{2}\lambda}^{+},\varphi(x))_{+}=(u_{w_{1}\lambda,}\varphi(x)u_{w_{2}\lambda})_{\lambda}\\
 & =(xu_{w_{1}\lambda},u_{w_{2}\lambda})_{\lambda}=(u_{w_{2}\lambda},xu_{w_{1}\lambda})_{\lambda}=(D_{w_{2}\lambda,w_{1}\lambda}^{-},x)_{-}.
\end{align*}

Since the opposite Demazure module $V^{w}(\lambda)=\Uv^{-}(\mfr{g})m_{w\lambda}$ is compatible with the canonical basis (see \cite[Proposition 4.1 (i)]{Kas:Demazure})
and $m_{w\lambda}$ is also a dual canonical basis element, we have
$D^-_{w_{1}\lambda, w_{2}\lambda}\in \mbf{B}_-^{\mathrm{up}}\cup
\{0\}$. It follows that $D^+_{w_{2}\lambda, w_{1}\lambda}$ is
contained in $\mbf{B}_+^{\mathrm{up}}\cup
\{0\}$

Fix $w_{1}, w_{2}\in W$ and $i\in I$ with $\ell(w_{j}s_{i})=\ell(w_{j})+1$.
The following just follows from \lemref{lem:rtwist}.
\begin{Lem}\label{lem:deltaMinor}
	We have 
\begin{NB}
	\begin{align*}
	\rho_-(\Delta_{w_{1}s_{i}\varpi_{i}, w_{2}s_{i}\varpi_{i}}\Delta_{w_{1}\varpi_{i}, w_{2}\varpi_{i}})
	=&v^{-A}D^-_{w_{1}s_{i}\varpi_{i}, w_{2}s_{i}\varpi_{i}}D^-_{w_{1}\varpi_{i}, w_{2}\varpi_{i}}, \\
	\rho_-(\Delta_{w_{1}\varpi_{i}, w_{2}s_{i}\varpi_{i}}\Delta_{w_{1}s_{i}\varpi_{i}, w_{2}\varpi_{i}})
	=&v^{-B}D^-_{w_{1}\varpi_{i}, w_{2}s_{i}\varpi_{i}}D^-_{w_{1}s_{i}\varpi_{i}, w_{2}\varpi_{i}},
	\end{align*}
	where 
	\begin{align*}
	A&=(w_{1}\varpi_{i}-w_{2}\varpi_{i}, w_{2}s_{i}\varpi_{i}), \\
	B&=(w_{1}s_{i}\varpi_{i}-w_{2}\varpi_{i}, w_{2}s_{i}\varpi_{i}).
	\end{align*}
\end{NB}
\begin{align*}
\psi_{+}\rho_{+}(\Delta_{+}^{*}(\Delta_{w_{1}s_{i}\varpi_{i},w_{2}s_{i}\varpi_{i}}\otimes\Delta_{w_{1}\varpi_{i},w_{2}\varpi_{i}})) & =v^{A}D_{w_{1}s_{i}\varpi_{i},w_{2}s_{i}\varpi_{i}}^{+}D_{w_{1}\varpi_{i},w_{2}\varpi_{i}}^{+},\\
\psi_{+}\rho_{+}(\Delta_{+}^{*}(\Delta_{w_{1}s_{i}\varpi_{i},w_{2}\varpi_{i}}\otimes\Delta_{w_{1}\varpi_{i},w_{2}s_{i}\varpi_{i}})) & =v^{B}D_{w_{1}s_{i}\varpi_{i},w_{2}\varpi_{i}}^{+}D_{w_{1}\varpi_{i},w_{2}s_{i}\varpi_{i}}^{+},
\end{align*}
where 
\begin{align*}
A&=(w_{1}\varpi_{i}-w_{2}\varpi_{i},w_{2}s_{i}\varpi_{i}),\\
B&=(w_{1}\varpi_{i}-w_{2}s_{i}\varpi_{i},w_{2}\varpi_{i}).
\end{align*}
\end{Lem}

\subsubsection*{Lusztig's parametrization}

Fix $w\in W$ with $\ell(w)=\ell$ and a reduced word $\overrightarrow{w}=(i_{1}, \dots, i_{\ell})\in R(w)$.
For $1\leq a\leq b\leq \ell$ with $i_{a}=i_{b}=i$, we define the vector $\mbf{c}[a, b]\in \mbb{Z}_{\geq 0}^{\ell}$ by
\[\mbf{c}[a, b]_{k}=\begin{cases}1 & \text{if $a\leq k\leq
    b$ and $i_{k}=i_{a}=i_{b}$}, \\ 0 & \text{otherwise},\end{cases}\]
    for $1\leq k\leq \ell$.
By convention, we define $\mbf{c}[0,0]=0$, $\mbf{c}[0,b]=\mbf{c}[1,b]$ if $b\geq 1$, and
$\mbf{c}[a,b]=0$ if $a>b$. We set $\wt[a, b]=-\wt\Gup(\mbf{c}[a, b], \overrightarrow{w})=\sum_{a\leq k\leq b \text{~with~}i_{a}=i_{k}=i_{b}}\beta_{k}$.
For $1\leq k\leq \ell$, we have 
$\wt[k_{\min},k]=\varpi_{i_{k}}-s_{i_{1}}\cdots
s_{i_{k}}\varpi_{i_{k}}$.

As in \cite[Section 13]{GeissLeclercSchroeer10}, for any $1\leq k\leq \ell$ and
$j\in I$, we define
\begin{align}
	k(j)&=\#\{1\leq s\leq k-1; i_{s}=j\},\\
	k_{\min}&=\min\{1\leq s\leq \ell; i_{s}=i_{k}\},\\
	k_{\max}&=\max\{1\leq s\leq \ell; i_{s}=i_{k}\},\\
k^{-}&=\max\{\{1\leq s\leq k-1; i_{s}=i_{k}\}\cup\{0\}\},\\
k^{+}&=\min\{\{k+1\leq s\leq \ell; i_{s}=i_{k}\}\cup\{\ell+1\}\}.
\end{align}

\begin{Prop}
For $1\leq a\leq b\leq \ell$ with $i_{a}=i_{b}=i$,
we have 
\begin{align}
\Gup(\mbf{c}[a, b], \overrightarrow{w})
&=D^-_{s_{i_{1}}\dots s_{i_{b}}\varpi_{i}, s_{i_{1}}\cdots s_{i_{a^{-}}}\varpi_{i}},\\
&=D^-_{s_{i_{1}}\dots s_{i_{b^{+}-1}}\varpi_{i}, s_{i_{1}}\cdots s_{i_{a-1}}\varpi_{i}}.
\end{align}
\end{Prop}
\begin{proof}
For a reduced word $(i_{a}, \cdots, i_{b})$ of $s_{i_{a}}\dots s_{i_{b}}$, we have
\[\Gup(\mbf{c}[1,b-a+1], (i_{a}, \cdots, i_{b}))=D^-_{s_{i_{a}}\cdots s_{i_{b}}\varpi_{i}, \varpi_{i}},\]
by \cite[Proposition 6.3]{Kimura10}.

By applying \cite[Theorem 4.20]{Kimura10} (see also \cite[Proposition 7.1]{GeissLeclercSchroeer11}), we get the result.
\end{proof}
\begin{Rem}
It follows that $B^{\mathrm{up}}_+(\mbf{c}[a, b], \overrightarrow{w})$
equals $D^+_{s_{i_{1}}\dots s_{i_{a^-}}\varpi_{i}, s_{i_{1}}\cdots
  s_{i_{b}}\varpi_{i}}$, which is denoted by $D(a^-,b)$ in \cite{GeissLeclercSchroeer11}.
\end{Rem}

\begin{Prop}[{\cite[Theorem
  6.20]{Kimura10}}]\label{prop:initialMinors}
$\{\Gup(\mbf{c}[k_{\min},k]; \overrightarrow{w})\}_{1\leq k\leq \ell}$ forms a strongly compatible family.
\end{Prop}

\subsubsection*{$T$-system}

  \begin{Prop} For any $i\in I$ and any $w_1,w_2\in W$ such that
    $l(w_js_i)=l(w_j)+1$ for $j=1$ and $2$, we have
    \begin{align}\label{eq:negativeTsystem}
      v^{A}D^+_{w_1s_i \varpi_i,w_2s_i
        \varpi_i}D^+_{w_1\varpi_i,w_2\varpi_i}=v^{-1+B}D^+_{w_1s_i\varpi_i,w_2\varpi_i}D^+_{w_1\varpi_i,w_2s_i\varpi_i}+D^-_{w_1\gamma_i,w_2\gamma_i}.
\end{align}
  \end{Prop}
  \begin{proof}
    The statement follows from Proposition \ref{prop:DeltaTsys} and
    Lemma \ref{lem:deltaMinor}.
  \end{proof}

\begin{Prop}[{\cite[Proposition 5.5]{GeissLeclercSchroeer11}}]\label{prop:TsysMinor}
For any $i\in I$ and any $1\leq a\leq b\leq l$ with $i_a=i_b=i$, we have
  \begin{align}\label{eq:positivTsystem}
    \begin{split}
      & v^A B_+^{\mathrm{up}}(\mbf{c}[a, b], \overrightarrow{w})
      B_+^{\mathrm{up}}(\mbf{c}[a^-,b^-],
      \overrightarrow{w})\\
      &=v^{-1+B}B_+^{\mathrm{up}}(\mbf{c}[a,b^-],
      \overrightarrow{w})B_+^{\mathrm{up}}(\mbf{c}[a^-,b],
      \overrightarrow{w})+B_+^{\mathrm{up}}(-\sum_{j\neq
        i}c_{ij}\mbf{c}[a^-(j),b^-(j)], \overrightarrow{w}).
    \end{split}
  \end{align}
\end{Prop}
\begin{NB}
\begin{proof}
In \eqref{eq:negativeTsystem}, take $u=s_{i_{1}}s_{i_{2}}\dots s_{i_{b-1}}, v=s_{i_{1}}s_{i_{2}}\dots
 s_{i_{a-1}}$. Then we apply the antiautomorphism $\Omega$ to this
 equation, and the statement follows.
\end{proof}
\end{NB}

\begin{Rem}\label{rem:compatible}
  It follows from \cite[Theorem 4.24]{Kimura10} that the last term in
  \eqref{eq:positivTsystem} can be rewritten as 
  \begin{align*}
    B_+^{\mathrm{up}}(-\sum_{j\neq i}c_{ij}\mbf{c}[a^-(j),b^-(j)],
    \overrightarrow{w})= q^C\prod_{j\neq i}B_+^{\mathrm{up}}(\mbf{c}[a^-(j),b^-(j)],\overrightarrow{w})^{-c_{ij}}
  \end{align*}
for some specific power $C$. Let us denote
$B_+^{\mathrm{up}}(\mbf{c}[a,b],\overrightarrow{w})$ by $D^+[a,b]$. Then we obtain
\begin{align*}
  v^AD^+[a,b]D^+[a^-,b^-]=v^{-1+B}D^+[a,b^-]D^+[a^-,b]+q^C\prod_{j\neq
  i}D^+[a^-(j),b^-(j)]^{-c_{ij}}.
\end{align*}
\end{Rem}

\begin{Eg}[{\cite[Example 6.2]{HernandezLeclerc11}}]
  Let the root datum be given as in Example \ref{eg:A_3}. Then we have
  \begin{align*}
    v^{-1}D[4,4]D[1,1]&=v^{-1}D[1,4]+D[2,2]D[3,3],\\
D[5,5]D[2,2]&=v^{-1}D[2,5]+D[4,4],\\
D[6,6]D[3,3]&=v^{-1}D[3,6]+D[4,4].
  \end{align*}
\end{Eg}

\section{Twisted $t$-analogue of $q$-characters}
\label{sec:twistedQTCharacters}

In this section, we introduce new quantizations and define a twisted $t$-analogue of $q$-characters
$\HChar$, which are slightly different from those used in Section \ref{sec:psedoModules}.

\begin{Rem}
  In fact, our $\HChar$ is a $t$-analogue of the $q$-character defined for the finite dimensional representations
  of the quantum loop algebra $U_q(Lg)$, where $g$ is any skew-symmetric
  Kac-Moody Lie algebra. It should be compared with the character
  introduced by \cite{Hernandez02}, which is defined for the case
  where $g$ is a simple Lie
  algebra.
\end{Rem}

\subsection{A new bilinear form}

Let $\tau$ denote the Auslander-Reiten translation of the derived category
$D^b(\C
Q-mod)$. It induces an automorphism of the Grothendieck group
$K_0(D^b(\C Q-\mod))$ which is denoted by $c$.

For any object $M$ of $D^b(\C Q-\mod)$, let $[M]$ denote its
class in the Grothendieck group. We identify the root lattice $Q=\oplus_{i\in I} \Z \alpha_i$ with the Grothendieck group $K_0(D^b(\C
Q-\mod))$ by sending the simple root $\alpha_i$ to $[S_i]$ the class of the
$i$-th simple module $S_i$, for all $i\in I$. Notice that $\{\alpha_i,i\in I\}$ is a $\Z$-basis of the
Grothendieck group. For any $i\in I$, denote the injective $\C Q$-module with the socle
$S_i$ by $\Iop_i$ and the projective $\C Q$-module with
the top $S_i$ by $\Pop_i$, .

Let $\beta$ be the linear map from $\N^{I\times \Z}$ to $K_0(D^b(\C
Q-\mod))$ \st
\begin{align*}
  \beta(w)=\sum_{(i,a)}w_i(a)*[\tau^{a}\Iop_i[-1]].
\end{align*}
In particular we have $\beta(e_i(0))=[\Iop_i[-1]]=[\tau\Pop_i]$,
$\beta(e_i(-1))=[\Pop_i]$.



\begin{Lem}
  For any $v\in \N^{I\times (\Z+\Hf)}$, we have $\beta(C_q v)=0$.
\end{Lem}






Let $\langle\ ,\ \rangle$ denote the Euler form on $K_0(D^b(\C
Q-\mod))$.
\begin{Lem}
For any pairs $(k,a)$, $(k',b)$ in $I\times \Z$,  $e_{k'}(b)\cdot \invCq(e_k(a))$ equals $\langle[\tau^{b-\Hf}\Pop_{k'}], [\tau^{a}\Pop_k]\rangle $ if $b\geq a+\Hf$ and
vanishes if not.
\end{Lem}
\begin{proof}
Take the vector $v=(v_{k'}(b))$, \st $v_{k'}(b)=
\langle[\tau^{b-a-\Hf}\Pop_{k'}], [\Pop_{k}]\rangle$ if $b\geq a+\Hf$
and $v_{k'}(b)=0$ if $b< a+\Hf$. We want to show $C_q v=e_k(a)$. Fix
any $j\in I$.

First, $e_j(d)\cdot C_q(v)$ is zero for any $d< a$.

Second, we have
\begin{align*}
  e_j (a)\cdot C_q(v)&=v_j(a+\Hf)-\sum_{i:i<j}b_{ij}v_j(a+\Hf)\\
&=p_{kj}-\sum_{i:i<j}b_{ij}p_{ki}\\
&=\delta_{kj}.
\end{align*}

Finally, for any $b>a+\Hf$, we use exact triangles in $D^b(\C Q-\mod)$
to obtain the following result.
\begin{align*}
  e_j (b-\Hf)\cdot C_q(v)&=v_j(b-1)-\sum_{l:l>j}b_{jl}v_l(b-1)
  +v_j(b)-\sum_{i:i<j}b_{ij}v_j(b)\\
&=\langle[\tau^{(b-1)-a-\Hf}\Pop_j], \Pop_{k}]\rangle -\sum_{l:l>j}b_{jl}\langle[\tau^{(b-1)-a-\Hf}\Pop_l], \Pop_{k}]\rangle
\\&\qquad +\langle[\tau^{b-a-\Hf}\Pop_j], \Pop_k]\rangle-\sum_{i:i<j}b_{ij}\langle[\tau^{b-a-\Hf}\Pop_i], \Pop_k]\rangle\\
&=\langle \tau^{b-a-\Hf} ([\tau^{-1}
\Pop_{j}]-\sum_{l:l>j}b_{jl}[\tau^{-1} \Pop_l]+\Pop_j-\sum_{i:i<j}b_{ij}[\Pop_i]),[\Pop_k]\rangle \\
&=0.
\end{align*}
\end{proof}

Recall that we have \begin{align*}
  \eMatrix(w^1,w^2)=-w^1[ \Hf]\cdot C_q^{-1} w^2 +w^2[\Hf]\cdot C_q^{-1}
  w^1.
\end{align*}
Let us define a new bilinear form $\cN(\ ,\ )$ on $\N^{I\times Z}$ \st for
any $w^1,$ $w^2$ in $\N^{I\times \Z}$, we have
\begin{align*}
  \cN(w^1,w^2)=w^1[\Hf]\cdot C_q^{-1}w^2-w^1[- \Hf]\cdot C_q^{-1}
w^2 -w^2[ \Hf]\cdot C_q^{-1} w^1 +w^2[- \Hf]\cdot
C_q^{-1} w^1.
\end{align*}

Clearly, we have $\cN(w^1,w^2)=-\cN(w^2,w^1)$. 

Define a symmetric bilinear form $(\ ,\ )$ on $K_0(D^b(\C Q-\mod) )$
\st for any $x,y\in K_0(D^b(\C Q-\mod) )$, $(x,y)=\langle
x,y\rangle+\langle y,x\rangle $. Notice that $\langle x, cy\rangle
=-\langle y,x\rangle$. Further define $\langle\ ,\ \rangle_a$ to be
the anti-symmetrized Euler form \st we have $\langle x,y\rangle_a=\langle x,y\rangle -\langle y,x\rangle$.



\begin{Lem}
  Given any integer $d\geq 1$ and any $i,j\in I$, we have
  \begin{align*}
    \cN(e_i(0),e_j(d))= (\beta(e_i(0)),\beta(e_j(d))),\\
    \eMatrix(e_i(0),e_j(d))=\langle c^{-1}\beta(e_j(d)), \beta(e_i(0))
    \rangle.
  \end{align*}
\end{Lem}
\begin{proof} We have
  \begin{align*}
    \cN(e_i(0),e_j(d))=&e_i(0)[\Hf]\cdot
    C_q^{-1}e_j(d)-e_i(0)[-\Hf]\cdot C_q^{-1}e_j(d)\\&\qquad -e_j(d)[\Hf]\cdot
    C_q^{-1}e_i(0)+e_j(d)[-\Hf]\cdot C_q^{-1}e_i(0)\\
=&0+0-\langle [\tau^{d-1} \Pop_j],[\Pop_i]\rangle +\langle [\tau^d
\Pop_j],[\Pop_i]\rangle\\
=&0+0-\langle c^{d-1} [\Pop_j],[\Pop_i]\rangle +\langle c^d[
\Pop_j],[\Pop_i]\rangle\\
=&(c^dc[\Pop_j],c[\Pop_i])\\
=&(\beta(e_j(d)),\beta(e_i(0))).
\end{align*}
Similarly, we have
\begin{align*}
  \eMatrix(e_i(0),e_j(d))=&-2 e_i(0)[\Hf]\cdot C_q^{-1}e_j(d)+2
  e_j(d)[\Hf]\cdot C_q^{-1}e_i(0)\\
=& 2 \langle [\tau^{d-1}\Pop_j],[\Pop_i]\rangle\\
=&2 \langle c^{-1}\beta(e_j(d)),e_i(0)\rangle
\end{align*}

\end{proof}



Similarly, we have the following result.

\begin{Lem}
The following equations hold:
  \begin{align*}
    \cN(e_i(0),e_j(0))&= \langle \Pop_j,\Pop_i\rangle-\langle \Pop_i,\Pop_j\rangle,\\
    \eMatrix(e_i(0),e_j(0))&=0,\\
    \cN(e_i(0),e_j(-1))&= \langle
    [\Pop_i],[\Pop_j]-c^{-1}[\Pop_j]\rangle .
  \end{align*}
\end{Lem}


\begin{Lem}\label{lem:quadraticFormDifference}
  The difference between the quadratic forms $\cN$ and $-2\eMatrix$ is the
  anti-symmetrized Euler form, \ie $\cN+2\eMatrix=\langle\ ,\ \rangle_a$.
\end{Lem}
\begin{proof}
  It suffices to check the statements for the unit vectors $e_i(a)$,
  $(i,a)\in I\times \Z$.
\end{proof}

\subsection{A new quantization of the cluster algebras}

Define the $2n\times 2n$ matrix $L$ whose entries are given by, for
any $i,j\in I$,
\begin{align*}
  L_{ij}&=\cN(e_i(0),e_j(0))=\langle \Pop_j,\Pop_i\rangle-\langle \Pop_i,\Pop_j\rangle,\\
L_{i,j+n}&=\cN(e_i(0),e_j(0)+e_j(-1))=\langle \Pop_j,\Pop_i\rangle-\langle\Pop_i,\tau^{-1}\Pop_j\rangle,\\
L_{i+n,j}&=\cN(e_i(0)+e_i(-1),e_j(0)),\\
L_{i+n,j+n}&=\cN(e_i(0)+e_i(-1),e_j(0)+e_j(-1)).\\
\end{align*}
It is easy to check that $L$ is skew-symmetric and $L_{i+n,j+n}$
equals $L_{i,j+n}-L_{j,i+n}$.

Let $\tB$ be given as in Section \ref{sec:psedoModules}.
\begin{Prop}
We have $L(-\tB)=\begin{bmatrix}2 \id_n\\0 \end{bmatrix}$.
\end{Prop}
\begin{proof}
For $1\leq l,k\leq n$, we have
  \begin{align*}
(L \tB)_{lk}   =&
\sum_{i:i<k}L_{li}b_{ik}+\sum_{j:j>k}L_{lj}(-b_{kj})-L_{l,k+n}+\sum_{j:j>k}b_{kj}L_{l,j+n}\\
=&(\langle \sum_i b_{ik}\Pop_i,\Pop_l\rangle-\langle \Pop_l,\sum_i
b_{ik}\Pop_i\rangle) + (\langle -\sum_j b_{kj}\Pop_j,\Pop_l\rangle-\langle
\Pop_l,-\sum_jb_{kj}\Pop_j\rangle  )\\& +(\langle -\Pop_k,\Pop_l\rangle,
\langle \Pop_l,\tau^{-1}\Pop_k\rangle)+(\langle \sum_j
b_{kj}\Pop_j,\Pop_l\rangle+\langle \Pop_l,-\sum_j
b_{kj}\tau^{-1}\Pop_j\rangle)\\
=&\langle \sum_i b_{ik}\Pop_i-\Pop_k,\Pop_l\rangle +\langle \Pop_l,-\sum_i
b_{ik}\Pop_i+\sum_j
b_{kj}\Pop_j+\tau^{-1}\Pop_k-\sum_j b_{kj}\tau^{-1}\Pop_j\rangle\\
=&\langle \sum_i b_{ik}\Pop_i-\Pop_k,\Pop_l\rangle +\langle \Pop_l,-\Pop_k+\sum_j
b_{kj}\Pop_j\rangle,
  \end{align*}
where we use exact triangles in the last equality. If $l=k$, the entry
equals $-2$. Else, it becomes
\begin{align*}
&\langle \sum_{i<k} b_{ik}\Pop_i,\Pop_l\rangle +\langle \Pop_l,\sum_{j>k}
b_{kj}\Pop_j\rangle -(\Pop_k,\Pop_l\rangle +\Pop_l,\Pop_k\rangle)\\
=&\sum_{i<k}p_{li}b_{ik}+\sum_{j>k}b_{kj}p_{jl}-(p_{lk}+p_{kl})
=0.
  \end{align*}
Similarly, we can compute $(L\tB)_{l+n,k}$:
\begin{align*}
 (L\tB)_{l+n,k}=&\sum_{i:i<k}L_{l+n,i}b_{ik}+\sum_{j:j>k}L_{l+n,j}(-b_{kj})-L_{l+n,k+n}+\sum_{j:j>k}b_{kj}L_{l+n,j+n}\\
=&\sum_{i:i<k}L_{l+n,i}b_{ik}+\sum_{j:j>k}L_{l+n,j}(-b_{kj})-L_{l+n,k+n}\\
& +\sum_{j:j>k}b_{kj}(L_{l,j+n}+L_{l+n,j})\\
=&\sum_{i:i<k}L_{l+n,i}b_{ik}-L_{l+n,k+n}+\sum_{j:j>k}b_{kj}L_{l,j+n}\\
=&-\langle \Pop_l,\sum_i b_{ik}\Pop_i\rangle+\langle \sum_i
b_{ik}\Pop_i,\tau^{-1}\Pop_l\rangle\\
&-(\langle\Pop_k,\Pop_l\rangle-\langle
\Pop_l,\tau^{-1}\Pop_k\rangle-\langle \Pop_l,\Pop_k\rangle+\langle
\Pop_k,\tau^{-1}\Pop_l\rangle)\\
&+\langle \sum_j b_{kj}\Pop_j,\Pop_l\rangle-\langle \Pop_l,\sum_j
b_{kj}\tau^{-1}\Pop_j\rangle\\
=&\langle \Pop_l,-\sum_{i<k}b_{ik}
\Pop_i+\tau^{-1}\Pop_k+\Pop_k-\sum_{j>k}b_{kj}\tau^{-1}\Pop_j\rangle\\
&+\langle \sum_i b_{ik}\Pop_i-\tau^{-1}\Pop_k-\Pop_k+\sum_j
b_{kj}\tau^{-1}\Pop_j,\tau^{-1}\Pop_l\rangle
=0.
\end{align*}
\end{proof}

\begin{Eg}\label{eg:quiverA_3}
Let the quiver $Q$ and the ice quiver $\tQ^z_1$ be given by Figure
\ref{fig:A3Quiver} and \ref{fig:levelOneA3Quiver} respectively. The associated $B$-matrix is
\begin{align*}
  \tB=
  \begin{pmatrix}
    0&0&1\\
    0&0&1\\
    -1&-1&0\\
    -1&0&0\\
    0&-1&0\\
    1&1&-1
  \end{pmatrix}.
\end{align*}
We have the matrices
\begin{align*}
\cN&=
  \begin{pmatrix}
    0&0&1&1&-1&-1\\
    0&0&1&-1&1&-1\\
    -1&-1&0&1&1&0\\
    -1&1&-1&0&0&1\\
    1&-1&-1&0&0&1\\
    1&1&0&-1&-1&0
  \end{pmatrix},\\
L&=
  \begin{pmatrix}
    0&0&1&1&-1&0\\
    0&0&1&-1&1&0\\
    -1&-1&0&0&0&0\\
    -1&1&0&0&0&0\\
    1&-1&0&0&0&0\\
    0&0&0&0&0&0
  \end{pmatrix}.
\end{align*}
It is easy to check that $L\cdot(-\tB)=
\begin{pmatrix}
  2\id_3\\
0
\end{pmatrix}.
$
\end{Eg}

\begin{figure}[htb!]
 \centering
\beginpgfgraphicnamed{fig:A3Quiver}
  \begin{tikzpicture}
    \node [shape=circle, draw] (v1) at (1,-3) {1}; \node
    [shape=circle, draw] (v2) at (3,-2) {2}; \node [shape=circle,
    draw] (v3) at (3,0) {3};

\draw[-triangle 60] (v2) edge
    (v3); \draw[-triangle 60] (v1) edge (v3);
  \end{tikzpicture}
\endpgfgraphicnamed
\caption{A quiver $Q$ of type $A_3$}
\label{fig:A3Quiver}
\end{figure}
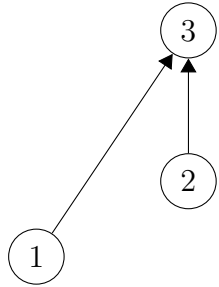

\begin{figure}[htb!]
 \centering
\beginpgfgraphicnamed{fig:levelOneA3Quiver}
  \begin{tikzpicture}
    \node [shape=circle, draw] (v1) at (1,-3) {1}; 
    \node  [shape=circle, draw] (v2) at (3,-2) {2}; 
    \node [shape=circle,  draw] (v3) at (3,0) {3};

\node [shape=circle, draw] (v4) at (-4,-3) {4}; 
    \node  [shape=circle, draw] (v5) at (-2,-2) {5}; 
    \node [shape=circle,  draw] (v6) at (-2,0) {6};

    \draw[-triangle 60] (v2) edge (v3); 
    \draw[-triangle 60] (v1) edge (v3);

    \draw[-triangle 60] (v6) edge (v1); 
    \draw[-triangle 60] (v6) edge (v2); 
    
    \draw[-triangle 60] (v1) edge (v4); 
    \draw[-triangle 60] (v2) edge (v5); 
    \draw[-triangle 60] (v3) edge (v6); 
       
  \end{tikzpicture}.
\endpgfgraphicnamed
\caption{A level $1$ ice quiver with $z$-pattern of $A3$-type principal part}
\label{fig:levelOneA3Quiver}
\end{figure}
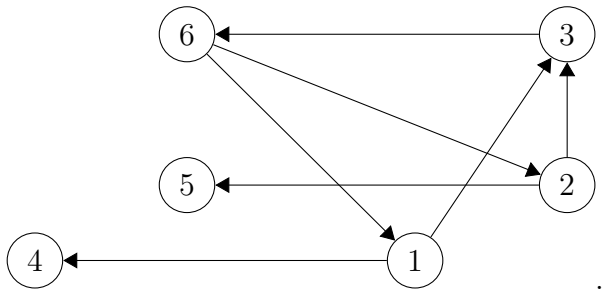

Let us define the involution $\rev:I\ra
I$ \st $\rev(i)=n+1-i$ for any $1\leq i\leq n$. For $1\leq i\leq
n$, define 
\begin{align}\label{eq:beta}
  \beta_i:=c^{-1}\beta(e_{\rev(i)}(0)),\  \beta_{n+i}:=c^{-1}\beta(e_{\rev(i)}(-1)).
\end{align}
These notations are justified by the following example.

\begin{Eg}
Assume we are given the root datum and choose the Weyl group element $w=c^2$ as
in Example \ref{eg:A_3}, where we have obtained the positive roots
$\beta_j$, $1\leq j\leq 6$. Let
the quiver $Q$ be given as in Example \ref{eg:quiverA_3}. Then we
have, for $1\leq i\leq 3$, $c^{-1}\beta(e_i(0))=[\Pop_i]=\beta_{\rev{(i)}}$,
$c^{-1}\beta(e_i(-1))=[\tau^{-1}\Pop_i]=\beta_{\rev{(i)}+3}$.
\end{Eg}

Let $\tL$ be the $2n\times 2n$ skew-symmetric matrix defined in
\cite[(10.2)]{GeissLeclercSchroeer11}. By \cite[Proposition
9.5]{GeissLeclercSchroeer11}, it is uniquely determined by the
following conditions: for any $1\leq j<i\leq n$,
\begin{align*}
  \tL_{ij}=(\beta_i,\beta_j),\\
\tL_{i+n,j}=(\beta_i,\beta_j)+(\beta_{i+n},\beta_j),\\
\tL_{i+n,j+n}=(\beta_i,\beta_j)+(\beta_{i+n},\beta_{j+n})+(\beta_{i+n},\beta_j)-(\beta_{j+n},\beta_i).
\end{align*}


\begin{Lem}
For all $i,j\in I$, we have $L_{ij}=\tL_{\rev(i),\rev(j)}$,
$L_{i+n,j}=\tL_{\rev(i)+n,\rev(j)}$,
$L_{i,j+n}=\tL_{\rev(i),\rev(j)+n}$,
$L_{i+n,j+n}=\tL_{\rev(i)+n,\rev(j)+n}$. Namely, we can identify $L$
with $\tL$ by permuting the indices.
\end{Lem}
\begin{proof}
  It suffices to consider the case $i>j$. Recall that we have
  \begin{align*}
    \cN(e_i(0),e_j(0))=(\beta(e_i(0)),\beta(e_j(0)),\\
    \cN(e_i(-1),e_j(0))=(\beta(e_i(-1),\beta(e_j(0)),\\
    \cN(e_i(0),e_j(-1))=-\cN(e_j(-1),e_i(0))=-(\beta(e_j(-1),\beta(e_i(0)).
  \end{align*}
Straightforward computation verifies the statement.
\end{proof}








\subsection{New $t$-deformations of Grothendieck rings and characters}

We modify the multiplication $\otimes$ of $\quotKGp\otimes_{\tBase}
\Z[t^{\pm\Hf}]$ \st \eqref{eq:otimes} is replaced by
\begin{align*}
  L(w^1)\otimes L(w^2)=(t^\Hf)^{\langle \beta(w^2),\beta(w^1)\rangle_a}\sum_{w^3} b^{w^3}_{w^1,w^2}(t^{-1})L(w^3)
\end{align*}
and denote this modified version of
$\quotKGp\otimes_{\tBase}\Z[t^{\pm\Hf}]$ by $\HGp$. Similarly, we
modify the twisted multiplication $*$ of
$\redTargSpace\otimes_{\tBase}\Z[t^{\pm\Hf}]$ \st \eqref{eq:twistedMultiplication} is replaced by
\begin{align}\label{eq:newTwistProd}
  m^1*m^2=t^{\Hf\cN(m^1,m^2)}m^1m^2.
\end{align}
and denote this modified version of $\redTargSpace\otimes_{\tBase}\Z[t^{\pm\Hf}]$ by
$\HRing$.

In analogy to $\tChar(\ )$, we define the $\tBase$-linear map $\HChar(\ )$ from
  $\HGp$ to $\HRing$ such that for all $w\in \N^{I\times \Z}$, we have
  \begin{align}
    \label{eq:HQtChar}
    \HChar(\can(w))=\sum_v \langle L_w(0) ,\pi_w(v)\rangle
    t^{\dim\grProjQuot(v,w)}Y^{w-C_qv}.
  \end{align}
The map $\HChar(\ )$ is called the \emph{twisted $t$-analogue of
$q$-characters}. Its truncation $\HChar\trunc(\ )$ is defined similarly.
\begin{Thm}
  $\HChar(\ )$ is an injective algebra homomorphism from $\HGp$ to $\HRing$.
\end{Thm}
\begin{proof}
By Lemma \ref{lem:quadraticFormDifference}, equation
\eqref{eq:newTwistProd}, which defines the twisted product of
$\HRing$, can be written as
  \begin{align*}
    m^1*m^2=&t^{\Hf\cN(m^1,m^2)}m^1m^2\\
=&(t^\Hf)^{\langle \beta(w^2),\beta(w^1)\rangle_a}
    (t^{\Hf})^{-2\eMatrix(w^1,w^2)}m^1m^2.
  \end{align*}
Then the statement can be easily deduced from Theorem
\ref{thm:injectiveHom} (\cf the correction technique in \cite{Qin11}).
\end{proof}

\section{Dual canonical basis}
\label{sec:dualCanonicalBasis}


Let $(\tB,L)$ be given as in Section
\ref{sec:twistedQTCharacters}. Consider the quantum cluster algebra
$\qGLSClAlg$ whose initial compatible pair is chosen to be $(\tB,
L)$. It is a subalgebra of the quantum torus $\torus(L)$.

For any linear combination $\sum_i d_i \alpha_i$, $d_i\in \Z$, we
define its degree $\deg(\sum_i d_i\alpha_i )$ to be $\sum_i
d_i$. Define the quadratic function $N(\ ):\N^{I\times \Z}\ra \Z$ \st
for any $w\in\N^{I\times \Z}$, we have
\begin{align*}
  N(w)=\langle\beta(w),\beta(w)\rangle+\deg c^{-1}\beta(w).
\end{align*}

By abuse of notation, let $\cor$ denote the $\Z$-linear map from $\HRing$ to
$\torus(L)$ \st we have
\begin{align}
\cor(t^{\frac{\lambda}{2}} Y^w)=q^{\frac{\lambda}{2}} x^{\ind(w)}
\end{align}
for any $w$ and any integer $\lambda$. The arguments of Section
\ref{sec:correspondence} (or \cite{Qin11}) imply that $\cor$ is an algebra homomorphism.







Denote the image $\HChar\trunc(\HGp)$ by $\rscQClAlg$. For any $w$, denote $\cor\HChar\trunc(\pbw(w))$ by $\pbwCl(w)$,
$\cor\HChar\trunc(\can(w))$ by $\canCl(w)$, and $\cor\HChar\trunc(\gen(w))$ by
$\genCl(w)$. Then $\pbwClBasis$, $\canClBasis$, $\genClBasis$ are three
homogeneous bases of the $K_0(D^b(\C Q-\mod)$-graded algebra
$\rscQClAlg$. 

\begin{Prop}\label{prop:contain}
  $\canClBasis$ and $\genClBasis$ contain all the quantum cluster
  monomials.
\end{Prop}
\begin{proof}
  The statement follows from Theorem \ref{thm:iso} and the existence
  of quantum $F$-polynomials (\cf \cite{Tran09} or the correction
  technique in \cite{Qin11}).
\end{proof}
Therefore, $\rscQClAlg$ is the subalgebra of $\qGLSClAlg$ generated by all the quantum cluster variables and the
frozen variables $x_{n+1},\cdots,x_{2n}$. So we also call $\rscQClAlg$
a quantum cluster algebra. 

Notice that the
structure constants of either $\pbwCl$ or $\canCl$ take values in
$\Z[q^\pm]$, since the map $\cor$ sends $t$ to $q$. Also, the
non-diagonal entries of the transition matrix between them takes values in $q\Z[q]$.



In order to be in accordance with the usual convention in constructing PBW basis (and transition
matrix in $q^{-1}\Z[q^{-1}]$), we modify the bases $\pbwClBasis$ and
$\canClBasis$ by introducing $\rscPbwCl(w)=q^{-\Hf\langle
  \beta(w),\beta(w)\rangle}\pbwCl(w)$, and $\rscCanCl(w)=q^{-\Hf\langle
  \beta(w),\beta(w)\rangle}\canCl(w)$. The elements
$\rscPbwCl(e_i(a))=\rscCanCl(e_i(a))$, $i\in I$, $a\in \{0,-1\}$, are called the dual
PBW generators of $\rscQClAlg$.



We follow the convention of Section \ref{sec:unipotentSubgroup} (and
thus of \cite{GeissLeclercSchroeer11}). Choose the Coxeter element $c$ to
be $s_{\rev(1)}s_{\rev(2)}\ldots s_{\rev(n)}$.

Denote $\intA=\Q[v^\pm]$. The image $\psi_+ \Uv^+(c^2)_{\mca{A}}^{\mathrm{up}}$
  is called the integral form of $A_v(n(c^2))$, which we denote by $A_\intA(n(c^2))$. This is an
  $\intA$-algebra. Denote the $\intA$-algebra $\rscQClAlg
  \otimes_{\Z[v^\pm]}{\intA}$ by $\rscQClAlg_{\intA}$.

\begin{Prop}\label{prop:isoAlg}
  There is a $K_0(D^b(\C Q-\mod))$-graded algebra isomorphism
  $\tilde{\kappa}$ from $\rscQClAlg_{\intA}$ to $A_{\intA}(n(c^2))$,
  which sends $\rscPbwCl(w)$ to the dual PBW basis element $E^{\mathrm{up}}(c^{-1}\beta(w))$.
\end{Prop}
\begin{proof}
As in \cite[Proposition 12.1]{GeissLeclercSchroeer11}, we want to
compare the $T$-systems in both algebras. In $A_{\intA}(n(c^2))$, we
have, for any $k\in I$,
\begin{align*}
v^A\BPos(c[\rev(k)+n,& \rev(k)+n],\ow)
\BPos(c[\rev(k),\rev(k)],\ow)=v^{-1+B}\BPos(c[\rev(k),\rev(k)+n],\ow)\\
& +\BPos(\sum_{i<k}b_{ik}c[\rev(i),\rev(i)]+\sum_{j>k}b_{kj}c[\rev(j)+n,\rev(j)+n],\ow),
\end{align*}
where $A=(\mu(\rev(k)+n,k),\varpi_k-\mu(\rev(k),k))$,
$B=(\mu(\rev(k),k),\varpi_k-\mu(\rev(k)+n,k))$, and we denote
$\mu(d,i)=s_{i_1}\cdots s_{i_d}\varpi_i$ for any $d\in\N$, \cf
\cite[Proposition 5.5]{GeissLeclercSchroeer11}.

Notice that the PBW generators $\BPos(c[a,b],\ow)$ of $A_{\intA}(n(c^2))$, $1\leq a\leq b\leq
  2n$, $i_a=i_b$, satisfy the $T$-systems in Proposition \ref{prop:TsysMinor}. Further using
Remark \ref{rem:compatible}, we deduce that $A_\intA(n(c^2))$ is
contained in the algebra $\cT'$ which is generated by
$\BPos(c[\rev(i),\rev(i)],\ow)$, $\BPos(c[\rev(i),\rev(i)+n],\ow)$, $i\in I$, and their inverses.

It follows from Proposition \ref{prop:initialMinors} and the
definition of $L$ that there is an algebra isomorphism $\tilde{\kappa}$
from the quantum torus $\cT_{\Z[v^\pm]}{\intA}$ to $\cT'$ such that we have
\begin{align*}
\tilde{\kappa}(\BPos (\mbf{c}[\rev(i),\rev(i)],\ow))=\rscCanCl(e_i(0)),\\
\tilde{\kappa}(\BPos (\mbf{c}[\rev(i),\rev(i)+n],\ow))=\rscCanCl(e_i(-1)+e_i(0)).
\end{align*}
We refer the reader to \cite[Proposition 11.5]{GeissLeclercSchroeer11}
for a detailed examination of $\tilde{\kappa}$ for general $w$.

Notice that for any $i,j \in I$, $1 \leq r\leq 2n$, we have $(\varpi_i,\alpha_j)=\delta_{ij}$, $(\varpi_i,\beta_{\rev(i)})=1$,
$(\beta_r,\beta_r)=2$. 

From now on, fix any $k\in I$. We compute $(\varpi_k,\beta_{\rev(k)+n})$ as
\begin{align*}
  (\varpi_k,\beta_{\rev(k)+n})&=(s_ks_{k+1}\cdots
s_n\varpi_k,s_{k-1}s_{k-2}\cdots s_1s_ns_{n-1}\cdots
s_{k+1}\alpha_k)\\
&=(\varpi_k-\alpha_k,s_{k-1}s_{k-2}\cdots
s_1s_ns_{n-1}\cdots s_{k+1}\alpha_k)\\
&=(\varpi_k,s_{k-1}s_{k-2}\cdots
s_1s_ns_{n-1}\cdots s_{k+1}\alpha_k)\\& \qquad\qquad-(\alpha_k,s_{k-1}s_{k-2}\cdots
s_1s_ns_{n-1}\cdots s_{k+1}\alpha_k)\\
&=(\varpi_k,\alpha_k)-(s_ns_{n-1}\cdots s_k\alpha_k,\beta_{\rev(k)+n})\\
&=1+(\beta_{\rev(k)},\beta_{\rev(k)+n}).
\end{align*}

So we have
\begin{align*}
  A&= (\varpi_k-\beta_{\rev(k)}-\beta_{\rev(k)+n},\varpi_k-(\varpi_k-\beta_{\rev(k)}))\\
&=(\varpi_k-\beta_{\rev(k)}-\beta_{\rev(k)+n},\beta_{\rev(k)})\\
&=-1-(\beta_{\rev(k)},\beta_{\rev(k)+n}),
\end{align*}
\begin{align*}
  B&=(\varpi_k-\beta_{\rev(k)},\beta_{\rev(k)}+\beta_{\rev(k)+n})\\
&=-1-(\beta_{\rev(k)},\beta_{\rev(k)+n})+(\varpi_k,\beta_{\rev(k)+n})\\
&=0.
\end{align*}

Consider the quantum cluster algebra $\rscQClAlg_\intA$. For any $i\in I$,
let $x_i^*$ denote the quantum cluster variable $\canCl(e_i(-1))$. We
have the following $T$-system:
\begin{align*}
x_k^*x_k=&q^{\Hf L(e_{n+k},e_k)}x_{n+k}\\
&+q^{\Hf L(\sum_{i<k}b_{ik}e_i+\sum_{j>k}b_{kj}e_{j+n},e_k)}\prod_{i<k}x_i^{b_{ik}}\cdot (x_{k+1}^*)^{b_{k,k+1}}\cdots (x_n^*)^{b_{kn}}.
\end{align*}
We have
$ L(e_{n+k},e_k)+2= L(\sum_{i<k}b_{ik}e_i+\sum_{j>k}b_{kj}e_{j+n},e_k)$.
The above $T$-system becomes
\begin{align*}
q^{-\Hf L(e_{n+k},e_k)-1} x_k^*x_k=q^{-1} x_{n+k}+\prod_{i<k}x_i^{b_{ik}}\cdot (x_{k+1}^*)^{b_{k,k+1}}\cdots (x_n^*)^{b_{kn}}.
\end{align*}
By definition, we have 
\begin{align*}
  x_k=q^{\Hf}\rscCanCl(e_k(0)),\ x_k^*=q^{\Hf}\rscCanCl(e_k(-1)),\\
  x_{n+k}=q^{\Hf\langle \beta_{\rev(k)}+\beta_{\rev(k)+n},
    \beta_{\rev(k)}+\beta_{\rev(k)+n}\rangle}\rscCanCl(e_k(-1)+e_k(0)).
\end{align*}
Also, we have $L(e_{n+k},e_k)=-(\beta(e_k(-1)),\beta(e_k(0)))$ and $\sum_{i<k}b_{ik}\beta(e_i(0))+\sum_{j>k}b_{kj}\beta(e_j(-1))=\beta(e_k(-1))+\beta(e_k(0))$.

Therefore, the $T$-system can be written as
\begin{align*}
&q^{-1-(\beta(e_k(-1)),\beta(e_k(0)))}
\rscCanCl(e_k(-1))*\rscCanCl(e_k(0))\\&\qquad =q^{-1}\rscCanCl(e_k(-1)+e_k(0))+\rscCanCl(\sum_{i<k}b_{ik}e_i(0)+\sum_{j>k}b_{kj}e_j(-1)).
\end{align*}

Therefore, $\tilde{\kappa}$ identifies the PBW generators. It follows
that it gives an isomorphism from $\rscQClAlg_\intA$ to $A_{\intA}(n(c^2))$.
\end{proof}



Let $A_{\Z[v^\pm]}(n(c^2))$ denote the free $\Z[v^\pm]$-module generated by the dual PBW basis
elements of $A_v(n(c^2))$.
\begin{Thm}\label{thm:2nd}
The map $\tilde{\kappa}$ is an algebra isomorphism from the quantum cluster
algebra $\rscQClAlg$ to
$A_{\Z[v^\pm]}(n(c^2))$. It sends $\rscCanBasis$ to the dual canonical basis
of $A_{\Z[v^\pm]}(n(c^2))$. In particular, every quantum cluster monomials up
to a $v$-power is sent into the dual canonical basis.
\end{Thm}
\begin{proof}
The first statement follows from Proposition
\ref{prop:isoAlg}.

  The bar-invariant basis $\rscCanBasis$ is uniquely determined by the
  dual PBW basis $\{\rscPbwCl(w)\}$ and the
  upper unitriangular property. Similarly, the $\sigma_+$-invariant dual
  canonical basis of $A_\intA(n(c^2))$ is uniquely determined by the
  dual PBW basis of $A_\intA(n(c^2))$ and the
  upper unitriangular property. Because the isomorphism
  $\tilde{\kappa}$ commutes with these two involutions and identifies
  the two dual PBW bases, it identifies the dual canonical bases as
  well. 

The last statement follows from Proposition \ref{prop:contain}.
\end{proof}




\bibliographystyle{amsalpha}
\bibliography{referenceEprint}

\end{document}